\documentclass[11pt]{article}

\usepackage{amssymb}
\usepackage[left=1in,top=1in,right=1in]{geometry}
\usepackage{tikz}
\usepackage{tikz-cd}
\usetikzlibrary{matrix, arrows,  decorations.markings,decorations.pathreplacing,  patterns,  plotmarks}
\usepackage{amsmath}
\usepackage{amssymb}
\usepackage{mathtools}
\usepackage{verbatim}
\usepackage{amsthm}
\usepackage{wrapfig}
\usepackage{sidecap}
\usepackage{enumerate}
\usepackage{thmtools}
\usepackage[disable]{todonotes}
\usepackage{wrapfig}
\usepackage[backend=bibtex, style=alphabetic, maxcitenames=5, maxbibnames=9 ]{biblatex}
\renewbibmacro{in:}{} \usepackage{enumitem}
\usepackage{subcaption}
\usepackage{hyperref}
\usepackage{cleveref}

\newtheorem{theorem}{Theorem}[subsection]
\newtheorem*{theorem*}{Theorem}
\newtheorem{lemma}[theorem]{Lemma}
\newtheorem*{lemma*}{Lemma}
\newtheorem{prop}[theorem]{Proposition}
\newtheorem{corollary}[theorem]{Corollary}
\newtheorem{definition}[theorem]{Definition}
\newtheorem{notation}[theorem]{Notation}
\newtheorem{claim}[theorem]{Claim}
\newtheorem*{claim*}{Claim}
\newtheorem{example}[theorem]{Example}
\newtheorem{conjecture}[theorem]{Conjecture}

\newtheorem{remark}[theorem]{Remark}
\newtheorem{assumption}[theorem]{Assumption}

\newtheorem{maintheorem}{Theorem}
\Crefname{maintheorem}{Theorem}{Theorems}
\newtheorem{maincorollary}[maintheorem]{Corollary}
\Crefname{maincorollary}{Corollary}{Corollaries}

\newtheorem{mainexample}[maintheorem]{Example}
\Crefname{mainexample}{Example}{Examples}

\crefname{claim}{claim}{claims}
\crefname{prop}{proposition}{propositions}
\crefname{prop}{proposition}{propositions}
 
\newcommand{\eps}{\varepsilon}
\newcommand{\RR}{\mathbb R}
\newcommand{\ZZ}{\mathbb Z}
\newcommand{\NN}{\mathbb N}
\newcommand{\CC}{\mathbb C}

\newcommand{\CP}{\mathbb{CP}}
\newcommand{\del}{\nabla}
\newcommand{\into}{\hookrightarrow}
\newcommand{\tensor}{\otimes}

\newcommand{\CF}{CF^\bullet}
\newcommand{\HF}{HF^\bullet}
\newcommand{\CM}{CM^\bullet}

\DeclareMathOperator{\id}{id}

\DeclareMathOperator{\Crit}{Crit}

\newcommand{\syza}{\pi_A}
\newcommand{\syzb}{\pi_B}
\newcommand{\YB}{Y}
\DeclareMathOperator{\val}{val}

\DeclareMathOperator{\Coh}{Coh}
\DeclareMathOperator{\st}{\;:\;}
\DeclareMathOperator{\Fuk}{Fuk}
\DeclareMathOperator{\str}{star}

\DeclareMathOperator{\Supp}{Supp}

\DeclareMathOperator{\grad}{grad}

\DeclareMathOperator{\Spec}{Spec}
\DeclareMathOperator{\tropa}{TropA}
\DeclareMathOperator{\tropb}{TropB}
\DeclareMathOperator{\suppA}{SuppA}
\DeclareMathOperator{\Def}{Def}
\DeclareMathOperator{\res}{res}
\DeclareMathOperator{\Aff}{Aff}

\newcommand{\pants}{{pants}} 
\begin{document}
\listoftodos
\clearpage
\setcounter{page}{1}
\newcommand{\Addresses}{{\bigskip
  \footnotesize

  \noindent J.~Hicks, \textsc{School of Mathematics,  University of Edinburgh}\par\nopagebreak
  \noindent \textit{E-mail address}: \texttt{jeff.hicks@ed.ac.uk}
}}

\title{\normalsize \textbf{Realizability in tropical geometry and unobstructedness of Lagrangian submanifolds}}
\author{\normalsize  Jeff Hicks}
\date{}

\tikzset{every picture/.style=thick}
\maketitle{}

\begin{abstract}
	We say that a tropical subvariety $V\subset \mathbb R^n$ is $B$-realizable if it can be lifted to an analytic subset of $(\Lambda^*)^n$. When $V$ is a smooth curve or hypersurface, there always exists a Lagrangian submanifold lift $L_V\subset (\mathbb C^*)^n$.
We prove that whenever $L_V$ has well-defined Floer cohomology, we can find for each point of $V$ a Lagrangian torus brane whose Lagrangian intersection Floer cohomology with $L_V$ is non-vanishing.
Assuming an appropriate homological mirror symmetry result holds for toric varieties, it follows that whenever $L_V$ is a Lagrangian submanifold that can be made unobstructed by a bounding cochain, the tropical subvariety $V$ is $B$-realizable.

As an application, we show that the Lagrangian lift of a genus zero tropical curve is unobstructed, thereby giving a purely symplectic argument for Nishinou and Siebert's proof that genus-zero tropical curves are $B$-realizable. We also prove that tropical curves inside tropical abelian surfaces are $B$-realizable. \end{abstract}
\setcounter{tocdepth}{1}
\tableofcontents
\section{Introduction}
	Mirror symmetry is a collection of equivalences between symplectic geometry ($A$-model) and algebraic geometry ($B$-model) on a pair of mirror spaces. 
A general proposal for constructing mirror pairs of a symplectic space $X_A$ and algebraic space $X_B$ comes from \citeauthor{strominger1996mirror}, who conjectured that mirror pairs can be presented as dual torus fibrations over an integral affine manifold $Q$. 
One relation between these spaces arises in the form of  \citeauthor{kontsevich1994homological}'s homological mirror symmetry  (HMS) conjecture, which predicts an equivalence between the Fukaya category of $X_A$ and the category of coherent sheaves on a mirror manifold $X_B$.
Roughly, the objects of the Fukaya category of $X_A$ are Lagrangian submanifolds $L\subset X_A$.
A blueprint for mirror symmetry is that Lagrangian submanifolds of $X_A$ relate to sheaves supported on a subvariety of $X_B$ via mutual comparison to tropical subvarieties on the base $Q$.

We consider the relatively well-understood example of $X_A=T^*\RR^n/T^*_\ZZ \RR^n$. $X_B=(\Lambda^*)^n$, and $Q=\RR^n$. On the $A$ side, it will be convenient for us to identify $X_A$ with $(\CC^*)^n$, which has holomorphic coordinates $x_i=e^{q_i+i\theta_i}$ and standard symplectic form $\sum_{i=1}^n dq_i \wedge d\theta_i$. Note that $X_A$ does not naturally come with a complex structure.  
On the $B$-side, we take $\Lambda$ to be the Novikov field
\[
    \Lambda := \left\{\sum_{i=0}^\infty  a_i T^{\lambda_i}  \middle|a_i\in \CC, \lambda \in \RR, \lim_{i\to\infty} \lambda_i = \infty\right\}
\]
whose valuation map $\val: \Lambda\to \RR\cup \{\infty\}$  is the smallest exponent appearing the expansion of $\sum_{i=0}^\infty  a_i T^{\lambda_i}$.
 On the $A$-side, the torus fibration is given by
\begin{align*}
    \syza: X_A \to & Q\\
    ( x_1, \ldots, x_n)\mapsto& (\log|x_1|, \ldots, \log|x_n|)=(q_1, \ldots, q_n).
\end{align*} 
whose fibers are Lagrangian tori. The dual fibration $\tropb: X_B\to Q$ is given by taking coordinate-wise valuation 
\begin{align*}
    \tropb: X_B \to & Q\\
    ( z_1, \ldots  z_n)\mapsto& (\val(z_1), \ldots , \val(z_n)).
\end{align*}
Instead of using tropical geometry as an intuition for HMS, this paper uses HMS and our understanding of the tropical to $A$-correspondence to study the tropical to $B$ correspondence. 
We now review these correspondences before stating our results. 
\subsubsection*{Tropical to $B$ correspondence}
The tropical-to-complex correspondence and its applications to enumerative geometry have been a particularly rich field of study since the pioneering work of \cite{mikhalkin2005enumerative} which related counts of tropical curves in $\RR^2$ to counts of curves in the \emph{complex} algebraic torus (as opposed to the $\Lambda$ analytic torus we study). This relation consists of two parts: \emph{tropicalization}, which associates to a holomorphic curve in $(\CC^*)^2$ a tropical curve in $\RR^2$; and \emph{realization}, which lifts every tropical curve $V\subset \RR^2$ to a holomorphic curve in $(\CC^*)^2$. Both of these constructions have been extended to greater generality; we provide a coarse overview of the constructions here:
\begin{itemize}
    \item \emph{$B$-Tropicalization:}  The tropicalization map associates to a closed analytic subset $\YB\subset X_B$ its tropicalization $\tropb(\YB)\subset Q$. The expectation (which holds for algebraic subvarieties, \cite{groves1984geometry}) is that the tropicalization is a \emph{tropical subvariety} (\cref{def:tropicalsubvariety}).
    \item \emph{$B$-Realization:} Starting with $V\subset Q$ a tropical subvariety we say that $V$ is \emph{$B$-realizable} if there exists closed analytic subset $\YB\subset X_B$ with $\tropb(\YB)=V$.
\end{itemize}
One goal of tropical geometry is to determine which tropical subvarieties $V\subset Q$ are $B$-realizable.
Examples such as  \cite[Example 5.12]{Mikhalkin2004Amoebas} show that there exist tropical curves $V\subset \RR^n, n>2$ that are non-realizable. 
In some cases, there are criteria determining if a tropical subvariety is $B$-realizable. 

For example, if $V\subset Q$ is a tropical hypersurface, then there exists a tropical polynomial (piecewise integral affine convex function) $f: Q\to \RR$ so that $V$ is the locus of points where $f$ is non-affine. The function $f$ is called a tropical polynomial as it can be written using the tropical sum and product operations:
\begin{align*}
    \oplus: (\RR\cup\{\infty\})\times (\RR\cup \{\infty\}) \to (\RR\cup\{\infty\}) && q_1\oplus q_2=\min(q_1, q_2)\\
    \odot: (\RR\cup \{\infty\}) \times (\RR\cup \{\infty\}) \to (\RR\cup \{\infty \} )&& q_1\odot q_2= q_1+q_2
\end{align*}
Let $f=\oplus_{\alpha\in \NN^n} a_\alpha \odot q^{\odot \alpha}$ be a tropical polynomial whose non-affine locus is $V$. Let $\Lambda$ be a complete non-Archimedean valued field, and let $X_B=(\Lambda^*)^n$ be the algebraic torus. For each $a_\alpha$, select a coefficient $c_\alpha\in \Lambda$  whose valuation is $\val(c_\alpha)=a_\alpha$. Then the zero set of the polynomial $\sum_{\alpha\in \NN^n}  c_\alpha z^\alpha$ defines a subvariety of $X_B$ which is the $B$-realization of $V$.
Observe that this construction does not produce a unique lift.

The other examples where we have $B$-realization criteria are tropical curves.
In  \cite{nishinou2006toric}, it was shown that if $V\subset Q$ is a trivalent tropical curve of genus 0, then $V$ is realizable. This was extended to all balanced maps from trees in \cite{ranganathan2017skeletons}.
In higher genus, the space of deformations of a tropical curve may have higher dimension than the expected dimension of a possible $B$-realization. In this case, we say that the tropical curve is superabundant (\cite{mikhalkin2005enumerative}). We expect that a generically chosen superabundant curve is not $B$-realizable. 
It is known that all 3-valent non-superabundant curves are realizable \cite{cheung2016faithful}. 
In the superabundant setting, \cite[Theorem 3.4]{speyer2014parameterizing} established that if $V\subset Q$ is a tropical curve of genus 1 and satisfies a condition called  \emph{well-spacedness}, then $V$ is realizable.
\subsubsection*{Tropical to $A$ correspondence}
The tropical-to-Lagrangian correspondence is a more recent construction, independently arrived at in a series of papers  \cite{matessi2018lagrangian,mikhalkin2018examples,hicks2019tropical,mak2020tropically}.
Each of the papers associates to a (certain type of) tropical subvariety $V\subset Q$ a Lagrangian submanifold $L_V^\eps\subset X_A$ whose projection to the base of the Lagrangian torus fibration $\syza(L_V^\eps)$ is contained within an $\epsilon$-neighborhood of the tropical subvariety $V$. We call this a geometric Lagrangian lift of $V$.
When $V$ is a hypersurface, \cite{hicks2019tropical} proves that under homological mirror symmetry $L_V^\eps\subset (\CC^*)^n$ is identified with a sheaf whose support is a hypersurface $Y\subset (\Lambda^*)^n$.

In contrast to $B$-realization, the constructions in \cite{matessi2018lagrangian,mikhalkin2018examples,hicks2019tropical,mak2020tropically}  can construct a geometric Lagrangian lift $L_V$ of \emph{any} smooth tropical curve $V\subset Q$. This difference occurs because the map $\syza: X_A\to Q$ does not provide a good tropicalization map. For example, for any subset $U\subset Q$ and $\eps>0$ there exists a Lagrangian submanifold $L\subset X_A$ with the property that the Hausdorff distance between $\pi_A(L)$ and $U$ is less than $\eps$. Additionally, it would be desirable to have a tropicalization map that only depends on the Hamiltonian isotopy class of the Lagrangian submanifold --- and $\syza(L)$ can change substantially when we apply a Hamiltonian isotopy to $L$.

To obtain a correspondence from Lagrangian submanifolds to tropical subsets of $Q$, and justify why the Lagrangian $L_V$ is the ``correct'' $A$-model realization of a tropical curve $V$, one needs to employ techniques from Floer theory. Not all Lagrangian submanifolds are amenable to such analysis. We call a Lagrangian submanifold unobstructed if its filtered $A_\infty$ algebra $\CF(L)$ admits a bounding cochain. The pair $(L, b)$ of Lagrangian submanifold equipped with a bounding cochain is called a Lagrangian brane. Examples of unobstructed Lagrangian submanifolds include those which bound no pseudoholomorphic disks for a given choice of almost complex structure. In particular, if $L$ is exact, it is unobstructed. 

If $(L_1, b_1), (L_2, b_2)$ are two Lagrangian branes then there exists a cochain complex $\CF((L_1, b_1), (L_2, b_2))$ which is generated by the intersections of $L_1$ and $L_2$, and whose cohomology groups $\HF((L_1, b_1), (L_2,b_2))$ are invariant under Hamiltonian isotopies of either $L_1$ or $L_2$. We can use this to define $A$-tropicalization and $A$-realization.
\begin{itemize}
    \item \emph{$A$-Tropicalization:} Starting with the fibration $X_A\to Q$ and a Lagrangian brane $(L, b)\subset X_A$,    
    we define the $A$-tropicalization
    \[\tropa(L,b):= \{q\in Q \st \exists (F_q, \del) \text{ with } \HF((L, b), (F_q, \del)) \neq 0\}\]
    where $F_q=\syza^{-1}(q)$ is equipped with a unitary local system $\del$, and $\HF((L, b), (F_q,\del))$ is the Lagrangian intersection Floer cohomology of $(L ,b)$ with $F_q$ deformed by the local system $\del$. 
    An advantage of $\tropa(L,b)$ over $\syza(L)$ is that the former depends only on the Hamiltonian isotopy class of $L$.
    \item \emph{$A$-Realizability:} In light of the definition of $A$-tropicalization, we say that  $V\subset Q$ is $A$-realizable if there exists a Lagrangian brane $(L,b)\subset X_A$ with $\tropa(L)=V$. 
\end{itemize}
The Lagrangian submanifold $L_V^\eps$ associated to $V$ provides a geometric candidate for an $A$-realization of $V$. However, to verify $A$-realizability, one still needs to check that $L_V^\eps$ is unobstructed with bounding cochain $b$ and that $\tropa(L_V,b)=V$. We call this last condition \emph{faithfulness}. 	\subsection{Results}
The three components (geometric realizability, unobstructedness, and faithfulness) of the $A$-realizability problem and its implications for the $B$-realizability problem in $Q=\RR^n$ are summarized in the following diagram.
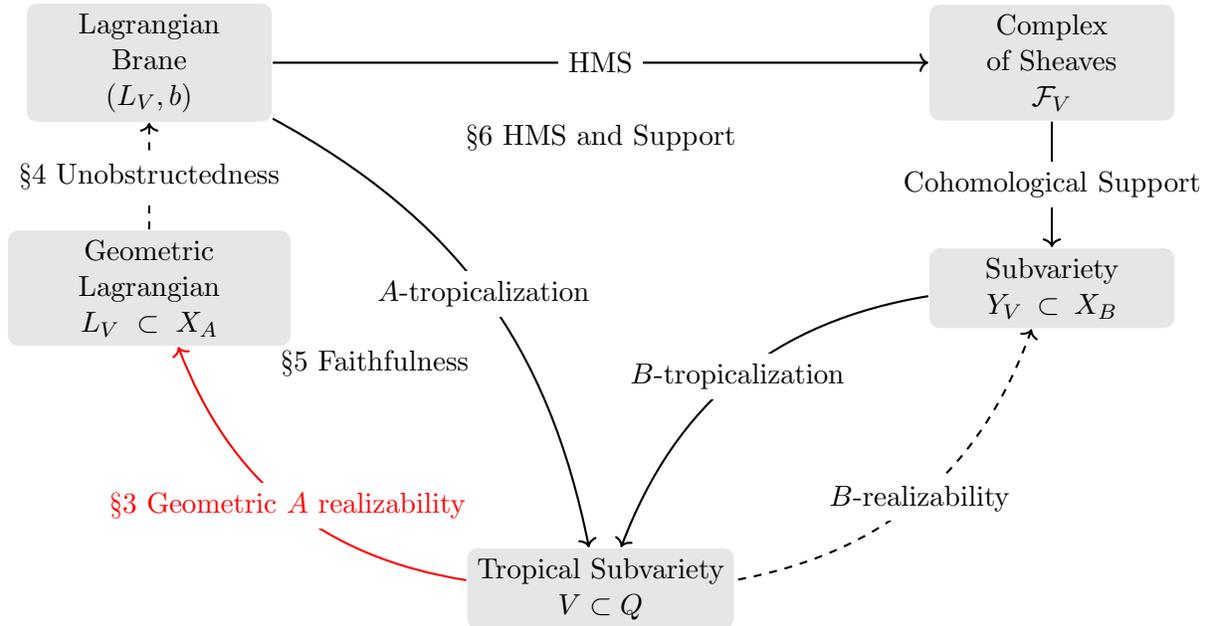
\[
    \begin{tikzpicture}
		\node[fill=gray!20, rounded corners, align=center] (v2) at (2,-2) {Tropical Subvariety\\ $V\subset Q$};
		\node[fill=gray!20, rounded corners, align=center, text width=3.5cm] (v1) at (-4,2) { Geometric Lagrangian  \\ $L_V\subset X_A$};
		\node[fill=gray!20, rounded corners, align=center, text width=3cm] (v3) at (8,5) {Complex of Sheaves \\ $\mathcal F_{V}$};
		\node[fill=gray!20, rounded corners, align=center, text width=3cm] (v4) at (8,2) {Subvariety\\ $ Y_V\subset X_B$};
	\node[fill=gray!20, rounded corners, align=center, text width=3cm] (v5) at (-4,5) {Lagrangian Brane \\ $(L_V, b)$};
		\draw  (v5) edge[bend left=25,->]node[fill=white]{$A$-tropicalization } (v2);
		\draw  (v2) edge[bend left,->,red]node[fill=white]{\S\ref{sec:geometricrealization} Geometric $A$ realizability} (v1);
		\draw  (v5) edge[->]node[fill=white]{HMS} (v3);
		\draw  (v3) edge[->]node[fill=white]{Cohomological Support} (v4);
		\draw  (v4) edge[bend left,<-,dashed]node[fill=white]{$B$-realizability} (v2);
		\draw  (v2) edge[bend left,<-]node[fill=white]{$B$-tropicalization} (v4);
\draw[dashed]  (v1) edge[->]node[fill=white]{\S\ref{sec:unobstructedness} Unobstructedness} (v5);
\node at (-1,1) {\S\ref{sec:faithfulness} Faithfulness};
\node at (2,4) {\S\ref{sec:realizableHMS} HMS and Support};

\end{tikzpicture} \]
The correspondences given by solid black lines always exist. Geometric $A$-realizability (the solid red arrow) is only known to exist for certain examples of tropical subvarieties of $Q$. For the applications we consider (smooth tropical hypersurfaces and curves) we always have geometric $A$-realizability. We conjecture that every tropical subvariety of $Q$ is geometrically $A$-realizable.
For any given tropical subvariety $V\subset Q$, there is no reason for either of the dashed arrows to hold. However, the following conjecture seems natural.
\begin{conjecture}
    Let $V\subset \RR^n$ be a tropical curve. Then a geometric Lagrangian lift $L_V$ is unobstructed if and only if $V$ is $B$-realizable. \label{conj:unobstructednessisrealizable}
\end{conjecture}
The main step in proving the forward direction of the conjecture is to establish the faithfulness of the Lagrangian brane lift, that is showing that $\tropa((L_V, b))=V$. The primary result of this article is to prove faithfulness (for \emph{all} tropical subvarieties admitting unobstructed Lagrangian lifts).
\begin{maintheorem}[Restatement of \Cref{thm:support}]
    Let $V\subset Q$ be a tropical subvariety. Let $(L_V^\eps, b)$ be a Lagrangian brane lift of $V$. Then $\tropa(L_V^\eps,b)=V$.
\end{maintheorem}
When we can apply homological mirror symmetry, we obtain the forward direction of \cref{conj:unobstructednessisrealizable}. Depending on the affine manifold $Q$ and Lagrangian $L_V$, we may require \cref{ass:hmscn}, which states that the family Floer construction of \cite{abouzaid2017homological}  extends to the non-compact and unobstructed setting.

\begin{maintheorem}[Restatement of \cref{cor:realizability}]
   Suppose \cref{ass:hmscn}. Let $V\subset \RR^n$ be a tropical subvariety. Suppose there exists $(L_V,b)\subset (\CC^*)^n$ a Lagrangian brane lift of $V$. Then $V$ is $B$-realizable.
\end{maintheorem}

The second goal of this paper is to show that this can be used to produce realizability criteria. We first recover a theorem of \cite{nishinou2006toric}.
\begin{maincorollary}[Restatement of \cref{thm:genuszerounobstructed}]
    Suppose \cref{ass:hmscn}. Every smooth genus zero tropical curve $V\subset \RR^n$ has a Lagrangian brane lift $(L_V, b)$ and is, therefore, $B$-realizable.
\end{maincorollary}
The results of \cite{nishinou2020realization} give necessary and sufficient conditions for when a tropical curve can be realized by a family of algebraic curves in a degenerating family of complex tori. In contrast to those results, our results show that every 3-valent tropical curve can be realized by a closed analytic subset. The following result \emph{does not} assume \cref{ass:hmscn}.
\begin{maincorollary}[Restatement of \cref{thm:abeliansurfaceunobstructedness}]
    Let $Q=T^2$ be a tropical abelian surface. Let $V\subset Q$ be a 3-valent tropical curve. $V$ has a Lagrangian brane lift $(L_V, 0)$, and is, therefore, $B$-realizable.
\end{maincorollary}
In summary, we can recover $B$-realizability results using unobstructedness for the first 5 cases in \cref{tab:realizability}.
\begin{table}\centering
        \begin{tabular}{p{3.5cm}|c|c|c}
        $V$ and $Q$  & \parbox{3cm}{\centering $A$-model\\ (Unobstructedness)} & \parbox{3cm}{\centering $B$-model\\ (Realizability)} & HMS Status\\\hline\hline
        Curves in abelian surfaces & \Cref{thm:abeliansurfaceunobstructedness} & \cite{nishinou2020realization}\footnotemark& $\checkmark$\\\hline
        Curves in $\RR^2$ & \cite{hicks2020tropical} & \cite{mikhalkin2005enumerative}&  (*)\\\hline
        Hypersurfaces of $\RR^n$ & \cite{hicks2020tropical} & Folklore&  (*)+(**)\\\hline
        Hypersurfaces in abelian varieties & \Cref{cor:abelianvarietyrealizable} & --- & (**)\\\hline
        Genus 0 curves in $\RR^n$ & Theorem C & \cite{nishinou2006toric} & (*)+(**)\\\hline
        Compact genus 0 curves in $\dim(Q)=3$ & \cite{mak2020tropically} & \cite{nishinou2006toric}& --- \\\hline
        Well-Spaced Genus 1 curves & Spec. in \cref{sec:speyer} & \cite{speyer2014parameterizing} & (*)+(**)
    \end{tabular}
    \caption{Relating $A$-unobstructedness to $B$-realizability. $(*)$ and $(**)$ refer to the needed extensions of family Floer cohomology (\cref{ass:hmscn}) to the non-compact and non-tautologically unobstructed settings.}
    \label{tab:realizability}
\end{table}
\footnotetext{The realization result of \cite{nishinou2006toric} considers $B$-tropicalizations coming from degenerating families of abelian surfaces so that a tropical curve is realized by a parameterized algebraic curve. The $B$-realization we take is by closed analytic subsets. In the setting of genus zero stable tropical curves in toric varieties these tropicalizations can be related by  \cite{ranganathan2017skeletons}. }
We also provide some insight into the existence of holomorphic curves with boundary on tropical Lagrangian submanifolds.
\begin{mainexample}[Restatement of \cref{example:obstructedline}]
    Let $V_c\subset \RR^3$ be a generic tropical line.
    The Lagrangian $L_{V_c}$ is unobstructed, but not tautologically unobstructed.
\end{mainexample}

\subsection{Outline}
In \cref{sec:background}, we give a toy computation that explores the entire road map above for a simple example, $V_{pants}\subset \RR^2$, the tropical pair of pants. 
In addition to providing context for the remainder of the paper, the computation reviews some background for tropical geometry and symplectic geometry. We also use this section to fix the notation for the paper. It is our hope that this section will be accessible to both tropical and symplectic geometers.

\Cref{sec:geometricrealization} discusses the geometric lifting problem. \Cref{def:geometricLagrangianlift} defines when a family of Lagrangian submanifolds $L^\eps_V$ is a geometric Lagrangian lift of a tropical subvariety $V$. We show that \cref{def:geometricLagrangianlift} distinguishes tropical subvarieties among all polyhedral complexes as the ones which permit geometric Lagrangian lifts.
\Cref{def:geometricLagrangianlift} requires that geometric Lagrangian lifts are monomially admissible, graded, and spin. In \cref{subsec:admissible,subsec:graded,subsec:spin}, we show that known constructions of geometric Lagrangian lifts of tropical subvarieties from \cite{matessi2018lagrangian,mikhalkin2018examples,hicks2019tropical} satisfy these conditions. 
We also prove \cref{lem:tropcurvetopology} which shows that for smooth genus zero tropical curves the map $H^2(L_V)\to H^2(\partial L_V)$ is an injection.

\Cref{sec:obstruction} investigates Lagrangian submanifolds which can be unobstructed by a bounding cochain (Lagrangian branes). We provide a brief overview of the pearly model for Lagrangian submanifolds in \cref{subsec:pearlymodel}. This is followed by examples of unobstructed geometric Lagrangian lifts (Lagrangian brane lifts) of tropical subvarieties in \cref{subsec:knownunobstructed} (summarized in \cref{tab:realizability}).
\Cref{subsec:unobstructedatboundary} gives a new method for checking unobstructedness of Lagrangian submanifolds inside non-compact symplectic spaces which have a potential function $W:X_A\to \CC$  (see \cref{def:bottleneck}).
\begin{maintheorem}[Restatement of \cref{thm:unobstructed}]
    Let $W: X_A\to \CC$ be a symplectic fibration outside of a compact set of $\CC$. 
    Let $L\subset X_A$ be a  $W$-admissible Lagrangian submanifold with boundary $M\subset W^{-1}(t)$, $t\in \RR_{\gg 0}$. Suppose $M$ is a tautologically unobstructed Lagrangian submanifold of $W^{-1}(t)$, and the connecting map $H^1(M)\to H^2(L,M)$ is surjective. Then there exists a bounding cochain $b$ so that $(L, b)$ is a Lagrangian brane.
\end{maintheorem}
The proof uses a lemma on filtered $A_\infty$ algebras (\cref{lemma:unobstructing}).
Since we have previously proven in  \cref{lem:tropcurvetopology} that the geometric Lagrangian lifts $L_V$ of smooth genus zero tropical curves satisfy the criterion of \cref{thm:unobstructed}, we obtain that such $L_V$ are unobstructed (\cref{thm:genuszerounobstructed}).

In \cref{sec:faithfulness}, we prove faithfulness (\cref{thm:support}) which shows that the $A$-tropicalization (Floer theoretic support) of a Lagrangian brane lift $L_V$ is $V$.
The proof uses that the Lagrangian intersection Floer cohomology between $(L_V,b)$ and $F_q$ is a deformation of the cohomology of a subtorus of $F_q$. An application of \cref{lem:submoduledeformation} shows that this can be ``undeformed'' by a bounding cochain so that $HF^0((L,\nabla_0,b_0), (F_q,\nabla, b))= \Lambda$.

\Cref{sec:realizableHMS} applies the previous constructions to address questions of realizability for tropical subvarieties.  
\cite[Remark 1.1]{abouzaid2017family} states that we expect that the family Floer functor can be adapted to include unobstructed Lagrangians.
We instead use \cref{ass:hmscn} --- the weaker assumption that the family Floer construction of \cite{abouzaid2017homological} can be employed for unobstructed Lagrangian submanifolds in the Lagrangian torus fibration $(\CC^*)^n\to \RR^n$ to construct a sheaf on the mirror space. We give a brief outline of the modifications to \cite{abouzaid2017homological} which would be required to prove \cref{ass:hmscn}. With this assumption, we prove the forward direction of \cref{conj:unobstructednessisrealizable} in \cref{cor:realizability}. We also discuss  the first 5 cases in \cref{tab:realizability}.

Finally, we discuss evidence towards the reverse direction of \cref{conj:unobstructednessisrealizable}. This requires us to understand some of the holomorphic disks which appear on Lagrangian lifts of tropical subvarieties. In \cref{example:obstructedline}, we show that the lift of the tropical line in $\RR^3$ bounds a holomorphic disk whose symplectic area is dictated by the internal edge length on the tropical line.
We also discuss applications of $B$ non-realizability to obstructedness in \cref{sec:speyer}. We consider the superabundant tropical elliptic curve $V\subset \RR^3$ from \cite[Example 5.12]{Mikhalkin2004Amoebas} and provide a sketch for how Speyer's well-spacedness criterion might be recovered from holomorphic disk counts on $L_V$.
\Cref{app:jacobian} looks at how to relate tropical line bundles on tropical curves to Lagrangian isotopies of their geometric Lagrangian lifts. We conjecture a relation between superabundance of a tropical curve $V$, and the relative ranks of $HF(L_V, b)$ and $H(L_V)$ (wide versus non-wide).

We provide some auxiliary results in the appendices. \Cref{app:pearlymodel} discusses how to adapt the pearly model of Floer cohomology in \cite{charest2019floer} to the setting of non-compact spaces equipped with a potential $W: X\to \CC$. In \cref{app:ainfinity}, we prove the results on filtered $A_\infty$ algebras used in this paper.

 \subsection{Acknowledgements}
	I'd particularly like to thank Ailsa Keating, Denis Auroux, Mark Gross, Grigory Mikhalkin, Nick Sheridan, and Ivan Smith for the many mathematical conversations which directly and indirectly contributed to this project. Additionally, I thank Andrew Hanlon and Umut Varolgunes for providing me with some thoughtful feedback and suggestions for this article.
I'm especially indebted to Dhruv Ranganathan, who also answered my many questions about tropical geometry; and Jean-Philippe Chass\'e and  Yoon Jae Nick Nho, who corrected an error in the reverse isoperimetric inequality for Lagrangian intersection Floer cohomology.
This article benefitted from the thoughtful comments and feedback of an anonymous referee.
Finally, the computation of holomorphic strips in \cref{exam:pairofpantsupport} stems from a discussion with Diego Matessi.
A portion of this work is adapted from Section 4.3 of my Ph.D. thesis at UC Berkeley.
This paper was supported by EPSRC Grant N03189X/1 at the University of Cambridge and ERC Grant 850713 (Homological mirror symmetry, Hodge theory, and symplectic topology) at the University of Edinburgh.  \section{A guided calculation to the support of the Lagrangian pair of pants}
	\label{sec:background}
	This section contains an expository computation that is designed to frame the main ideas of the paper, provide background, and fix notation.
The exposition here is not intended to be comprehensive, although we hope that through explicit examples, direct computations, and additional references, we've made this section accessible to both the tropical and symplectic geometry communities.
As a result, the materials outside of \cref{exam:pairofpantsupport,sec:disconnect} are expository.
As we will frequently use notation from \cref{exam:runningexample,exam:localsystems}, we suggest that the readers take a look at these computations of Lagrangian intersection Floer cohomology for conormal bundles in the cotangent bundle of the torus.

\subsection{\texorpdfstring{$A$-model, $B$-model}{A-model, B model}, and Lagrangian torus fibrations}
We provide a high-level overview of the \cite{strominger1996mirror,gross2001topological} viewpoint on mirror symmetry.
Let $Q$ be an integral affine manifold, that is a manifold equipped with a choice of integrable full-rank lattice $T_\ZZ Q \subset TQ$.
This identifies a dual lattice $T^*_\ZZ Q\subset T^*Q$, and also a flat connection on $TQ$.
There are three kinds of geometries that we may associate with $Q$: symplectic geometry, complex geometry, and tropical geometry.
\subsubsection*{$A$-model}
A symplectic manifold is a $2n$-manifold $X_A$ with a choice of two form $\omega\in \Omega^2(X_A)$ which is closed ($d\omega=0$) and nondegenerate ($\omega^n\neq 0$).
The submanifolds of interest for us in $X_A$ are Lagrangian submanifolds $L\subset X_A$, which are $n$-dimensional submanifolds on which the symplectic form vanishes ($\omega|_L=0$).
For any manifold $Q$, the cotangent bundle $T^*Q$ (whose local coordinates are $(q, p)$) carries a canonical symplectic form $\sum_{i=1}^n dq_i\wedge dp_i$.
This descends to a symplectic form on the quotient $X_A:= T^*Q/T^*_\ZZ Q$.

Given an integral affine submanifold $\underline V\subset Q$ so that $T_\ZZ \underline V\subset T_\ZZ Q$, the periodized conormal bundle $L_{\underline V}:=N^*\underline V/N^*_\ZZ \underline V\subset X_A$ is an example of a Lagrangian submanifold.
The simplest example is when we pick a point $q\in Q$ so that \begin{equation}
    L_{\underline q} = N^*q/N^*_\ZZ q= T^*_qQ/T^*_{q,\ZZ}Q
    \label{eq:conormalmodel}
\end{equation}
is a Lagrangian torus of $X_A$. We will call this Lagrangian torus $F_q\subset X_A$.
For this reason, we call the projection $\syza:X_A\to Q$  a Lagrangian torus fibration.
\subsubsection*{$B$-model}
We can also build an almost-complex manifold from the data of $Q$.
An almost complex structure on $X_B^\CC$ is an endomorphism $J: TX_B^\CC\to TX_B^\CC$ which squares to $-\id$.
The submanifolds of interest in the $B$-model are the \emph{almost-complex submanifolds} $\YB^\CC \subset X_B^\CC$ whose tangent spaces are fixed under the almost complex structure so that $J(T_{ y} \YB^\CC)= T_{ y}\YB^\CC$.

As $Q$ is integral affine, there exists a connection on $TQ$ whose flat sections are locally constant sections of $T_\ZZ Q$.
This provides a splitting $T(TQ)=T_qQ\oplus \ker(\pi)$.
We define an almost complex structure on $TQ$ which interchanges the components of this splitting with a sign:
\[J:=\begin{pmatrix} 0 & -\id \\ \id & 0 \end{pmatrix}.\]
The almost complex structure on $TQ$ descends to an almost complex structure on $X_B^\CC:= T^*Q/T^*_\ZZ Q$; the fibers of $\syzb: X_B^\CC\to Q$ are real tori.

Given an integral affine submanifold $\underline V\subset Q$, the periodized tangent bundle
\[ Y_{\underline V}^\CC:=T\underline{ V}/T_\ZZ \underline{V}\subset X_B^\CC\]is an example of an almost-complex submanifold. If we start with $q\subset Q$ a point, we see that $ Y_q\subset X_B^\CC$ is a point of $X_B^\CC$.
\subsubsection*{Mirror symmetry from Lagrangian torus fibrations}
We now describe in more detail the relationship between the Lagrangian tori of $X_A$ and the points of $X_B^\CC$. First, we note that for fixed $q\in Q$, there is a torus worth of points $ z$ in $X_B^\CC$ with the property that $\syzb( z)=q$.

In contrast to the complex lift, there is only one Lagrangian torus $F_q\subset X_A$ with $\syza(F_q)= \{q\}$.
To get a matching family of Lagrangian lifts to our complex lift, we consider  Lagrangian tori equipped with the additional data of a local system. Let $(F_q, \del)$ be a pair consisting of a Lagrangian torus $F_q$ and a choice of $U(1)$ local system on $F_q$. Then there is a bijection between pairs $(F_q, \del)\subset X_A$ and points $ z\in X_B^\CC$.
A similar story holds for the Lagrangian and complex lifts of integral affine subspace $\underline V \subset Q$.

To generalize beyond the submanifolds $\underline V, L_{\underline V}$ and $ Y_{\underline V}$ discussed above, we need to look at tropical geometry.

\begin{notation}
    Unless otherwise stated, we will only consider $Q=\RR^n$, so that  $X_A=(\CC^*)^n$ and $X_B^\CC = (\CC^*)^n$.
\end{notation}
\subsection{A quick introduction to tropical geometry and \texorpdfstring{$B$-tropicalization}{B-tropicalization}}
A convex polyhedral domain is the intersection of finitely many closed half-spaces in $\RR^n$,
\[\underline V= \{q\in Q \st \langle q, \vec v_i \rangle \geq \lambda_i\},\]
where $\vec v_i$ is a collection of vectors in $\RR^n$, and $\lambda_i$ is some set of constants in $\RR$. We say that this is a rational convex domain if each of the $\vec v_i \in \ZZ^n$, equivalently if there is a full lattice $T_\ZZ \underline V\subset T\underline{V}$ which is a sublattice of $T_\ZZ Q$.
A tropical subvariety is built out of these pieces. 
\begin{definition}
    A $k$-dimensional tropical subvariety $V\subset Q$ is a collection of $k$-dimensional rational convex polyhedral domains $\{\underline V_s\subset Q\}$ and weights $\{w_s\in \NN\}$ which are required to satisfy the following conditions:
            \begin{itemize}
        \item \emph{Polyhedral Complex condition:} At each pair of rational convex polyhedral domains, the intersection
              $\underline V_s \cap \underline V_t$
              is either empty, or a boundary facet of both $\underline V_s$ and $\underline V_t$.
        \item \emph{Balancing condition:} At facets $\underline W\subset \underline V_s$ let $\underline V_1, \ldots, \underline V_k$ be the rational polyhedral domains containing $\underline W$.
                      Consider lattices $T_\ZZ \underline W$, each of which is a sublattice of $T_\ZZ \underline V_i$ for each $i\in \{1, \ldots k\}$.
              Select for each $i$ a vector $\vec v_i\in T_\ZZ \underline V_i$ so that $T_\ZZ \underline V_i= T_\ZZ \underline W\oplus \langle \vec v_i\rangle$ as oriented lattices.
              We require that 
              \[\sum_{i}w_{i} \vec v_i \equiv 0 \in T_\ZZ Q/ T_\ZZ \underline W.\] 
    \end{itemize}
    \label{def:tropicalsubvariety}
\end{definition}
\begin{example}
    Consider the polyhedral domains  in $Q= \RR^2$
    \begin{align*}
        \underline V_1 = & \{(-t, 0)\st t\in \RR_{\geq 0}\} \\
        \underline V_2 = & \{(0,-t )\st t\in \RR_{\geq 0}\} \\
        \underline V_3 = & \{(t, t)\st t\in \RR_{\geq 0}\}.
    \end{align*}
    As the directions $\langle -1, 0 \rangle,\langle 0, -1\rangle , \langle 1, 1\rangle$ sum to zero this is balanced and gives us a tropical curve. 
        The collection of these three polyhedral domains is called the \emph{standard tropical pair of pants}. The curve $V_{pants}\subset \RR^2$ is drawn in \cref{fig:troppants}.
    \label{exam:tropicalpants}
\end{example}
We say that a tropical curve $V\subset \RR^n$ is \emph{smooth} if every 0-dimensional stratum is locally modelled after the pair of pants.
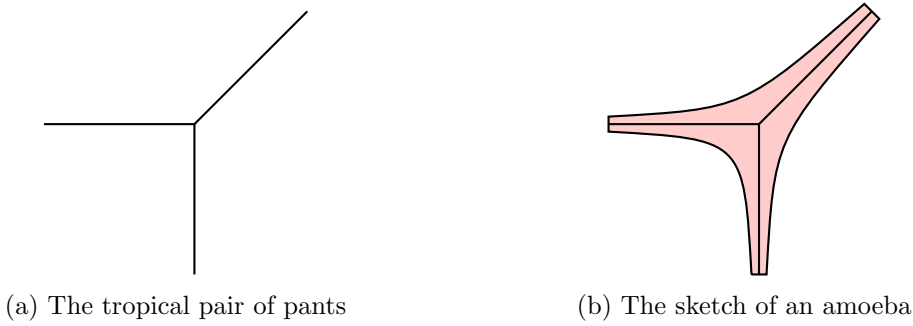
\begin{figure}
    \begin{subfigure}{.45\linewidth}
    \centering
    \begin{tikzpicture}

\draw (0,0) -- (-2,0);
\draw (0,0) -- (0,-2);
\draw (0,0) -- (1.5,1.5);
\end{tikzpicture}     \caption{The tropical pair of pants}
    \label{fig:troppants}
    \end{subfigure}
    \begin{subfigure}{.45\linewidth}
        \centering
        \begin{tikzpicture}

\draw[fill=red!20] (-0.1,-2) .. controls (-0.2,-0.2) and (-0.2,-0.2) .. (-2,-0.1) .. controls (-2,0) and (-2,0) .. (-2,0.1) .. controls (-0.2,0.2) and (-0.2,0.2) .. (1.4,1.6) .. controls (1.5,1.5) and (1.5,1.5) .. (1.6,1.4) .. controls (0.2,-0.2) and (0.2,-0.2) .. (0.1,-2) .. controls (0,-2) and (0,-2) .. (-0.1,-2);

\draw (0,0) -- (-2,0);
\draw (0,0) -- (0,-2);
\draw (0,0) -- (1.5,1.5);\end{tikzpicture}         \caption{The sketch of an amoeba}
        \label{fig:amoeba}
    \end{subfigure} 
    \caption{The tropical pair of pants approximates the Amoeba of a curve.}
\end{figure}
\begin{notation}
    Given $V\subset Q$ a tropical subvariety, we will use $V^{(0)}$ to denote the union of the interiors of the $\underline V_s$; we use $V^{(1)}$ to denote the union of the interiors of the boundaries of the $\underline V_s$; more generally we will use $V^{(i)}$ to denote the codimension $i$ linearity strata of $V$. 
        For any $\underline W\subset V^{(i)}$, let $\str(\underline W)$ be the set of all stratum which contain $\underline W$. If $V$ is a tropical curve, we will usually call the strata vertices and edges, and use $v,w$ for vertices and $e, f$ for edges.
    \end{notation}

\subsection{$B$-tropicalization}
$B$-tropicalization is the process of taking a subvariety of $X_B^\CC$ and obtaining a tropical subvariety of $Q$.
The first approach one considers is the image of $\YB^\CC\subset X_B^\CC$  under the  $B$-torus fibration
\begin{align*}
    \syzb:X_B^\CC\to Q.
\end{align*}
Under good conditions, $\syzb(\YB^\CC)\subset Q$ approximates a tropical subvariety of $Y$; see for instance \cite{Mikhalkin2004Amoebas}. 
The image $\syzb(\YB^\CC)$ is called the \emph{amoeba} of $\YB^\CC$, which computationally can be checked to approach the tropical curve (see \cref{fig:amoeba}).

To obtain a theory where the tropicalization of a subvariety is a tropical subvariety, we look to non-Archimedean geometry. Let $\Lambda$ be the Novikov field.
Given $M$ a rank $n$ lattice, denote by $X_B$ the torus $\Spec \Lambda[M]$.
The points of $X_B$ can be identified with $n$-tuples of invertible elements of $\Lambda$ so we will frequently write
\[X_B = \{(z_1, \ldots, z_n)\st z_i\in \Lambda^*\}.
\]
We build a tropicalization map by taking the valuation coordinate-wise:
\begin{align*}
    \tropb: X_B\to& M\tensor\RR=Q\\
    (z_1, \ldots, z_n)\mapsto& (\val(z_1), \ldots, \val(z_n))
\end{align*}
Given a $\YB\subset X_B$ a closed analytic subset, we call the image $\tropa(\YB)\subset Q$ its tropicalization.
\begin{example}
    Consider $M=\RR^2$, and the closed analytic subset $\YB\subset (\Lambda^*)^2$ given by
    \[\YB=\{(z_1, z_2) \st 1+z_1 + z_2 =0.\}\]
    We compute the valuation of such a point $(z_1, z_2)\in \YB$. Since $\val(1+z_1 + z_2)\geq \min(\val(1), \val(z_1 ),\val( z_2))$, with equality holding whenever the valuations differ, we obtain that for all $(z_1, z_2)\in \YB$  at least one of the following equalities hold:
    \begin{align*}
        \val(z_1)=\val(z_2) &  & \val(z_1)=\val(1) &  & \val(z_2)=\val(1)
    \end{align*}
    This means that the image of $\tropb(\YB)$ agrees with $V_{pants}\subset \RR^2$ from \cref{exam:tropicalpants}. It follows that $V$ is $B$-realizable.
\end{example}
This phenomenon holds much more broadly.
\begin{theorem}[\cite{groves1984geometry,gubler2007tropical}]
    Let $\YB \subset X_B$ be an irreducible $k$-dimensional analytic subset. Then $\tropb(\YB)$ is a $k$-dimensional polyhedral complex.
\end{theorem}
It is expected that when $\YB$ is an irreducible $k$-dimensional analytic subset, $\tropb(\YB)$ is a $k$-dimensional tropical subvariety. To our knowledge, this result has not appeared in the literature. A discussion on the current status of tropicalization for analytic subsets is included in \cite[Section 5.3]{sheridan2020lagrangian}.
\subsection{Floer cohomology and \texorpdfstring{$A$}{A}-tropicalization}
The definition of the $A$-tropicalization of a Lagrangian submanifold requires a little more exposition because we wish to do some computations of the $A$-tropicalization.
Our goal is to replace the Lagrangian torus fibration map $\syza:X_A\to Q$ with a correspondence of subsets
\[\tropa:\{\text{Lagrangian branes}\}\to \{\text{subsets of $Q$}\}\]
which only depends on the Hamiltonian isotopy class of the Lagrangian brane.
\subsubsection{Lagrangian Intersection Floer Cohomology}

The main computational tool that we will use in this paper is Lagrangian intersection Floer cohomology.
We first equip a symplectic manifold $(X, \omega)$ with an $\omega$-compatible choice of almost complex structure $J$.
\begin{definition}[\cite{floer1988morse}]\label{def:floercomplex}
Given a pair of transversely intersecting Lagrangian submanifolds $L_0, L_1\subset X$ and choice of almost complex structure $J$ satisfying the following conditions:
\begin{enumerate}[label=(\roman*)]
    \item $X, L_1, L_2$ are compact,\label{cond:compact}
    \item The symplectic area of all disks with boundary on $L_i$ vanish  $\omega(\pi_2(X, L_i)=0$,\label{cond:tautunobst}
    \item The Lagrangians $L_i$ are equipped with spin structures\label{cond:spin}
    \item The Lagrangians $L_i$ are graded (in the sense of \cite{seidel2000graded})\label{cond:graded}
    \item The moduli spaces of $J$-holomorphic strips in \cref{eq:strips} are regular. \label{cond:reg}
\end{enumerate} 
the Lagrangian intersection Floer cohomology is a chain complex
\begin{itemize}
    \item Whose generators are the points of intersection between $L_0$ and $L_1$, so that as a vector space
          \[\CF(L_0, L_1):= \bigoplus_{x\in L_0\cap L_1} \Lambda_x,\]
          where $\Lambda$ is the \emph{Novikov field}.
          The grading $\deg(x)$ of an intersection point $x\in L_0\cap L_1$ is determined by the Maslov index.
    \item The differential on this complex is defined by a count of holomorphic strips with boundary on $L_0 \cup L_1$ and ends limiting to the intersection points. Namely, let $x_\pm\in L_0\cap L_1$ be two intersection points, and $\beta\in H^2(X, L_0\cup L_1)$.
          Let
          \begin{equation}\mathcal M_\beta(L_0, L_1, x_+, x_-):=\left\{u:\RR_s\times[0, 1]_t\to X_A \middle|
              \begin{matrix}
                  u(s, 0)\in L_0, u(s, 1)\in L_1 \\ \lim_{s\to \pm\infty}u(s, t)=x_\pm\\ 
                  \bar \partial_J u =0           \\ [u]=\beta\in H_2(X_A, L_0\cup L_1)
              \end{matrix}
              \right\}/(s\mapsto s+c)\label{eq:strips}
            \end{equation}
          denote the moduli space of holomorphic strips with ends limiting to $x^\pm$ in the relative homology class $\beta$, up to reparameterization of the strip along the $s$-coordinate. Using the grading data on $L_0, L_1$, one can compute that
          \[\dim(\mathcal M_\beta(L_0, L_1, x_+, x_-))=\deg(x_-)-\deg(x_+)-1.\]
          The spin structures on $L_0, L_1$ provide orientations for the spaces $\mathcal M_\beta(L_0, L_1, x_+, x_-)$; in particular when $\deg(x_+)+1=\deg(x_-)$, then $\dim(\mathcal M_\beta(L_0, L_1, x_+, x_-))=0$ and we can count the points in this moduli space with signs.
          The structure coefficients of the differential  $d:\CF(L_0, L_1)\to \CF(L_0, L_1)$ are obtained by counting the elements in $\mathcal M_\beta(L_0, L_1, x_+, x_-)$,
          \[
              \langle d(x_+), x_-\rangle = \sum_{\beta\in H^2(X, L_0\cup L_1)} T^{\omega(\beta)} \#\mathcal M_\beta(L_0, L_1, x_+, x_-)
          \]
          where $\#$ is the signed count of points with orientation and  $T^{\omega(\beta)}$ records the symplectic area of the strip $u$ whose homology class is $\beta$.
        \end{itemize}
            \end{definition}
          The proof that this is a chain complex proceeds in a similar method to Morse theory.
          Because of \cref{cond:compact}, one can use Gromov compactness to prove that the 1-dimensional moduli spaces of strips have compactifications whose boundaries are given by products of the 0-dimensional moduli spaces of strips
          \[\partial\mathcal M_\beta(L_0, L_1, x_+, x_-)) = \bigsqcup_{x_0\in L_0\cap L_1 }\mathcal M_\beta(L_0, L_1, x_+, x_0))\times \mathcal M_\beta(L_0, L_1, x_0, x_-)).\]
          To ensure that the only broken configurations which show up in the compactification are given by strips breaking (as opposed to disk bubbling), we use \cref{cond:tautunobst} which states that there are no holomorphic disks with boundary on either $L_0$ or $L_1$.
          The compactification is compatible with the orientations given to the moduli spaces of holomorphic strips.
          Since the signed count of boundary components of a 1-dimensional manifold is zero, $\langle d^2(x_+), x_-\rangle=0$.
Unless otherwise stated, all Lagrangians we consider will be $\ZZ$-graded and spin.
A major feature of Lagrangian intersection Floer cohomology is its invariance under Hamiltonian isotopy.
\begin{theorem}[\cite{floer1988morse}]
    Let $L_0, L_1$ be Lagrangian submanifolds of $(X, \omega)$ satisfying conditions \cref{cond:tautunobst,cond:reg,cond:compact,cond:graded,cond:spin}. Let $\phi: X\to X$ be a Hamiltonian isotopy. Suppose that $L_0, L_1$ intersect transversely and we've picked $\phi$ in such a way that  $\phi(L_0), L_1$ intersect transversely. Then $\HF(L_0, L_1)=\HF(\phi(L_0), L_1)$.
    \label{thm:floerintersection}
\end{theorem}
For this reason, whenever $L_0, L_1$ do not intersect transversely, we can compute their Floer cohomology by taking a Hamiltonian perturbation which makes their intersection transverse; the resulting cohomology groups are independent of the choice of perturbation taken.
One can similarly show that it does not depend on the choice of an almost complex structure.

The conditions of \cref{cond:tautunobst,cond:compact} can be weakened. For example, \cref{cond:compact} --- which is required to prove that the moduli space of strips admit compactifications --- can be replaced with the weaker condition of monomial admissibility (\cref{def:admissiblitycondition}) or $W$-admissibility (\cref{def:bottleneck}). Later we will look at weakening the condition \cref{cond:tautunobst} to \emph{unobstructedness} (\cref{sec:obstruction}). We now drop \cref{cond:compact} and compute the Lagrangian intersection Floer cohomology between two Lagrangians in a cotangent bundle.  The computation we give is a direct generalization of \cite[Example 3.1]{smith2015symplectic}.
\begin{example}[Running Example]
    \label{exam:runningexample}
    Let $F_0=T^n$ be the $n$-dimensional torus. Let $T^{n-k}\subset F_0$  be the subtorus spanning the first $n-k$ coordinates on $T^n$.
    Then $T^*F_0$ is an example of an exact symplectic manifold.
    The zero section $F_0$ and the conormal bundle $N^*T^{n-k}$ are examples of exact Lagrangian submanifolds.
    Lagrangian intersection Floer cohomology requires that our Lagrangians intersect transversely, so we will apply a Hamiltonian perturbation to one of the Lagrangians to achieve transverse intersections.
    Pick $\lambda_0\in \RR_{>0}.$
    Consider the Hamiltonian function 
    \begin{equation}
        H=\sum_{i=1}^{n-k} \lambda_0\cos(\theta_i)
        \label{eq:LagrangianPrimitive}
    \end{equation}
     on $T^*F_0$.
    Let $\phi: T^*F_0\to T^*F_0$ be the time-one Hamiltonian flow of $H$.
    The resulting intersections of $\phi(N^*T^{n-k})$ with $F_0$ are the points
    \[\phi(N^*T^{n-k})\cap F_0 = \{(a_1\pi, \ldots, a_{n-k}\pi, 0, \ldots ,0)\st a_i\in \{0, 1\}\}\]
    and the index of each intersection point $x$ is given by $\deg(x)=\sum_{i=1}^{n-k} a_i$.
    We will call the corresponding generators of Floer cohomology $x_I$, where $a_i=1$ whenever $i\in I\subset \{1, \ldots, n-k\}$. Write $I\lessdot J$ if $I= J\cup\{x_i\}$ for some $i$. 
    As a vector space $\CF(\phi(N^*T^{n-k}), F_0)$ matches $\CM(T^{n-k})$ for the Morse function $H$.

    The differential on $\CF(\phi(N^*T^{n-k}), F_q)$ is related to the Morse differential.
    Let $|I|+1=|J|$, so that $\deg(x_I)$ and $\deg(x_J)$ differ by one. If $I$ and $J$ differ at more than two elements, then $\mathcal M_\beta(\phi(N^*T^{n-k}), F_0, x_I, x_J)$ has non-zero dimension.
    If $I\lessdot J$ differ at a single element $j$, then there are exactly two holomorphic strips travelling between $x_I$ and $x_J$,
    \[  \mathcal M(\phi(N^*T^{n-k}), F_0,x_I, x_J )=\{u_{I\lessdot J}^+, u_{I\lessdot J}^-\}\]
    which as points receive opposite orientations.
    By our choice of perturbation, the symplectic areas of the strips $u_{I\lessdot J}^-$ and $u_{I\lessdot J}^-$ agree (and is exactly $\lambda_0$).
     Therefore,
    \[\langle d (x_I), x_J\rangle = \left \{ \begin{array}{cc}
        T^{\omega(u_{I\lessdot J}^+)}- T^{\omega(u_{I\lessdot J}^-)}=0& \text{if $I\lessdot J$}\\ 0 & \text{otherwise} \end{array}\right.\]
    and we conclude that
    \[\HF(\phi(N^*T^{n-k}), F_q)= \Lambda\langle x_I\rangle = \bigwedge_{i\in \{1, \ldots, n-k\} }\Lambda \langle x_i\rangle.\]
\end{example}
    The example relates to the discussion of tropicalization as $T^*F_0$ can be identified with $X_A=(\CC^*)^n=T^*Q/T^*_\ZZ Q$.
    If $T^{n-k}$ is a linear subtorus of $F_q$, it corresponds to a $n-k$-dimensional subspace of $\tilde T^{n-k}\subset T^*_0Q$; let $\underline V\subset T_0Q$ correspond to the set of vectors which are annihilated by $\tilde T^{n-k}$. By abuse of notation, we use $\underline V$ to denote the integral affine subspace of $Q$ with prescribed tangent space at $0$.
    Under this identification $N^*T^{n-k}\subset T^*F_0$ is $L_{\underline V}\subset X_A$. Using that Lagrangian intersection Floer cohomology is invariant under symplectomorphisms, and noting that if $q\not\in \underline V$ then $F_q\cap L_{\underline V}=\emptyset$, we have computed 
        \[HF^\bullet(L_{\underline V}, F_q)=\left\{\begin{array}{cc} \bigwedge_{i\in \{1, \ldots, n-k\}} \Lambda\langle x_i\rangle & \text{if $q\in \underline V$}\\
        0 & \text{ if $q\not\in \underline V$}
    \end{array}\right.\]
\subsubsection{Local Systems}
Recall that the points of $X_B$ are in bijection with pairs $(F_q, \nabla)$ of Lagrangian torus fibers equipped with local systems. We now discuss how to incorporate this data into Lagrangian intersection Floer cohomology. The \emph{unitary Novikov elements} 
\[U_\Lambda:=\left\{a_0+\sum_{i=1}^\infty a_i T^{\lambda_i}\st \lim_{i\to\infty} \lambda_i=\infty, \lambda_i > 0 , a_0\in \CC^*, a_i\in \CC \right\}\]
are those elements whose non-zero lowest-order term is a constant. We now consider $(L_i, \nabla_i)$, which are Lagrangian submanifolds with the additional choice of a trivial $\Lambda$-line bundles $E_i$ and a $U_\Lambda$ local system $\del_i$.
Given $L_0, L_1$ which intersect transversely satisfying \cref{cond:tautunobst,cond:reg,cond:compact,cond:graded,cond:spin} we define $\CF((L_0, \del_0), (L_1, \del_1))$ to be the chain complex
\begin{itemize}
    \item whose underlying vector space is  $\bigoplus_{x\in L_0\cap L_1}\hom((E_0)_x, (E_1)_x))$ and;
    \item whose differential is given by taking a $\del_i$-weighted count of the holomorphic strips with boundary in $L_0\cup L_1$. More precisely: denote by $\partial_i u$ be the boundary of $u$ contained in $L_i$, and let $P^{\nabla_i}_\gamma: (E_i)_{\gamma(0)}\to (E_i)_{\gamma(1)}$ be the parallel transport induced by the local system along a path $\gamma: [0, 1]\to L_i$.

          As in the definition of Lagrangian intersection Floer cohomology without local systems, let $x_+, x_-\in L_0\cap L_1$ be intersection points with $\deg(x_+)+1=\deg(x_-)$. Given $\phi_x\in \hom((E_0)_{x_+}, (E_1)_{x_+}))$ and a holomorphic strip $u\in \mathcal M_\beta(L_0, L_1, {x_+}, {x_-})$ we obtain a map between the fibers above $x_0$,
          \[ P_{(\partial^1 u)}^{\del_1}\circ \phi_{x_+}\circ  P_{(\partial^0 u)^{-1}}^{\del_0}\in\hom( (E_0)_{x_-}, (E_1)_{x_-} ).\]
          The differential on $\CF(L_0, L_1)$ is defined by taking the contributions $u\cdot \phi_{x_+}$ over all holomorphic strips between $x_+$ and $x_-$, weighted by the symplectic area.
          \[ d_{\del_0, \del_1}(\phi_{x_+}):= \sum_{{x_-}\st \deg({x_-})=\deg({x_+})+1} \sum_{u\in M_\beta(L_0, L_1, {x_+}, {x_-})}\pm T^{\omega(\beta)}  P_{(\partial^1 u)}^{\del_1}\circ \phi_{x_+}\circ  P_{(\partial^0 u)^{-1}}^{\del_0}.\]
          where the sign is determined by the orientation of the moduli space.
\end{itemize}
When $\del_i$ are the trivial local systems, this recovers $\CF(L_0, L_1)$.
\begin{example}[Running Example, continued]
    We now return to \cref{exam:runningexample}. Fix coordinates on $F_q$, and let $\{c_1, \ldots, c_n\}$ be generators of $H^1(F_q, \ZZ)$ associated to the coordinate directions.  A local system on $F_q$ is determined completely by its monodromy on the $c_i$. Given a $\Lambda$-unitary local system $\nabla_1$ on $F_q$, we write $z_i= P_{c_i}^{\nabla_1}.$
    Let $\del_0$ be the trivial local system on $L_{\underline V}$.
    We now compute the differential on $\CF((L_{\underline V}, \del_0), (F_q, \del_1))$. Given $\phi_{x_I}\in \hom( (E_0)_{x_I}, (E_1)_{x_I} )$, and $I\lessdot J$ an index which differs at one spot $j$, we have
    \begin{align*}
        d_{\del_0, \del_1}(\phi_I)= & \left(T^{\omega(u_{I\lessdot J}^+)}P_{(\partial^1 u_{I\lessdot J}^+)}^{\del_1}\circ\phi_I\circ  P_{(\partial^0 u_{I\lessdot J}^+)^{-1}}^{\del_0}\right) \\&-\left ( T^{\omega(u_{I\lessdot J}^-)} P_{(\partial^1 u_{I\lessdot J}^-)}^{\del_1}\circ\phi_I\circ  P_{(\partial^0 u_{I\lessdot J}^-)^{-1}}^{\del_0}\right) \\
        \intertext{Recall that all of the holomorphic strips between intersection points differing in index by 1 have the same area   $\lambda_0=\omega(u_{I\lessdot J}^+)=\omega(u_{I\lessdot J}^-)$. Using that $\del_0$ is trivial local system }
        =                           & T^{\lambda_0} P_{(\partial^1 u_{I\lessdot J}^+)}^{\del_1}(\id - P_{c_j}^{\del_1})\circ\phi_I\circ P^{id}_{\partial^0{u_{I\lessdot J}^+}}
    \end{align*}
    This vanishes if and only if $P^{\del_1}_{c_j}=z_j=1$ for all $1\leq j \leq n-k$.
    We conclude that
    \[\HF((L_{\underline V}, \id),  (F_q, \del_1))= \left\{ 
        \begin{array}{cc} H^\bullet (T^{n-k}) & z_j=1 \text{ for all } 1\leq j \leq n-k \\
             0                        & \text{otherwise}\end{array}\right.\]
    \label{exam:localsystems}
\end{example}

\begin{notation}
    Given two Lagrangians $L_0, L_1$ which intersect transversely, we will pick at each intersection point $x\in L_0\cap L_1$ an isomorphism in $\hom((E_0)_x, (E_1)_x))$; by abuse of notation, we will denote this isomorphism also by $x \in \hom((E_0)_x, (E_1)_x))$. We can in this way write
    \[\CF(L_0,L_1)=\Lambda \langle x\rangle\]
    and the differential on this complex will be given by the structure coefficients
    \[\langle d(x),y\rangle = \sum_{u\in \mathcal M_\beta(L_0, L_1, x, y)}T^{\omega(\beta)} P_{\partial u}^{\del_1,\del_2} \]
    where $P_{\partial u}^{\del_1,\del_2}\in U_\Lambda$ is a unitary element determined by $P_{\partial u}^{\del_1,\del_2}\cdot y= P_{(\partial^1 u)}^{\del_1}\circ x\circ  P_{(\partial^0 u)^{-1}}^{\del_0}$. This allows us to use the simpler (and more commonly employed) notation from \cref{def:floercomplex}.
    \label{not:chains}
\end{notation}
\subsubsection{\texorpdfstring{$A$}{A}-tropicalization}
We are now ready to define the $A$-tropicalization.
When considering a complex space $X^\CC_B$ on the $B$ side, we used the projection $\syzb: X^\CC_B\to Q$ to obtain from each subvariety of $X^\CC_B$ an amoeba which approximated the tropical subvariety. 
Just as with the $B$-tropicalization, given a Lagrangian submanifold $L\subset Q$ we could consider the Lagrangian torus fibration image of a Lagrangian submanifold $\syza(L)\subset Q$.
However, since even Hamiltonian isotopic Lagrangian submanifolds can have different projections to the base of the Lagrangian torus fibration, this does not provide a very good definition of $A$-tropicalization. Instead, we use Lagrangian intersection Floer theory to define the $A$-tropicalization.
\begin{definition}[Preliminary]
    Let $L\subset X_A$ be a Lagrangian submanifold satisfying the conditions of \cref{def:floercomplex}. We define the \emph{$A$-tropicalization} or Floer theoretic support  of $L$ to be the set
    \[\tropa(L):= \left\{q\in Q \middle| \parbox{8cm}{There exists a Lagrangian brane $(F_q, \del)$ with $\HF(L,( F_q, \del)) \neq 0$}\right\}.\]
    \label{def:preliminaryAtropicalization}
\end{definition}
The $A$-tropicalization is a decategorification of a much more powerful invariant captured by family Floer theory due to \cite{fukaya2002floer,abouzaid2017family}. From this viewpoint, the chain complexes $\CF(L, (F_q, \del))$ should be considered as the stalks of a sheaf which appropriately bundled together into a sheaf on $X_B$.
This viewpoint on tropicalization is also employed in \cite{sheridan2020lagrangian}.
The $A$-tropicalization is a refinement of projection to the base of the Lagrangian torus fibration in the following sense:
\begin{prop}
    Let $L\subset X_A$ be a Lagrangian brane. Then $\tropa(L)\subset \syza(L)$.
    \label{prop:syzcontainstrop}
\end{prop}
\begin{proof}
    Suppose that $q\notin \syza(L)$. Then $F_q= \syza^{-1}(q)$ is disjoint from $L$. As the Floer intersection is complex is generated on the intersection points, $\CF(L,( F_q, \del))=0$.
\end{proof}
While $\tropa(L)\subset \syza(L)$ always holds, it will rarely be the case that $\syza(L)\subset \tropa(L)$.
By invariance of $\tropa(L)$ under Hamiltonian isotopies, we obtain that 
\[\tropa(L)\subset \bigcap_{\phi\in \text{Ham}(X_A) }(\syza(\phi(L)))\] where $\text{Ham}(X_A)$ is the set of Hamiltonian isotopies of $X_A$.
However, there is no reason to expect even this to be equality. \Cref{sec:faithfulness} proves that when $L$ is a Lagrangian constructed from the data of a tropical subvariety of $Q$ the above inclusion becomes equality. We see a toy version of this statement below. \begin{example}[Continuation of Running Example]
    We now are able to compute the $A$-tropicalization of a Lagrangian submanifold.
    Let $\underline V\subset Q$ be an integral affine $k$-subspace, so that $L_{\underline V}\subset X_A$ is a $T^{n-k}\times \RR^k$ Lagrangian submanifold. We now compute the $A$-tropicalization of $q$.
    Since $\syza (L_{\underline V})= \underline V$, by \cref{prop:syzcontainstrop} we have that $\tropa(L_{\underline V})\subset \underline V$.
    By \cref{exam:localsystems}, whenever $q\in \underline V$ there exists a local system so that $\HF(L_{\underline V}, (F_q, \del_1))\neq 0.$
    Therefore $\tropa(L_{\underline V})=\underline V$.
\end{example}
In this example, we see there are three steps of the $A$-realizability problem.
\begin{enumerate}
    \item  First, we constructed a \emph{geometric} lift $L_{\underline V}$ of $\underline V$.
    \item The second step is to show that we have well-defined Floer cohomology groups. In the example above, this follows  from  $\pi_2(X_A, L_{\underline V})=0$, but more generally amounts to showing that the Lagrangian $L_{\underline{V}}$ is \emph{unobstructed}.
    \item Finally, the computation of support from \cref{exam:localsystems} proves that this is a \emph{faithful} lift of $\underline V$.
\end{enumerate}

In the example of the lift of $\underline V$, we can do slightly more than compute the tropicalization of $L_{\underline V}$. We compute the \emph{A-support}, which is the set of pairs $(F_q, \del)$ which have non-trivial pairing with $L_{\underline V}$:
\begin{equation}
    \suppA(L_{\underline V}) = \{(F_q, \del)\st q\in \underline V, P^\del_{c}=0 \text{ for $c\cdot \underline V=0$}\}
    \label{eq:asupport}.
\end{equation}
Here, we identify $H_1(F_q,\ZZ)$ with $T^*_\ZZ(Q)$. At each point $q\in \underline V$, there is a $(U_\Lambda)^k$ choice of local systems satisfying the above criteria. The support can be identified with the set $\suppA(L_{\underline V})=\underline V\times (U_\Lambda)^k= (\Lambda^*)^k\subset X_B$.

\subsection{\texorpdfstring{$A$}{A}-tropicalization for the pair of pants}
\label{exam:pairofpantsupport}
In this subsection, we carry out the entire $A$-realizability process with the tropical curve $V_{\pants}$ from \cref{exam:tropicalpants}.
This computation first appeared in unpublished work from \cite[Section 4.3]{hicks2019tropical}, and stems from a discussion with Diego Matessi. We use this example computation to outline the remainder of the paper.

\subsubsection*{Geometric Realizability: \cref{sec:geometricrealization}} We first need to discuss the process of building a Lagrangian submanifold which geometrically is a lift of $V$ in the sense that $\syza(L_{V_\pants})$ approximates $V_{\pants}$. In dimension 2, one can obtain Lagrangian submanifolds in $(\CC^*)^2$ by hyperK\"{a}hler rotation of complex curves\footnote{We only use this construction for the ease by which it builds a Lagrangian pair of pants in dimension 2; we emphasize at this juncture that hyperK\"{a}hler rotation is \emph{not} mirror symmetry.}.
We therefore can build a Lagrangian lift of $V_{pants}$ by starting with the holomorphic lift $\{(z_1, z_2) \st 1+z_1+z_2=0\}\subset (\CC^*)^2$ and applying hyperK\"{a}hler rotation. For every $\epsilon>0$, we can find a Lagrangian submanifold $L^{\eps}_{V_{pants}}\subset X_A$ Hamiltonian isotopic to our hyperK\"{a}hler rotation with the following properties:
\begin{itemize}
    \item When restricted to the complement of a neighborhood of $0\in Q$, we have
          \[L^{\eps}_{V_{pants}}\setminus \syza^{-1}(B_\eps(0))= L_{\underline V_1}\cup L_{\underline V_{2}}\cup L_{\underline V_{3}} \setminus \syza^{-1}(B_\eps(0)).\]
          This is one of the properties which characterizes a Lagrangian lift of a tropical curve.
    \item Furthermore, we can construct this Lagrangian so that it is symmetric under the permutation of coordinates $(z_1, z_2)$ on $X_A$.
\end{itemize}

\subsubsection*{Unobstructedness: \cref{sec:unobstructedness}} The next step to the $A$-realization process is to show that the Lagrangian submanifold one builds can be analyzed with Floer theory. In this example, $L_{V_\pants}$ is exact and so $\omega(\pi_2(X_A, L_{V_\pants}))$ vanishes. It follows that $\HF(L_{V_\pants}, F_q)$ will be well defined.

\subsubsection*{Faithfulness: \cref{sec:faithfulness}} We now compute $\tropa(L_{V_\pants})$. Consider the Lagrangian pair of pants $L_{V_{pants}}$, the Lagrangian fiber $F_q$, and the holomorphic cylinder $z_1=z_2$ as drawn in \cref{fig:largestrip}.
We take Hamiltonian perturbations so that the Lagrangian submanifolds intersect transversely.
Nearby the point $q$, the Lagrangian $L_{V_{pants}}$ agrees with $L_{\underline V_3}$; therefore $F_q\cap L_{V_{pants}}= F_q \cap L_{\underline V_3}$. Following the notation from \cref{exam:localsystems}, we call the degree 0 intersection point $x_\emptyset$, and the degree 1 intersection point $x_1$. In addition to the agreement of intersection points, there are two ``small strips'' contributing to the differential on $\CF(L_{V_{pants}}, F_q)$ which match the strips in the differential of $\CF(L_{\underline V_3}, F_q)$. We call these holomorphic strips $u^+_{x_\emptyset<x_1}, u^-_{x_\emptyset <x_1}$.

From the symmetry of our setup, the Lagrangian  $L_{V_{pants}}$ intersects the complex plane $z_1=z_2$ cleanly along a curve.
Furthermore, the holomorphic cylinder $z_1=z_2$ intersects $F_q$ along a circle; therefore the portion of $z_1=z_2$ bounded by $L_{V_{pants}}$ and $F_q$ gives an example of a holomorphic strip with boundary on $L_{V_{pants}}$ and $F_q$. The ends of this holomorphic strip limit toward $x_\emptyset$ and $x_1$.
The valuation projection of this strip is a line segment connecting the point $q$ with the vertex of the tropical pair of pants.
For this reason, we will call this holomorphic strip $u^{qv}$.
The area of this strip is the length of the line segment corresponding to $\syza( u^{qv})$. The three holomorphic strips are more readily seen by considering the argument projection of $L_{V_{pants}}$ to $F_q$ as in \cref{fig:argumentprojection}

\begin{figure}
    \begin{subfigure}[t]{.4\linewidth}
    \centering
    \begin{tikzpicture}

    \draw[fill=gray!20] (-2.5,0.5) .. controls (-2,1) and (-1.5,1.5) .. (-1,2) .. controls (-0.75,2.25) and (-0.25,2.5) .. (0,2.5) .. controls (0.5,2.5) and (1,2.5) .. (1.5,2.5) .. controls (1,2.5) and (0.5,2.5) .. (0,2.5) .. controls (-0.25,2.5) and (-0.5,2.75) .. (-0.5,3) .. controls (-0.5,3.5) and (-0.5,4) .. (-0.5,4.5) .. controls (-0.5,4) and (-0.5,3.5) .. (-0.5,3) .. controls (-0.5,2.75) and (-0.75,2.25) .. (-1,2);

        \node at (-2,4) {$Q$};\begin{scope}[scale=-1, shift={(1,-5)}]
    
        \draw[dashed] (-0.5,2.5) -- (-2.5,2.5);
        \draw[dashed] (-0.5,2.5) -- (-0.5,0.5);
        \draw[dashed] (-0.5,2.5) -- (1.5,4.5);
    \end{scope}
    \draw[blue, thick] (-2.5,0.5) -- (1.5,4.5);
    \node at (2,5) {$z_1=z_2$};
    \node[right] at (1.5,2.5) {$L_{V_\pants}^\epsilon$};
    \node at (-1.5,1) {$F_q$};
    \draw[dotted]  (-0.5,2.5) ellipse (0.5 and 0.5);
\node at (-0.5,2) {$\epsilon$};
\node[above] at (-1.5,1.5) {$e_3$};
\node[above] at (1,2.5) {$e_1$};
\node[left] at (-0.5,4) {$e_2$};
\node[green, fill, circle, scale=.4] at (-2,1) {};
\clip (-1.5, 1.5)--(-2.5, .5) --(-2.5, 1.5);
\node[red, fill, circle, scale=.4] at (-2,1) {};
\end{tikzpicture}     \caption{The intersection of the blue holomorphic cylinder and the tropical Lagrangian pair of pants is clean, and gives a holomorphic strip with boundary on $L_{V_{\pants}}$ and $F_q$. }
    \label{fig:largestrip}
    \end{subfigure}\hspace{.1\linewidth}
    \begin{subfigure}[t]{.4\linewidth}
        \centering
        \begin{tikzpicture}

\draw  (-2,2) rectangle (1,-1);
\draw[fill=gray!20] (-2,0.5) -- (-0.5,0.5) -- (-0.5,-1) -- cycle;
\draw[fill=gray!20] (-0.5,2) -- (-0.5,0.5) -- (1,0.5) -- cycle;
\begin{scope}
\clip  (-2,2) rectangle (1,-1);
\end{scope}
\node[left] at (-1.25,-0.25) {$x_\emptyset$};
\node[above] at (0.25,1.25) {$x_1$};
\node at (-1.25,-0.25) {};
\draw[red, thick, ->] (-1.25,-0.25) -- (-2,0.5) (1,0.5) -- (0.25,1.25);
\draw[green, thick, ->] (-1.25,-0.25) -- (-0.5,-1) (-0.5,2) -- (0.25,1.25);
\draw[blue, thick, ->] (-1.25,-0.25) -- (0.25,1.25);
\node[fill=black, circle, scale=.2] at (0.25,1.25) {};
\node[fill=black, circle, scale=.2] at (-1.25,-0.25) {};
\node at (-1.25,0.25) {$L_V$};
\end{tikzpicture}         \caption{The argument projection of $L_{V_\pants}$ to $F_q$. The intersection points are labelled. The three holomorphic strips are denoted by the arrows, with $u^{qv}$ drawn in blue.}
        \label{fig:argumentprojection}
    \end{subfigure}
    \caption{}
\end{figure}
We will think of $u^{qv}$ as being a ``big strip'' as we can choose $\lambda_0$ small enough so that  $\lambda_0=\omega(u^+_{x_\emptyset<x_1})=\omega(u^-_{x_\emptyset<x_1})\ll \omega(u^{qv})$.
If no local systems are used, the differential on $\CF(L_{V_{pants}},F_q)$ is
\[
    d(x_\emptyset)=\left(T^{\omega(u^+_{x_\emptyset<x_1})}-T^{\omega(u^-_{x_\emptyset<x_1})}+T^{\omega(u_{qv})}\right)\cdot x.
\]
This does not vanish, so $\HF(L_{V_{pants}}, F_q)=0$.

However, to compute the $A$-support we must compute Lagrangian intersection Floer cohomology where we equip $F_q$ with a local system. We characterize the local system $\del$ on $F_q$ in terms of its holonomy along the $\arg(z_1)$ and $\arg(z_2)$ loops of $F_q$, giving us quantities $(\exp(b_1), \exp(b_2))\in (U_\Lambda)^2$.
We'll denote this non-unitary local system by $\del_{b_1,b_2}.$
Given a point $q=(-a, -a)\in \underline V_3$, we compute the quantities
\begin{align*}
    \omega(u^+_{x_\emptyset<x_1})=\lambda_0                                                             &  & \omega(u^-_{x_\emptyset<x_1})=\lambda_0                                                            &  & \omega(u^{qv}) =-a+\lambda_0                                     \\
    P^{\del_{b_1,b_2}}_{\partial u^+_{x_\emptyset<x_1}}= \exp\left(\frac{1}{2} (b_1-b_2)\right) &  & P^{\del_{ b_1,b_2}}_{\partial u^-_{x_\emptyset<x_1}}= \exp\left(\frac{1}{2} (b_2-b_1)\right) &  & P^{\del}_{\partial u^{qv}}= \exp\left(\frac{1}{2} (b_1+b_2)\right).
\end{align*}
The weights given by the local system are determined by the paths drawn in \cref{fig:argumentprojection}.
from which we obtain the differential on the $\CF(L_{V_{pants}},(F_q, \del_{b_1,b_2}))$:
\begin{align*}
    \langle d_{\del_{b_1, b_2}}(x_{\emptyset}), x_1 \rangle = & \overbrace{\left(P^{\del_{b_1,b_2}}_{\partial u^+_{x_\emptyset<x_1}}\cdot T^{\omega(u^+_{x_\emptyset<x_1})}-P^{\del_{b_1,b_2}}_{\partial u^-_{x_\emptyset<x_1}}\cdot T^{\omega(\partial u^-_{x_\emptyset<x_1})} \right)}^{\text{Small Strips near $q$}}
    +\overbrace{P^{\del}_{\partial u^{qv}}\cdot T^{\omega(u^{qv})}}^{\text{Large Strips}}                                                                                                                                                                       \\
    =                                        & T^{\lambda}_0\left( \exp\left(\frac{1}{2} (b_1-b_2)\right)  - \exp\left(\frac{1}{2} (b_2-b_1)\right)  +   \exp\left(\frac{1}{2} (b_1+b_2)\right) \cdot T^{-a}   \right)                                                       \\
    =                                        & T^{-a+{\lambda_0}}\exp\left(-\frac{1}{2}(-b_1-b_2)\right)\left( T^{a} \exp(b_1)-  T^{a} \exp(b_2) + 1 \right)
\end{align*}
This always has a $U_\Lambda$-worth of solutions obtained by setting $b_1= \log(T^{-a}(T^a \exp(b_2)-1))$.
Therefore $(-a, -a)\in \tropa(L_{V_{pants}})$. From this we conclude that $\tropa(L_{V_{pants}})=V_{pants}$.

This is one of the rare situations where we can compute the Floer theoretic support explicitly: under the substitution $z_1= T^{a_1} \exp(b_1), z_2 =  T^{a_2} \exp(b_2) $, the Lagrangian tori $(F_{a_1, a_2}, \del_{b_1, b_2})$belong to the support of $L_{V_{pants}}$ if and only if  $z_1-z_2 +1 =0$. This should be compared with the computation of the $B$-realization of $V_{pants}$ from \cref{exam:tropicalpants}.

\subsection*{$B$-realizability: \cref{sec:realizableHMS}}
\label{sec:disconnect}
The matching of the supports of the $A$- and $B$- realizations of $V_{pants}$ can be captured in the language of homological mirror symmetry. This requires a description of the \emph{Fukaya category} of a symplectic manifold. The \emph{Fukaya pre-category} of a compact symplectic manifold $(X,\omega)$ is given by:
\begin{itemize}
    \item Objects are given by mutually transverse Lagrangian submanifolds $L\subset X$ which are graded, spin, and tautologically unobstructed (\cref{sec:unobstructedness}).
    \item For $L_0\neq L_1$, the morphisms $\hom(L_0, L_1)$ are given by Lagrangian intersection Floer cochains $\CF(L_0, L_1)$.
    \item $k$-compositions of morphisms
          \[m^k: \bigotimes_{i=0}^{k-1} \hom^{g_i}(L_i, L_{i+1}) = \hom^{2-k+\sum g_i}(L_0, L_k),\]
          are given by counts of holomorphic polygons with boundary on the $L_k$.
\end{itemize}
This is an \emph{$A_\infty$ pre-category}, meaning that for every collection of objects $L_0, \ldots, L_k$, the  filtered $A_\infty$ relations hold:
\[\sum_{j_1+j+j_2=k} (-1)^\clubsuit m^{j_1+1+j_2}(\id^{\tensor j_1} \tensor m^{j} \tensor \id^{\tensor j_2})=0.\]
Here $\clubsuit=j_1+\sum_{1}^{j_1} g_i$, and $k\geq 1$.

The pre-category can be appropriately completed to give a triangulated $A_\infty$ category, the Fukaya category $\Fuk(X_A)$.
Some of the hypotheses of the construction can be dropped or modified: for example, if $X_A$ is a cotangent bundle (and not compact) there is a version of the Fukaya category (the wrapped Fukaya category, $\mathcal W(X_A)$) which can be defined with appropriate Lagrangian submanifolds. $X_A=(\CC^*)^n=T^*F_0$ is one of these cases.

The homological mirror symmetry conjecture predicts that on mirror spaces the Fukaya category and derived category of coherent sheaves are derived equivalent.
\begin{theorem*}
    Let $X_A=(\CC^*)^n$, and $X_B^\CC=(\CC^*)^n$.
    There is an equivalence of derived categories:
    \[ \mathcal F: \mathcal W(X_A)\to D^b_{dg}\Coh(X_B^\CC).\]
    between the wrapped Fukaya category of exact admissible Lagrangian submanifolds of $X_A$ and the bounded derived category of coherent sheaves on $X_B^\CC$. 
\end{theorem*}
The proof of the theorem first shows that the zero section $L(0)$ of $\syza: X_A\to Q$ is a Lagrangian submanifold that generates $\Fuk(X_A)$. Then, $\HF(L(0), L(0))$ is shown to be the algebra $\CC[(\ZZ)^n]=\hom(\mathcal O_{(\CC^*)^n}, \mathcal O_{(\CC^*)^n}).$ Since this generates $D^b_{dg}\Coh(X_B^\CC)$, we know that these two categories are equivalent.
However, this proof is non-constructive: given an arbitrary exact Lagrangian submanifold $L\subset X_A$, there is no immediate way of determining the corresponding mirror sheaf in $D^b_{dg}\Coh(X_B^\CC)$.
There are a few objects which we can match up under this functor. Let $(F_q, \del)$ be an exact fiber of the Lagrangian torus fibration.
Then $\mathcal F(F_q, \del) \simeq \mathcal O_{z}$ for some $z \in X_B^\CC$.
From here, we obtain the following toy result, whose extension to the general $V$ is the objective of the remainder of this paper.
\begin{theorem}
    $V_{pants}\subset \RR^2$ is $B$-realizable.
\end{theorem}
\begin{proof}
    From \cref{exam:pairofpantsupport}, we proved that $V_{pants}$ is $A$-realizable by a Lagrangian $L_{V_{pants}}$. The support of the mirror sheaf $\mathcal F(L_{V_{pants}})$ is a $B$-realization of $V_{pants}$.
\end{proof}
\begin{remark}
    There are other approaches to homological mirror symmetry which would yield the same theorem. A stronger result than what is given here would be to show that $V_{pants}$ is realizable, and its realization compactifies to a subvariety (a line) in the projective plane. To prove this result, one would first show that $L_V$ belongs to an appropriate ``partially-wrapped Fukaya category'', and apply a homological mirror symmetry theorems for toric varieties for the appropriate partially-wrapped Fukaya category (\cite{kuwagaki2020coherent,ganatra2018sectorial} or \cite{abouzaid2006homogeneous,hanlon2019monodromy,hanlon2020functoriality}). Then one would need a mirror symmetry statement for the exact Lagrangian torus fiber equipped with non-unitary local systems, and replicate the argument of \cref{exam:pairofpantsupport}.
\end{remark} \section{Geometric realization}
	\label{sec:geometricrealization}
	
The flexibility of Lagrangian submanifolds both complicates and simplifies the construction of a Lagrangian lift of a tropical subvariety. 
The additional flexibility means that we have a lot of wiggle room to construct a potential lift; however, identifying a Lagrangian submanifold as ``the'' lift of a tropical subvariety becomes impossible. 
For example, given any candidate lift $L_V$ of a tropical subvariety $V$, one could apply a Hamiltonian isotopy to $V$ to obtain a new Lagrangian submanifold.
More generally, each potential Lagrangian lift $L_V$ of $V$ is supposed to represent the data of a sheaf on $X_B$ whose support has tropicalization $V$; there are many such sheaves!

Despite all of this flexibility, we already have a good idea of what the Lagrangian lift $L_V$ of $V$ should look like from \cref{eq:conormalmodel}. Recall that $V^{(0)}$ is the union of the interiors of the top dimensional polyhedral domains $\underline V$ defining $V$. At each component we can take the conormal torus construction to obtain a Lagrangian chain: 
\[L_{V^{(0)}}:=\bigcup_{\underline V\subset V^{(0)}} L_{\underline V}.\]
Intuitively, a geometric Lagrangian lift of $V$ should approximate the chain $L_{V^{(0)}}$.
\begin{remark}
    Fix an orientation on $F_q$, a fiber of the SYZ fibration. Then $L_{\underline V}$ inherits an orientation (which in local coordinates comes from $dq_1\wedge \cdots \wedge dq_k\wedge dp_{k+1}\wedge \cdots \wedge dp_n$). We will assume that we have fixed an orientation on $F_q$ in advance so that $L_{\underline V}$ are equipped with a standard orientation. 
\end{remark}
We propose the following definition for a geometric Lagrangian lift of a tropical subvariety (which is similar to that proposed in \cite[Definition 2.1]{mikhalkin2018examples}).
\begin{definition}
    A family of oriented Lagrangian submanifolds $L_V^\eps$ for $\eps>0$ is a \emph{geometric} Lagrangian lift of an weight-1 polyhedral complex $V\subset Q$ if the following conditions hold:
    \begin{enumerate}[label=(\roman*)]
        \item The Lagrangians $L_V^\eps$ are all Hamiltonian isotopic,\label{item:hamiltonianiso}
        \item \label{item:conormalcondition} Let  $V^{(i)}$ be the collection of codimension $i$ strata of $V$. We require that away from the codimension 1 strata,
        \begin{equation}L^\eps_V \setminus \syza^{-1}(B_\eps(V^{(1)}))= L_{\underline V^{(0)}}\setminus \syza^{-1}(B_\eps (V^{(1)}))\label{eq:conormal}\end{equation}
        as oriented submanifolds.
        \item The Lagrangians $L_V^\eps$ are embedded, graded, spin, and admissible (in the sense of \cref{def:admissiblitycondition}).
    \end{enumerate}
    \label{def:geometricLagrangianlift}
\end{definition}
\begin{remark}
\Cref{def:geometricLagrangianlift} has two simplifying requirements; one is included due to current technical limitations in the definition of Floer cohomology, and the second is for convenience.

The requirement that $L_V^\epsilon$ is embedded is a technically needed assumption; we believe that this condition can be dropped without modifying the main results of this paper. Our reason for restricting ourselves to the embedded setting is that the Charest-Woodward pearly model as written does not include a description of Floer cohomology for immersed Lagrangian submanifolds.

While \cref{def:geometricLagrangianlift} looks only at weight-1 polyhedral complexes, one can extend the story to weighted polyhedral complexes by asking that at each top dimensional stratum $\underline V\subset \underline V^{(0)}$ with weight $m$, the realization $L_{\underline V}$ is $m$-disjoint copies of $N^*\underline V/N^*_\ZZ\underline V$. All results in this paper can be extended to the weighted setting.
\end{remark}

The constructions from \cite{matessi2018lagrangian,mikhalkin2018examples,mak2020tropically,hicks2019tropical} all satisfy (\cref{def:geometricLagrangianlift}:\cref{item:hamiltonianiso,item:conormalcondition} ).  To prove that the previous definitions give examples of geometric Lagrangian lifts, we need to additionally show that they are admissible, graded, and spin. 
We prove these properties for certain examples of Lagrangian lifts in \cref{subsec:admissible,subsec:graded,subsec:spin}.

While \cref{def:geometricLagrangianlift} only asks that we take the lift of a weight-1 polyhedral complex, the only polyhedral complexes which admit such lifts are tropical ones.
\begin{prop}
    Let $V$ be a weight-1 rational polyhedral complex, and suppose that it has a Lagrangian lift $L_V^\eps$ satisfying (\cref{def:geometricLagrangianlift}:\cref{item:hamiltonianiso,item:conormalcondition}). Then $V$ is a tropical subvariety. 
    \label{prop:lagimpliesbalanced}
\end{prop}
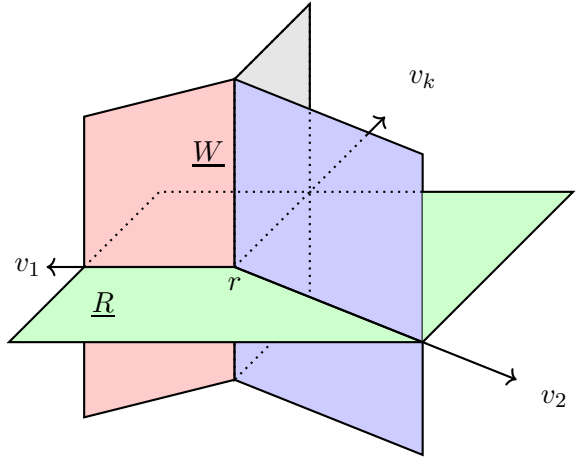
\begin{figure}
    \centering
    \begin{tikzpicture}
    \draw[fill=gray!20] (-0.5,2.5) -- (0.5,3.5) -- (0.5,-0.5) -- (-0.5,-1.5) -- cycle;

    \draw[fill=red!20] (-0.5,2.5) -- (-0.5,-1.5) -- (-2.5,-2) -- (-2.5,2) -- cycle;
    \draw[fill=blue!20] (-0.5,2.5) -- (2,1.5) -- (2,-2.5) -- (-0.5,-1.5) -- cycle;
    \draw[dotted] (-0.5,2.5) -- (0.5,3.5) -- (0.5,-0.5) -- (-0.5,-1.5) -- cycle;
    \draw[fill=green!20] (-2.5,0) -- (-3.5,-1) -- (2,-1) -- (-0.5,0) -- (-2.5,0);
    \draw[dotted] (-2.5,0) -- (-1.5,1) -- (0,1);
    \draw[fill=green!20] (2,1) -- (4,1) -- (2,-1);
    \draw[dotted] (0,1) -- (2,1);
    \draw[thick, ->] (-0.5,0) -- (-3,0);
    \draw[thick, ->] (-0.5,0) -- (3.25,-1.5);
    \draw[thick, dotted] (-0.5,0) -- (1.25,1.75);
    \draw[thick, ->] (1.25,1.75) -- (1.5,2);
    \node[below] at (-0.5,0) {$r$};
    \node[left] at (-0.5,1.5) {$\underline W$};
    
    \node at (-2.25,-0.5) {$\underline R$};
    \node at (-3.25,0) {$v_1$};
    \node at (3.75,-1.75) {$v_2$};
    \node at (2,2.5) {$v_k$};
    \end{tikzpicture}     \caption{Polyhedral complexes discussed in \cref{prop:lagimpliesbalanced}}.
    \label{fig:lagimpliesbalanced}
\end{figure}
\begin{proof}
        Select an interior point $r\in \underline W\subset V^{(1)}$ of the codimension 1 stratum of $V$.
    Pick $U_r\subset T_r Q$ a rational subspace so that $ U_r\oplus T_rW = T_rQ$.
    Let $R\subset Q$ be a small polyhedral domain passing through $r$ with tangent space $U_r$.
    Then $V|_{R}$ is a weight-1 rational polyhedral curve. By taking $R$ small enough, $V|_{R}$ has a single vertex and edges pointing in directions $v_1, \ldots v_k$ corresponding to facets $F_1, \ldots, F_k$ containing $W$.
    We need to prove that $\sum_{i=1}^k v_i=0$. See \cref{fig:lagimpliesbalanced}.

    Consider the symplectic manifold $Y_A:= T^*R/T^*_\ZZ R\subset T^*Q/T^*_\ZZ Q$, with the Lagrangian torus fibration $\pi_{Y_A}: Y_A\to R$. Let $i: R\to Q$ be the inclusion.
    Select $\eps$ small enough so that $B^\eps(W)\cap R$ is an interior set of $R$.
    Given a Lagrangian submanifold $L\subset T^*Q/T^*_\ZZ Q$, we can take a Hamiltonian perturbation of $L$ so that  
    \[L_{i^*}\circ L:=\{(r, i^*(p))\st (r, p)\in L, r\in R\}\]
    is a Lagrangian submanifold of $Y_A$.
        See \cite[Section 5.2]{hanlon2020functoriality} for a more general discussion of this construction from the perspective of Lagrangian correspondences.
    By definition $\pi_{Y_A}(L_{i^*}\circ L) = \pi_{Y_A}(L)\cap R$, so $L_{V|_R}^\eps:=L_{i^*}\circ L_V^\eps$ is a geometric realization of $V|_R\subset R$. 
    We, therefore, have reduced to the setting which is the lift of a tropical curve with a single vertex.
    
    Given a tropical curve $V|_R\subset R$ with a single vertex $v$, the Lagrangian $L_{V|_R}^\eps$ is a manifold with boundary. Consider the projection $\arg_R: Y_A\to F_r= T^*_rR/T_{\ZZ,r}^* R$. Considering $\arg_R(L_{V|_R}^\eps)$ as a $(\dim(F_r)-1)$ chain, we obtain the relation in homology
    \[0=[\arg_R(\partial(L_{V|_R}^\eps))]\in H_{\dim_{F_r}-1}(F_r).\] 
    There is an identification (as vector spaces) that sends an integral basis $e_1, \ldots, e_n$ to the class of the perpendicular subtorus
    \begin{align*}
        T_rR\to& H_{\dim_{F_r}-1}(F_r)\\
        e_i\mapsto& [\{\eta\in T^*_rR\st \eta(e_i)=0\}]
    \end{align*}
        The boundary of $L_{V|_R}^\eps$ lies in the region where \cref{eq:conormal} holds and we have an agreement of oriented submanifolds, we can therefore compute:
    \begin{align*}
        [\arg_R(\partial (L_{V|_R}^\eps))]=\sum_{i=1}^k [\{\eta\in T^*_rR\st \eta(e_i)=0\}]
    \end{align*}
        proving that $\sum e_i=0$. 
\end{proof}
\begin{notation}
    From here on, we will drop the $\eps$ in $L_V^\eps$ and simply write $L_V$ for a Lagrangian which belongs to such a family.
\end{notation}
\subsection{Geometric Lagrangian lift: admissibility}
\label{subsec:admissible}
When Lagrangian submanifolds are non-compact, we need to place taming conditions on them so that they are Floer-theoretically well-behaved. 
\begin{definition}[\cite{hanlon2019monodromy}]
    Let $W_\Sigma:X_A\to \CC$ be a Laurent polynomial whose monomials are indexed $A$, the set of rays of a fan $\Sigma$. A \emph{monomial division} $\Delta_\Sigma$ for $W_\Sigma=\sum_{\alpha\in A} c_\alpha z^\alpha$ is an assignment of a closed set $U_\alpha \subset Q$ to each monomial $\alpha\in A$ so that the following conditions hold:
    \begin{itemize}
        \item
              The sets $U_\alpha$ cover the complement of a compact subset of $Q=\RR^n$;
        \item
              There exist constants $k_\alpha \in \RR_{>0}$ so that for all $z$ with $\val(z)\in U_\alpha$ the expression 
              \[
                  \max_{\alpha\in A} (|c_\alpha z^\alpha|^{k_\alpha})
              \]
              is always achieved by $|c_\alpha z^\alpha|^{k_\alpha}$; and
        \item
              $U_\alpha$ is a subset of the open star of the ray $\alpha$ in the fan $\Sigma$.
    \end{itemize}
    A Lagrangian $L\subset X_A$ is \emph{$\Delta_\Sigma$-monomially admissible} if over $\syza^{-1}(U_\alpha)$ the argument of $c_\alpha z^\alpha$ restricted to $L$ is zero outside of a compact set. \label{def:admissiblitycondition}
\end{definition}
We will always assume that the $\arg(c_\alpha)=0$.
An advantage of using the monomial admissibility condition for Lagrangian submanifolds is that it is a relatively simple check to see if a Lagrangian submanifold satisfies the condition.
\begin{theorem*}[Theorem 3.1.7 of \cite{hicks2020tropical} ]
    Suppose that $V$ is the tropicalization of a hypersurface whose Newton polytope has dual fan $\Sigma$. Then the construction of  $L_V$ from \cite{hicks2019tropical} is $\Delta_\Sigma$-monomially admissible.
\end{theorem*}

    Let $V\subset Q$ be a tropical curve. We say that $V$ is adapted to $\Sigma$ if each semi-infinite edge of $V$ points in the direction of a ray of $\Sigma$.
\begin{claim}
    Suppose that $V\subset \RR^n$ is a weight-1  tropical curve adapted to $\Sigma$.  Any  Lagrangian lift $L_V$ is $\Delta_\Sigma$-monomially admissible. 
\end{claim}
\begin{proof}
    Let $V^{(0)}_\infty=\{e_i\}_{i=1}^k$ denote the semi-infinite edges of $V$. We note that there exists a compact set $K\subset Q$ so that $L_C\setminus \syza^{-1}(K)=\bigsqcup_{e\in V^{(0)}_\infty} L_{\underline e}\setminus \syza^{-1}(K)$. Furthermore, $K$ can be chosen so that $e\setminus K\subset U_\alpha$ if and only if $e$ points in the direction $\alpha\in \Sigma$.
    Over this region, we observe that $\arg(z^\alpha)|_{N^*e/N^*_\ZZ e}=0$. 
\end{proof}
\begin{remark}
    If some of the semi-infinite edges of $V$ are weighted, we must replace the last condition in monomially admissible with ``there exists a discrete set of values $\{\theta_i\}$ such that the argument of $c_\alpha z^\alpha|_{L\cap C)\alpha}\subset \{\theta_i\}$''. 
        The Floer theoretic arguments in \cite{hanlon2019monodromy} can be applied to this setting as well (simply by letting $\theta_i$ be $k$-roots of unity, and replacing $\alpha$ with $k\alpha$). 
\end{remark}
\subsection{Geometric Lagrangian lifts: homologically minimal and graded}
\label{subsec:graded}
The additional amount of flexibility that symplectic geometry affords us means that there are many geometric Lagrangian lifts of a single tropical subvariety. Some of these lifts differ for unimportant reasons: for instance, we could have included some extra topology in our Lagrangian by attaching a Lagrangian with vanishing Floer cohomology to a previously constructed lift. The following condition is imposed to weed out some of these worst offenders.
\begin{definition}
    Let $j: L_{V^{(0)}}\setminus \syza^{-1}(B_\epsilon(V^{(1)}))\into L_V$ be the inclusion that is induced from the inclusion of the codimension zero strata of $V$ into $V$.
             We say that a lifting is \emph{ homologically minimal} if there exists a section $i: V\to L_V\subset X_A$ so that  $H_1(L_V)$ is generated by the images of 
    \begin{align*}
        (i)_*:&H_1(V)\to H_1(L_V)\\
        (j)_*:&H_1(L_{V^{(0)}}\setminus \syza^{-1} B_\epsilon(V^{(1)}))\to H_1(L_V)
    \end{align*}
        Let $i_{L_V}: L_V\to X_A$ be the inclusion of our Lagrangian submanifold. We say that $L_V$ is an untwisted realization of $V$ if the composition 
    \[V \xrightarrow{(i_{L_V}\circ i)} X_A \xrightarrow{\arg} F_q 
    \]
        is null-homologous (for any choice of $q\in Q$).
    \end{definition}

\begin{remark}
For a fixed tropical subvariety $V$, there can be several geometric Lagrangian lifts of $V$ which are meaningfully different. We expand on how these different choices of lifts correspond to tropical line bundles of $V$ in \cref{app:jacobian}.
\end{remark}
The homologically minimal condition places some constraints on our Lagrangian submanifolds.
\begin{lemma}
    If $L_V$ is homologically minimal and untwisted, then $L_V$ is graded.
\end{lemma}
\begin{proof}
    We recall the definition of graded from \cite[Example 2.9]{seidel2000graded}.
     Since $c_1(X_A)=0$, we can take a section $\bigwedge_{i=1}^n \left(dq_i+id\theta_i\right)^{\tensor 2}$ of $\Lambda^n(TX_A, J)^{\tensor 2}$.
    This determines a map 
    \begin{align*}
        {\det}^2\circ s_L: L\to& S^1\\
        x\mapsto& \left(\bigwedge \left(dq_i+id\theta_i\right)(T_xL)\right)^2
    \end{align*}
    A Lagrangian is $\ZZ$ graded if this map can be lifted to $\RR$. 

    Consider a homologically minimal Lagrangian and untwisted Lagrangian $L_V$. There exist generators $\{[\alpha_k],[\beta_l]\}$ for $H_1(L_V)$ so that $\alpha_k$ is in the image of $i$ and $\beta_l$ are in the image of $j$. Since the compositions
    \begin{align*}
        {\det}^2\circ s_{L_V}\circ i: V\to& S^1\\
        {\det}^2\circ s_{L_V}\circ j: (L_{V^{(0)}})\to & S^1
    \end{align*}
    are constantly $0$, it follows that there is no obstruction to lifting ${\det}^2\circ s_{L_V}: L_V\to S^1$ to $\RR$.
\end{proof}

\begin{prop}
    Suppose that $V\subset \RR^n$ is either a smooth tropical curve or a smooth tropical hypersurface. Then the construction of $L_V$ given by \cite{matessi2018lagrangian,mikhalkin2018examples,hicks2019tropical} produces a homologically minimal Lagrangian lift $L_V$. The lifts are therefore graded.
    \end{prop}
\begin{proof}
    In the cases of tropical curves, this follows from computing the homology of $L_V$ from a cover given by $L_{\str(v)}$. For hypersurfaces, this is proven in \cite[Proposition 3.18]{hicks2020tropical}.
\end{proof}
Unless otherwise specified, the lift of a smooth tropical curve or hypersurface will always be the one given by \cite{matessi2018lagrangian,mikhalkin2018examples,hicks2019tropical} .

\subsection{Geometric Lagrangian lifts: spin}
\label{subsec:spin}
We start with a lemma on the topology of lifts of smooth genus zero tropical curves.
\begin{lemma}
    Let $V\subset \RR^n$ be a smooth genus 0 tropical curve. 
    \begin{enumerate}[label=(\roman*)]
        \item For any semi-infinite edge $f\in V_\infty^{(0)}$, the restriction map  $\res^V_f:H^1(L_V)\to H^1(L_f)$ is a surjection.\label{item:surjection}
        \item For any semi-infinite edge $f$, the restriction map $\res^V_{V_\infty\setminus f}:H^2(L_V)\to \bigoplus_{g\neq f} H^2(L_g)$ is an injection. \label{item:injection}
    \end{enumerate}
    \label{lem:tropcurvetopology}
\end{lemma}
\begin{proof}
    We prove \cref{item:surjection,item:injection} by induction on the number of vertices in $V$. 
    
    \emph{Base Case:} Suppose that $V$ has one vertex. Then $V$ is planar, and there exists a splitting of $(\CC^*)^n= (\CC^*)^2\times (\CC^*)^{n-2}$ so that $L_V= L_{\pants}\times T^{n-2}$, where $L_{\pants}\subset (\CC^*)^2$ is the standard pair of pants. The boundary of the pair of pants is $S^1_{e_1}\cup S^1_{e_2}\cup S^1_{e_3}$, where $e_1, e_2, e_3$ label the three edges of the pair of pants. A direct computation shows that 
    \begin{align*}
        H^0(L_{\pants})\to H^0(S^1_{e_1}) && H^1(L_{\pants})\to H^1(S^1_{e_1})
    \end{align*}
    surjects, and that 
    \begin{align*}
        H^0(L_{\pants})\to H^0(S^1_{e_1}\cup S^1_{e_2}) && H^1(L_{\pants})\to H^1(S^1_{e_1}\cup S^1_{e_2})
    \end{align*}
    inject. An application of K\"unneth formula gives \cref{item:surjection,item:injection} for $L_V$.

    \emph{Inductive Step:}
    Let $f\in V_\infty^{(0)}$ be any semi-infinite edge and let $v$ be the vertex of $V$ belonging to that edge. Let $W$ be the tropical curve given by vertices not equal to $v$, so that $L_{\str(v)}, L_W$ cover  $L_V$ with intersection $L_{\str(v)}\cap L_W= L_{e}=T^{n-1}\times e$, as in \cref{fig:covering}. This can be done because $V$ is a tree.
    \begin{figure}
        \centering
        \begin{tikzpicture}[rotate=90]

    \draw[line width = 2pt, blue] (-2,1) -- (-1,1) -- (-1,-0.5) (0,2) -- (-1,1);
    \draw[thick, red] (-1,0.5) -- (-1,-0.5) -- (-1,-1) -- (-2,-1) (-1,-1) -- (0,-2) -- (0,-2.5) (0,-2) -- (1,-2);
    \node[above] at (-0.5,1.5) {$f$};
    \node[above] at (-1,0) {$e$};
    \node[above left] at (-0.5,-1.5) {$W$};
    \node[below left] at (-1,1) {$star(v)$};
    \end{tikzpicture}         \caption{Covering our tropical curve $V$ with two charts: $W$ and a pair of pants $\str(v)$ centered at $v$.}
        \label{fig:covering}
    \end{figure}
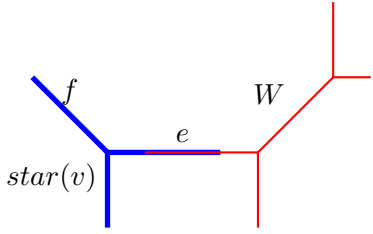

    We first prove \cref{item:surjection}. We use $L_{\str(v)}, L_W$ to compute the first cohomology of $L_V$ using Mayer-Vietoris, and show that the red arrow in the diagram below is a surjection. 
    \[
        \begin{tikzpicture}
            \node (v4) at (2,1.5) {$H^1(L_{\str(v)})\oplus H^1(L_W)$};
            \node (v3) at (-4.5,1.5) {$H^1(L_V)$};
            \node (v7) at (2,-1) {$ H^1(L_f) \oplus  0$};

            \node (v2) at (-4.5,-1) {$H^1(L_f) $};
            \draw  (v3) edge[->] node[above]{$\res^V_{\str(v)}\oplus \res^V_W$} (v4);
            \draw  (v2) edge[double distance=2pt] (v7);
            \draw  (v4) edge[->>] node[fill=white]{$\res^{\str(v)}_f\oplus 0$} (v7);
            \draw  (v3) edge[red, ->] (v2);

            \node (v1) at (7.5,1.5) {$H^1(L_e)$};
            
            \draw  (v4) edge[->] node[above]{$\text{res}^{\str(v)}_e - \res^W_e$} (v1);
         \end{tikzpicture}
     \]
    From the base case: given $\alpha\in H^1(L_f)$, there exists $\alpha'\in H^1(L_{\str(v))})$ with $\res^{\str(v)}_f(\alpha')=\alpha$. From the induction hypothesis, there exists $\beta'\in H^1(L_W)$ with $\res^W_e(\beta')=\res^{\str(v)}_e(\alpha')$. Therefore $(\alpha', \beta')\in \ker(\res^{\str(v)}_e-\res^W_e)$, and by exactness of the rows is in the image of $\res^V_{\str(v)}\oplus \res^V_W$. Let $\alpha''$ be in the preimage of $(\alpha', \beta')$. By commutativity of the below diagram, we conclude $\res^V_f(\alpha'')=\alpha$.

We now prove \cref{item:injection}. We compute  $H^2(L_V)$ using Mayer-Vietoris, and show that the blue arrow below is injective. By \cref{item:surjection}, the leftmost arrow is surjective. Using the exactness of the sequence, we conclude that $\res^V_{\str(v)}\oplus \res^V_W$ is injective on the second cohomology groups.
    \[ 
        \begin{tikzpicture}
            \node (v6) at (-9,1.5) {$H^1(L_{\str(v)})\oplus H^1(L_W)$};
            \node (v4) at (2,1.5) {$H^2(L_{\str(v)})\oplus H^2(L_W)$};
            \node (v3) at (-2.5,1.5) {$H^2(L_V)$};
            \node (v10) at (-5,1.5) {$H^1(L_e)$};
            \node (v7) at (2,-1) {$\displaystyle \bigoplus_{\substack{g\in \str(v)_\infty^{(0)}\\ g\neq e,f}}H^2(L_g)\oplus  \bigoplus_{\substack{g\in W_\infty^{(0)}\\ g\neq e}}H^2(L_g)$};

            \node (v2) at (-2.5,-1) {$\displaystyle \bigoplus_{\substack{g\in V_\infty^{(0)}\\ g\neq f}}H^2(L_g)$};
            \draw  (v3) edge[right hook->] node[above]{$\res^V_{\str(v)}\oplus \res^V_W$}(v4);
            \draw  (v2) edge[right hook->] (v7);
            \draw  (v10) edge[->] node[fill=white]{$0$} (v3);
            \draw  (v4) edge[->] node[fill=white]{$C\oplus D$} (v7);
            \draw  (v3) edge[blue, ->] node[fill=white]{$\bigoplus_{g\neq f}\res^V_g$}  (v2);

            \draw  (v6) edge[->>] (v10);
        \end{tikzpicture}\]
        Let $C= \bigoplus_{\substack{g\in \str(v)_\infty^{(0)}\\ g\neq e,f}} \res^{\str(v)}_g$ and $D= \bigoplus_{\substack{g\in W_\infty^{(0)}\\ g\neq e}} \res^{W}_g$.
        Consider now a class $\alpha\in H^2(L_V)$. Suppose that $\bigoplus_{g\neq f}\res^V_g(\alpha)=0$. We will show that $\alpha=0$. 
        By commutativity of the diagram, $(C\oplus D)\circ (\res^V_{\str(v)}\oplus \res^V_W)=0$. Because $D$ is injective and $(\res^V_{\str(v)}\oplus \res^V_W)$ is injective, we obtain that $C\circ \res^V_{\str(v)}(\alpha)=0$, and $\res^V_W(\alpha)=0$. 
                We now break into two cases:
        \begin{itemize}
            \item Case I: $\res^V_{\str(v)}(\alpha)=0$. This implies $(\res^V_{\str(v)}\oplus \res^V_W)(\alpha)=0$, which by injectivity of $\res^V_{\str(v)}\oplus \res^V_W$ tells us that $\alpha=0$.
            \item Case II: $\res^V_{\str(v)}(\alpha)\neq 0$. Observe then that  $(C\oplus \res^{\str(v)}_e)\circ \res^V_{\str(v)}$ is injective from the base case, so $\res^{\str(v)}_e\circ \res^V_{\str(v)}(\alpha)\neq 0$. Since $\res^V_W(\alpha)=0$, we obtain that 
        \[(\res^{\str(v)}_e\oplus \res^W_e)\circ (\res^V_{\str(v)}\oplus \res^V_W)(\alpha)\neq 0.\]
        This violates the exactness of the top row, so case II cannot occur.
        \end{itemize}
\end{proof}

\begin{prop}
    In the setting where $V\subset \RR^n$ has genus 0 , the constructions of \cite{matessi2018lagrangian,mikhalkin2018examples,hicks2019tropical}  give homologically minimal untwisted geometric Lagrangian lifts $L_V$ of $V$. 
\end{prop}
\begin{proof}
    We prove that this Lagrangian submanifold is homologically minimal because the homology of the pair of pants is generated by the homology of the legs. If $n=2$, then $L_V$ is a surface, and therefore spin.

    To prove that the $n\geq 3$ cases are spin, we induct on the number of vertices in $V$. For the 1-vertex case, $L_{\str(v)}\simeq L_{\pants}\times T^{n-2}$. The manifolds $L_{\pants, v}\times T^{n-2}$ have trivializations given by embedding $L_{\pants, v}$ into $\RR^2$, and is therefore spin.

    As in the proof of \cref{lem:tropcurvetopology}, write $V=L_W\cup L_{\str(v)}$, where $e$ is the common edge $L_W\cap L_{\str(v)}$ By the induction hypothesis a spin structure on $L_W$. By pullback, this gives a spin structure over $L_e$.
    Since $H^1(L_{\str(v)}, \ZZ/2\ZZ)\to H^1(L_{e})$ surjects, there is no obstruction to picking a spin structure on  $L_{\str(v)}$ agree with the prescribed spin structure on $L_e$.
\end{proof}
This method of proof can be extended to a slightly larger set of examples. We say that a smooth tropical curve $V$ has planar genus if there exists cycles $c_1, \ldots c_k\subset V$ so that $\{[c_1], \ldots, [c_k]\}$ generate $H_1(V)$, and there exist 2-dimensional planes $\underline V_k\subset \RR^n$ so that $c_i\subset \underline V_k$. 
\begin{corollary}
    If $V\subset \RR^n$ is a smooth tropical curve $V$ with planar genus, then $L_V$ is spin. 
    \end{corollary}
The other setting where tropical Lagrangian lifts have been studied is the setting of hypersurfaces. 
\begin{lemma}
    If $V\subset \RR^n$ is a smooth tropical hypersurface, the construction of \cite{matessi2018lagrangian,hicks2020tropical} of $L_V$ is spin.
\end{lemma}
\begin{proof}
    We break into several cases. 
    \begin{itemize}
        \item If $n=2$, then $L_V$ is a punctured surface (and therefore spin).
        \item If $n=3$, then $L_V$ is an orientable 3-manifold (and therefore spin).
        \item If $n\geq 4$, then by\cite{hicks2020tropical} the Lagrangian $L_V$ is the connect sum of two copies of $\RR^n$ at several contractible regions $U_\alpha$ indexed by $\Delta$, the Newton polytope of the defining tropical polynomial for $V$. Assume that $\dim(\Delta)=n\geq 4$ (as otherwise, we may reduce to one of the previous cases).
        Following \cite[Proposition 3.18]{hicks2020tropical}, we take two charts $L_r, L_s\simeq \RR^n\setminus \bigcup_{\alpha\in \Delta}U_\alpha$ so that $L_V=L_r\cup L_s$. The $L_r, L_s$ are homotopic to $V$. The overlap $L_r\cap L_s\simeq \bigcup_{\alpha\in \Delta}\partial U_\alpha$, where each $\partial U_\alpha$ is homotopic to either $S^{n-1}$ or $D^{n-1}$. By Mayer-Vietoris, we compute
        \[ \bigoplus_{\alpha\in \Delta} H^1(\partial U_\alpha)\to H^2(L_V, \ZZ/2\ZZ)\to H^2(L_r, \ZZ/2\ZZ)\oplus H^2(L_r \ZZ/2\ZZ) .\]
        The left and right terms are zero when $n\geq 4$, so $H^2(L_V, \ZZ/2\ZZ)=0$ and our Lagrangian is spin.
    \end{itemize}
\end{proof}

  \section{Unobstructed Lagrangian lifts of tropical subvarieties}
	\label{sec:unobstructedness}
	\label{sec:obstruction}
Since the geometric Lagrangian lifts $L_V$ we construct will not be exact, to obtain a Lagrangian Floer \emph{cohomology} theory we need to show that these Lagrangian submanifolds have $\Lambda$-filtered $A_\infty$ algebra which can be unobstructed.

\subsection{Pearly model for Floer cohomology}
 We will adopt the model employed in  \cite{charest2019floer} to define $\CF(L)$. 
\begin{theorem*}[\cite{charest2019floer}]
    Let $L\subset X$ be a compact relative spin and graded Lagrangian submanifold inside a rational compact symplectic manifold $X$. Pick $h: L\to \RR$ a Morse function, and $D\subset X\setminus L$ a stabilizing divisor. There exists a choice of perturbation datum $\mathcal P$ which defines a filtered $A_\infty$ algebra $\CF(L, h, \mathcal P, D)$ whose 
    \begin{itemize}
        \item Chains are given by the Morse cochains of $L$, so that $\CF(L, h, \mathcal P, D)= \Lambda \langle \Crit(h)\rangle$. 
        \item Product structures come from counting configurations of treed disks. More precisely, given a collection of critical points $\underline x= (x_1,  \ldots, x_k)$, we define the structure coefficients
        \[\langle m^{k}(x_1\tensor \cdots \tensor x_k), x_0 \rangle = \sum_{\beta\in H_2(X, L)} (-1)^\heartsuit (\sigma(u)!)^{-1}T^{\omega(\beta)} \cdot \# \mathcal M_{\mathcal P}(X, L, D, \underline x, \beta)\]
        which determine the $A_\infty$ product structure.
        Here, $\# \mathcal M_{\mathcal P}(X, L, D, \underline x, \beta)$ is the count of points in the moduli space of $\mathcal P$-perturbed pseudoholomorphic treed disks, $\sigma(u)$ denotes the number of stabilizing points on each of these treed disks, and $\heartsuit=\sum_{i=1}^k i |x_i|$.
    \end{itemize}
    The $\Lambda$-filtered $A_\infty$ homotopy class does not depend on the choices of perturbation, divisor, and Morse function taken in the construction.
\end{theorem*}
When the choice of $h, \mathcal P$ and $D$ are unimportant, we will write $\CF(L)$ instead of $\CF(L, h, \mathcal P, D)$.
The most visible difference between the tautologically unobstructed setting and this more general definition is that there now exists a curvature term $m^0: \Lambda\to \CF(L)$, which obstructs the squaring of the differential to zero. We say that $L$ is \emph{unobstructed} if $\CF(L)$ has a bounding cochain $b\in \CF(L)$ (\cref{subsub:boundingcochains}). When $L$ is unobstructed, the deformed $A_\infty$ structure on $\CF(L, b)$ is a chain complex.

In this section, we discuss whether a geometric Lagrangian lift $L_V$ of a tropical subvariety is an unobstructed Lagrangian submanifold.
\label{subsec:pearlymodel}
We give an example computation in the pearly disk model to fix notation. 
\begin{example}[Running Example, Continued]
    We return to \cref{exam:localsystems}. First, we examine $\CF(F_q, \del_1)$. Since $F_q$ bounds no topological disks, it does not bound any holomorphic disks. 
        Therefore the Floer complex is the Morse-tree algebra of $F_q$. Give $F_q$ the Morse function 
        \begin{equation}
        f=\sum_{i=1}^n \cos(\theta_i)\label{eq:MorsePrimitive}
    \end{equation} 
    We label the generators of \[\CF(F_q, \nabla_1)=\Lambda\langle y^1_I\rangle\]
    where the $I\subset \{1, \ldots n\}$. The differential is given by $m^1(y^1_I)=0$, and for a particular set of perturbations the product structure is 
    \[
        m^{2}(y^1_I\tensor y^1_J)=\left\{\begin{array}{cc} \pm y^1_{I\cup J} & \text{if $I\cap J=\emptyset$}\\ 0 & \text{otherwise}\end{array}\right.
    \]
    where the sign is determined by the number of transpositions required to reorder $I\cup J$.
    \label{exam:morsemodel}
\end{example}
\begin{remark}
    To our knowledge, it is unknown if there exists a perturbation scheme for Morse flow trees so that all higher products $m^k: \CM(S^1)^{\tensor k}\to \CM(S^1)[2-k]$ vanish. 
\end{remark}
To work in the setting where $X_A$ is non-compact, we need to place restrictions on the non-compact behavior of the Lagrangian $L$ to ensure that the moduli spaces of pseudoholomorphic treed-disks considered by \cite{charest2019floer} remain compact. A natural condition to impose is that $L\subset X_A$ is admissible (\cref{def:bottleneck}) with respect to a potential function $W: X_A\to \CC$, so that the projection $W(L)$ fibers over the real axis $\RR_{>0}$ outside of a compact set. Denote by $Y_A=W^{-1}(t)$ for $t\in \RR_{\gg 0}$. Choices of different sufficiently large $t$ yield fibers which are symplectomorphic. The restriction of $L$ to a large fiber will be called $M:= L|_{Y_A}$; this is a Lagrangian submanifold of $Y_A$. By \cref{thm:sympfibrationexactsequence} there exists a treed-disk model for Lagrangian Floer cohomology $\CF(L)$ for $W$-admissible Lagrangians $L$.
Furthermore, there exist compatible choices of perturbation datum so the standard projection
\[\CF(L)\to \CF(M)\]
is a $\Lambda$-filtered map of $A_\infty$ algebras.

A useful lemma of \cite{hanlon2019monodromy} states that when we have a monomially admissible Lagrangian $L$, there exists a potential function $W$ so that $L$ is $W$-admissible.
From the data of a fan and $t\in \RR$, \cite{abouzaid2009morse} constructs \emph{tropicalized potential}, which is a symplectic fibration $W^{t, 1}_\Sigma: (\CC^*)^n\to \CC$ outside of a compact set. 
\begin{lemma}[Section 4.4 of \cite{hanlon2019monodromy}, Remark 2.10 \cite{hanlon2020functoriality}]
    Suppose that $L$ is a Lagrangian submanifold that is monomially admissible with respect to a monomial division adapted to $\Sigma$ (in the sense of \cref{def:admissiblitycondition}). Then $L$ can be made admissible for the tropicalized potential.
    \label{lemma:equivadmissibility}
\end{lemma}

\subsection{Geometric Lagrangians versus Lagrangian branes}
\label{subsec:knownunobstructed}
\begin{definition}
    \label{def:Lagrangianbranelift}
    We say that an unobstructed Lagrangian submanifold $(L_V,b)$ is a \emph{Lagrangian brane lift} of $V$ if $L_V$ is a geometric Lagrangian lift of $V$.
\end{definition}
Before developing constructions of bounding cochains for geometric Lagrangian lifts, we give some examples of geometric Lagrangian lifts which are known to be unobstructed (or tautologically unobstructed) Lagrangian submanifolds. 
\begin{example}[Lagrangian pair of pants]
    In \cite{matessi2018lagrangian}, it was shown that the tropical pair of pants centered at the origin was an exact Lagrangian submanifold; a similar proof was given in \cite{hicks2021tropical} which showed that all tropical Lagrangian submanifolds constructed from the data of a dimer are exact. 
    
\begin{claim}
    Let $V\subset \RR^n$ be a tropical variety so that $0\in \underline V_i$ for all facets $\underline V_i \subset V$. Let $L_V^\eps$ be a homologically minimal lift of $V$. Then $L_V^\eps$ is exact. 
    \end{claim}
\begin{proof}
    Let $\eta=pdq$ be the primitive for $\omega$ on $X_A=T^*T^n$.
    We need to show that $\eta$ is exact on $L_V^\eps$; equivalently show that $\eta(\gamma)=0$ for all $[\gamma]\in H_1(L_V^\eps)$. 
    Observe that $L_V^\eps$ retracts onto $L_V^\eps \cap \syza^{-1}(B_\epsilon(0))$. Therefore, for every loop $\gamma\in H_1(L_V^\eps)$, there exists $\gamma'$ which is homotopic to $\gamma$ and lives within in $\syza^{-1}(B_\epsilon(0))$; by letting $\epsilon\to 0$ we obtain $[\gamma']=[\gamma]$ and $\gamma'\subset F_0$. 
        As $F_0$ is exact, $\eta(\gamma')=0$.
\end{proof}
    Since these Lagrangians are exact, they are tautologically unobstructed and we can conclude that $L_V^\eps$ is a tropical Lagrangian brane.
\end{example}
In some cases, one obtains tautological unobstructedness (or unobstructedness) of the Lagrangian submanifold for free. 
\begin{example}
    Curves $V\subset\RR^2$ provide an example of where we can obtain tautological unobstructedness. Since $L_V$ is a graded Lagrangian submanifold, the only holomorphic curves which might cause us difficulty are Maslov index 0 curves. However, the expected dimension of Maslov index 0 disks with boundary on a 2-dimensional Lagrangian is negative, therefore for a generic choice of almost complex structure these disks disappear and $L_V^\eps$ is tautologically unobstructed.

    It is possible for non-regular Maslov index 0 disks to appear with boundary on $L_V^\eps$, even in simple examples (see \cref{exam:feynman}). More generally, \cite{hicks2021tropical} shows that there exists a ``wall-crossing'' phenomenon which occurs for isotopies between tropical Lagrangian submanifolds, and that the count of these Maslov index 0 holomorphic disks play a crucial role in understanding coordinates on the moduli space of tropical Lagrangian submanifolds.
\end{example}
\begin{example}
    We now examine a setting outside of the mirrors to toric varieties. Let $Q$ be any tropical abelian surface; then $X_A:=T^*Q/T_\ZZ Q$ is a symplectic 4-torus. Given any tropical curve $V\subset Q$, there is a Lagrangian surface $L_V^\eps\subset X_A$. By the same reasoning as above, $L_V$ is tautologically unobstructed for generic choice of almost complex structure. 
    \label{exam:abeliansurface}
\end{example}
\begin{example}
    Another example where we know unobstructedness for geometric Lagrangian lifts is \cite{mak2020tropically}. In that setting, the base of the Lagrangian torus fibration has non-trivial discriminant locus, and the tropical Lagrangians constructed are lifts of compact genus-0 tropical curves in the base. \Citeauthor{mak2020tropically} show that the associated tropical Lagrangians are homology spheres and therefore are always unobstructed by a choice of bounding cochain (\cite{fukaya2010lagrangian}).
\end{example}
In general, other techniques are required to prove that a geometric Lagrangian lift of a tropical subvariety is unobstructed.
\begin{example}
    Given any smooth tropical hypersurface $V\subset \RR^n$, \cite{hicks2019tropical} shows that the tropical Lagrangian lift can be equipped with a bounding cochain so that $(L_V,b)$ is an unobstructed Lagrangian submanifold of $(\CC^*)^n$.
    The proof uses that $L_V$ can be constructed as a mapping cone of two Lagrangian sections in the Fukaya category; as these sections bound no holomorphic strips or disks, one expects that their Lagrangian connect sum can be equipped with a bounding cochain. In practice, the process of constructing the bounding cochain is delicate.
    
    We furthermore expect that similar methods should show that given $V=V_1\cap \cdots \cap V_k$ a transverse intersection of tropical hypersurfaces $V_i$, there exists $L_V$ an unobstructed Lagrangian lift of $V$. The Lagrangian $L_V$ is constructed by using the fiberwise sum of the lifts  \cite{subotic2010,hanlon2020functoriality}, so that 
    \[L_V=L_{V_1}+_Q\cdots +_Q L_{V_k}.\]
    While the resulting Lagrangian submanifold $L_V$ may be immersed, over the top dimensional stratum of $V$ the Lagrangian submanifold $L_V$ satisfies \cref{def:geometricLagrangianlift}. This provides the geometric realization. To obtain unobstructedness, we can also write $L_{V}$ as the geometric composition of unobstructed Lagrangian correspondences (each giving the fiberwise sum with $L_{V_i}$). It is expected (from the work \cite{wehrheim2010functoriality,fukaya2017unobstructed}) that the geometric composition of unobstructed Lagrangian correspondences is unobstructed in this setting, from which it follows that $L_V$ is unobstructed by the pushforward bounding cochain. \label{exam:hypersurfaceunobstructed}
\end{example}
\begin{example}
    Given a smooth tropical hypersurface $V$ of a tropical abelian variety $Q=\RR^n/M_\ZZ$, \cite[Example 5.2.0.7]{hicks2019tropical} constructs an unobstructed Lagrangian lift $(L_V,b)$ inside the symplectic torus $T^*Q/T^*_\ZZ Q$. The proof of unobstructedness is easier than the hypersurface setting (as one does not need to worry about issues of compactness).
    \label{exam:abelianvarieties}
\end{example}
The next two examples were suggested by Dhruv Ranganathan.
\begin{example}
    Suppose that $L_1\subset X_1, L_2\subset X_2$ be tautologically unobstructed Lagrangian submanifolds. Then $L_1\times L_2\subset X_1\times X_2$ is again a tautologically unobstructed Lagrangian submanifold. Furthermore, if the methods in \cite{amorim2017kunneth} can be adapted to the Charest-Woodward model of Floer cohomology that we use, then the product of unobstructed Lagrangians is unobstructed. It follows that when $V_i\subset Q_i$ have Lagrangian brane lifts, then so does $V_1\times V_2\subset Q_1\times Q_2$.
\end{example}
\begin{example}
    Suppose for $t\in [0, 1]$ we have geometric Lagrangian lifts $L_{V_t}$ of a family of tropical subvarieties $V_t$. Furthermore, suppose that for $t\in [0, 1)$ the lift is a Lagrangian brane lift. Then $L_{V_1}$ is a Lagrangian brane lift of $V_1$. The proof uses Fukaya's trick to choose perturbation data so that for $t$ close to 1, the $L_{V_t}$ all have the same moduli spaces of pseudoholomorphic disks. Then there exists a subsequence of bounding cochains for the $L_t$ which converge to a bounding cochain on $L_1$. 
\end{example}
There are few general criteria for determining if a Lagrangian submanifold is unobstructed. To highlight some of the subtlety of the problems, we exhibit a tropical Lagrangian which bounds a non-regular Maslov-index zero disk.
\begin{example}
    Consider the projection to the base of the Lagrangian torus fibration of the tropical Lagrangian submanifold $L_{V_4}$ drawn in \cref{fig:feynman}. Let $\ell$ be the dashed red line. Take the standard metric on $\RR^2$ so we may identify $T\RR^2$ with $T^*\RR^2$. Provided that one takes a symmetric construction of the Lagrangian pairs of pants (for example, using the construction of \cite{matessi2018lagrangian}), the holomorphic cylinder $T\ell/T_\ZZ\ell$ tropicalizing to the line $\ell$ intersects the Lagrangian $L_{V_4}$ cleanly along an $S^1$. This yields an isolated holomorphic disk with boundary on $L_{V_4}$. This is \emph{not} a regular holomorphic disk.
    \label{exam:feynman}
\end{example}
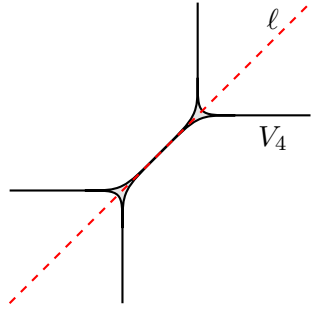
\begin{figure}
    \centering
    \begin{tikzpicture}

\draw (-1,0) -- (0.5,0) -- (0.5,-1.5) (1.5,2.5) -- (1.5,1) -- (3,1) (0.5,0) -- (1.5,1);
\draw[fill=gray!20] (0,0) .. controls (0.5,0) and (0.5,0) .. (0.5,-0.5) .. controls (0.5,0) and (0.5,0) .. (1,0.5) .. controls (0.5,0) and (0.5,0) .. (0,0);
\draw[fill=gray!20] (1.5,1.5) .. controls (1.5,1) and (1.5,1) .. (1,0.5) .. controls (1.5,1) and (1.5,1) .. (2,1) .. controls (1.5,1) and (1.5,1) .. (1.5,1.5);
\draw[dashed, red] (-1,-1.5) -- (3,2.5);
\node[below] at (2.5,1) {$V_4$};
\node[above] at (2.5,2) {$\ell$};
\end{tikzpicture}     \caption{The projection to the base of the Lagrangian torus fibration of a tropical Lagrangian. The holomorphic cylinder above the red dashed line cleanly intersects the tropical Lagrangian.}
    \label{fig:feynman}
\end{figure}
Further examples of non-regular Maslov-index zero disks are given in \cite{hicks2021observations}. In \cref{subsubsec:obstructedness}, we give examples of geometric Lagrangians lifts which are unobstructed, but not tautologically unobstructed for any choice of admissible almost complex structure on $(\CC^*)^n$.
In \cref{sec:speyer}, we show that there exists $V$ such that $L_V$ is obstructed.

\subsection{Unobstructedness at boundary}
\label{subsec:unobstructedatboundary}
We now give a method for constructing a bounding cochain for a Lagrangian which is $W$-admissible. 
By \cref{lemma:equivadmissibility}, for every $L$ a $\Delta_\Sigma$-monomially admissible Lagrangian there exists a tropicalized potential $W^{t, 1}_\Sigma: (\CC^*)^n\to \CC$ so that  $L$ is $W^{t,1}_\Sigma$-admissible. We state our results for $W$-admissible Lagrangians as the methods may be of interest beyond the monomially admissible setting.
\begin{theorem}
    Let $W: X\to \CC$ be a potential function, and suppose that $L$ is a $W$-admissible Lagrangian submanifold whose restriction to a large fiber is  $M=L\cap (W^{-1}(t))$ (where $t\in \RR_{\gg 0}$). Suppose that there exists $M_0\subset M$ a union of connected components of $M$ with the property that
    \begin{enumerate}[label=(\roman*)]
         \item the Lagrangian $M_0$ bounds no holomorphic disks and\label{item:tautologicallyunobstructed}
         \item the map $H^1(M_0)\to H^2(L, M_0)$ is surjective. \label{item:contactboundarysurjects}
    \end{enumerate}
    Then $L$ is unobstructed.
    \label{thm:unobstructed}
\end{theorem}
The idea of proof is to construct the bounding cochain for $L$ by ``lifting the curvature term of $L$ to the boundary $M_0$''. 
The condition that $H^1(M_0)\to H^2(L, M_0)$ shows that curvature term (which takes values in the subcomplex $H^2(L, M)$) is the coboundary of something coming from the boundary $M_0$ of $L$. The algebraic content of this statement is \cref{lemma:unobstructing}.
\begin{proof}
    We show that the $A_\infty$ algebras $A=\CF(L, M_0), B= \CF(L), C=\CF(M_0)$ satisfy the conditions \cref{item:anticommutes,item:surjects,item:ses}
    of \cref{lemma:unobstructing}. 
    From \cref{thm:sympfibrationexactsequence}, the sequence $A\to B \to C$ is exact, and $A$ is an $A_\infty$ ideal. Since $M_0$ bounds no holomorphic disks, $C$ is tautologically unobstructed and $A$ is a strong ideal, giving us  (\cref{lemma:unobstructing}:\cref{item:ses}).
    Because $M_0$ bounds no holomorphic disks, $\CF(M_0)=\CM(M_0)$, which is quasi-isomorphic to $\Omega^\bullet(M)$. Thus we have (\cref{lemma:unobstructing}:\cref{item:anticommutes}).
    Finally, the hypothesis that $H^1(M_0)\to H^2(L, M_0)$ surjects is exactly (\cref{lemma:unobstructing}:\cref{item:surjects}).
\end{proof}
We give an example that relates to the discussion in \cite[Section 5.2]{ekholm2013notes}.
\begin{example}[Aganagic-Vafa Brane]
    Let $A\in \RR_{>0}$ be some constant. The Aganagic-Vafa (AV) brane is a Lagrangian submanifold $L_A\subset \CC^3$ parameterized by 
    \begin{align*}
        D^2\times S^1\to& \CC^3\\
        (r_1, \theta_1, \theta_2)\mapsto&\left(\left(\sqrt{A^2+r^2}\right)e^{-i(\theta_1+\theta_2)},re^{i\theta_1}, re^{i\theta_1}\right) 
    \end{align*}
    The Lagrangian $L_A$ is admissible for the potential function $W(z_1, z_2, z_3)=z_1z_2z_3$.
    The restriction to the fiber $M_A\subset W^{-1}(s)=(\CC^*)\times (\CC^*)$ is a product-type torus, so it bounds no holomorphic disks, and we may apply \cref{thm:unobstructed} to conclude that this Lagrangian is unobstructed by a bounding cochain.

    The bounding cochain corrects this Lagrangian submanifold so that it agrees with predictions from mirror symmetry. By application of the open mapping theorem, the only holomorphic disks with boundary on $L_A$ for the standard complex structure must lie in the fiber $W^{-1}(0)$; in fact, the only simple holomorphic disk with boundary on $L_A$  is parameterized by
    \begin{align*}
        u: (D^2, \partial D^2)\to& (\CC^2, L_A)\\
        z\mapsto& (Az, 0, 0).
    \end{align*}
        A computation shows that the partial Maslov indices of this disk are $(2, -1, -1)$ and therefore this is a regular Maslov index zero disk by \cite{oh1995riemann}.
    This shows that the bounding cochain constructed by \cref{thm:unobstructed} is nontrivial.

    Under an additional assumption \cite[Assumption 5.2.3]{hicks2019wallcrossing} one can compute the $m^0$-term, which counts the multiple covers of the disk $u$ with an appropriate weight. The bounding cochain is $\sum_{k=1}^\infty \frac{1}{k}T^{k\omega(u)} x$, where $x\in \CM(M_A)$ is a meridional class of the torus. 

    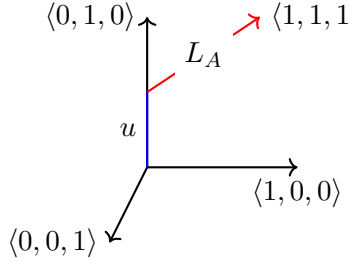
\begin{figure}
        \centering
        \begin{tikzpicture}

\draw[<->] (0,1.5) -- (0,-0.5)-- (2,-0.5); 
\draw[<-] (-0.5,-1.5) -- (0,-0.5) ;
\draw[->,red] (0,0.5) -- (1.5,1.5);

\node[below] at (2,-0.5) {$\langle 1,0,0\rangle$};
\node[left] at (-0.5,-1.5) {$\langle 0,0,1\rangle $};
\node[left] at (0,1.5) {$\langle 0,1,0\rangle$};
\node[right] at (1.5,1.5) {$\langle 1,1,1$};
\node[fill=white] at (0.75,1) {$L_A$};
\draw[thick, blue] (0,0.5) -- (0,-0.5);
\node[left] at (0,0) {$u$};
\end{tikzpicture}         \caption{The projection of the AV Lagrangian $L_A$ to the base $(\RR_{\geq 0})^3$ of the Lagrangian torus fibration of the Aganagic-Vafa Lagrangian  $L_A\subset \CC^3$. We also draw the projection of the single simple holomorphic disk which contributes to a bounding cochain for $L_A$.}
    \end{figure}
    We remark that the Lagrangian $L_A$ is an example of a tropical Lagrangian submanifold considered in \cite{mikhalkin2018examples}, and the projection of $L_A$ under the moment map $\CC^3\to Q=\RR_{\geq 0}^3$ is the ray $( |A|^2, 0, 0) + t\langle 1, 1, 1 \rangle$.
\end{example}

\begin{corollary}
    Let $V\subset Q$ be a genus zero smooth tropical curve. Let $L_V$ be a homologically minimal geometric Lagrangian lift of $V$. Then $L_V$ is unobstructed, so there exists $(L_V, b)$ a Lagrangian brane lift of $V$.
    \label{thm:genuszerounobstructed}
\end{corollary}
\begin{proof}
    We show that the Lagrangian $L_V$ satisfies the criteria of \cref{thm:unobstructed}. 
    Let $V_\infty^{(0)} \subset V$ be the set of semi-infinite edges of $V$. 
    The boundary of this tropical Lagrangian realization $M\subset Y_A$ is contained within the lift of the semi-infinite edges $ \bigsqcup_{e\in V_\infty^{(0)}} L_{e}=T^{n-1}_e\times e$.
    Therefore, $M$ is the disjoint union of tori indexed by the semi-infinite edges of $V$, $\bigcup_{e\in V_\infty^{(0)}} T^{n-1}_e$.
    At each edge, we see that $\pi_2(X_A, L_{e} )=0$.
    It follows that $M\subset Y_A$ bounds no holomorphic disks, so we satisfy (\cref{thm:unobstructed}:\cref{item:tautologicallyunobstructed}).
    Select $f\in V_{\infty}^{(0)}$ any edge, and let $M_0=\bigcup_{\substack{g\in V_\infty^{(0)}}\\ g\neq f}T^{n-1}$.
    
    It remains to prove (\cref{thm:unobstructed}: \cref{item:contactboundarysurjects}) that the image of $H^1(M_0)$ generates $H^2(L_V, M_0)$. 
    From \cref{lem:tropcurvetopology}, for any semi-infinite edge $f$ of $V$, $\res^V_{V_\infty\setminus f}:H^2(L_V)\to \bigoplus_{\substack{g\in V_\infty^{(0)} \\ g\neq f}} H^2(L_g)$ is an injection. From the long exact sequence for relative cohomology,
    \[\bigoplus_{\substack{g\in V_\infty^{(0)} \\ g\neq f}} H^1(L_g)\to  H^2(L_V,M_0)\xrightarrow{0}H^2(L_V)\into \bigoplus_{\substack{g\in V_\infty^{(0)} \\ g\neq f}} H^2(L_g)\]
    the leftmost arrow surjects.
\end{proof} \section{Faithfulness: unobstructed lifts as \texorpdfstring{$A$}{A}-realizations}
	\label{sec:faithfulness}
	Given a Lagrangian torus fibration $X_A\to Q$, the  \emph{A-tropicalization} of a tautologically unobstructed Lagrangian submanifold $L$ is the set of points $q\in Q$ so that $\HF(L, (F_q, \nabla))\neq 0$ for some choice of local system $\nabla$ on $F_q$ (\cref{eq:asupport}). 
We now describe the $A$-tropicalization when $L$ is unobstructed by bounding cochain. Because we again work in the scenario where the space $X$ is non-compact, we must apply a taming condition at infinity to study Floer cohomology. By the same arguments for \cref{thm:sympfibrationexactsequence}, whenever $L_0$ is admissible and $L_1$ is compact for a potential $W: X_A\to \CC$, there exists a well defined $\CF(L_0,\del_0)-\CF(L_1,\del_1)$ bimodule $\CF((L_0,\del_0), (L_1,\del_1))$ given by \cite{charest2019floer}. As in the setting of \cref{def:floercomplex}, $\CF((L_0,\del_0), (L_1,\del_1))$ is generated on the transverse intersections between $L_0$ and $L_1$. The $A_\infty$ bimodule structure comes from counting pseudoholomorphic treed strips. 
Observe that we require $L_1$ to be compact to avoid issues of determining how to apply wrapping Hamiltonians in the definition. 
\begin{example}[Running Example, Continued]
    We return to \cref{exam:morsemodel}.
    Now we bring in the second Lagrangian $(L_{\underline V}, \del_0)$ which we give the trivial local system. 
    The bimodule $\CF((L_{\underline V}, \del_0), (F_q, \del_1))$ has the same generators as \cref{exam:localsystems}; following \cref{not:chains}, we call these generators $x^{01}_J$, where $J\subset \{0, \ldots, n-k\}$. Since neither $F_q$ nor $L_{\underline V}$ bound holomorphic disks, the differential agrees with \cref{exam:localsystems},
    \[m^{1}(x^{01}_I)=\sum_{I\lessdot J} \pm T^{\lambda_0} (\id-P^{\nabla_1}_{c_j}) x_J.\]
    where $\lambda_0$ is the area of the small holomorphic strips. 
    
    We now describe the module product structure. This is given by counts of configurations of a Morse flow-line on $F_q$ which are incident to a strip with boundary on $F_q\cup L_{\underline V}$. Recall that the Hamiltonian push off for $L_{\underline V}$ is given by \cref{eq:LagrangianPrimitive} while the Morse function for $F_q$ is given by \cref{eq:MorsePrimitive}. As before, we use $\{y^1_I\}_{I\subset \{1, \ldots n\}}$ to label the critical points of $f: F_q\to \RR$.
    The moduli space of strips from $\underline x^{01}_I$ and $\underline x_J$ is non-empty when $I<J$; the boundary of the strips sweep out the subtorus spanned by the indices of $J\setminus I$. The downward flow space of $y_K$ is the subtorus spanned by indices $\{1, \ldots, n\}\setminus K$. These two subtori intersect transversely only when 
    $(J\setminus I)\sqcup (\{1, \ldots n\}\setminus K)= \{1, \ldots, n\}$, which can be rephrased as 
    \begin{align*}
        J=K\cup I && K\cap I = \emptyset.
    \end{align*}
    See \cref{fig:productStrip} for a treed strip that contributes to the product.
    From this, it follows that the module product structure is given by 
    \[m^{2}(x^{01}_I\tensor y^{1}_J)=\left\{\begin{array}{cc} P^{\del_1}_{\partial u^+}T^{|J|\cdot \lambda_0}x^{01}_{I\cup J} & \text{if $I\cap J=\emptyset$ and $I\cup J\subset\{0, \ldots, n-k\}$}\\ 0 & \text{otherwise} \end{array} \right.\]
    Here, $|J|\cdot \lambda_0$ is the area of a holomorphic strip from $x^{01}_i$ to $x^{01}_{I\cup J}$, and $P^{\del_1}_{\partial u}$ is the holonomy of the local system along the $F_q$ boundary of the strip. 
    We remark that when $J=\emptyset$, the same formula holds (simply that $u^+$ is regarded as the constant strip at $x^{01}_I$). 
    The map $m^{2}(x^{01}_\emptyset, -): HF^1((F_q, \nabla_0))\to T^{\lambda_0}HF^1((L_{\underline V}, \del_0),(F_q, \del_1))$ surjects whenever the local system $\del_1$ has holonomy of the form $\id+ T^{\lambda_1}A$ along all the $F_q$ boundary of all strips. 
    \label{exam:boundingcochain}
\end{example}
\begin{figure}
    \centering
    \begin{tikzpicture}

\usetikzlibrary{patterns}

\fill[gray!20]  (-2.5,2) rectangle (1.5,-2);
\draw[blue, thick] (-2.5,0) .. controls (-2,1.5) and (-1,1.5) .. (-0.5,0) .. controls (0,-1.5) and (1,-1.5) .. (1.5,0);
\draw[red, thick] (-2.5,0) -- (1.5,0);
\node[below] at (-1.5,0) {$y_{\emptyset}^1$};
\node[above] at (0.5,0) {$y_1^1$};
\node[above right] at (-0.5,0) {$x_{1}^{01}$};
\node[below right] at (1.5,0) {$x_{\emptyset}^{01}$};
\draw[dashed] (-2.5,2) -- (1.5,2) (-2.5,-2) -- (1.5,-2);
\draw (-2.5,-2) -- (-2.5,2);
\draw (1.5,-2) -- (1.5,2);
\node[fill=red, circle, scale=.4] at (-1.5,0) {};
\node[fill=red, circle, scale=.4] at (0.5,0) {};
\node[fill=purple, circle, scale=.4] at (-0.5,0) {};
\node[fill=purple, circle, scale=.4] at (1.5,0) {};
\fill[pattern = north west lines] (-0.5,0) .. controls (0,-1.5) and (1,-1.5) .. (1.5,0);
\node[fill=gray!20] at (0.5,-0.5) {$u_+$};
\end{tikzpicture}     \caption{The treed strip contributing to the bimodule produce $m^2(x^{01}_{\emptyset}\tensor y^1_{1})=x^{01}_1$.}
    \label{fig:productStrip}
\end{figure}
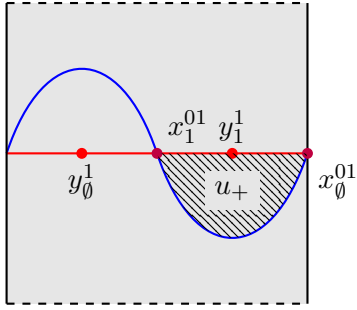
\subsection{Definition of support}
When Lagrangians $(L_0, \del_0)$ and $(L_1,\del_1)$ are unobstructed by bounding cochains $b_0, b_1$, we can deform the Lagrangian intersection Floer cohomology  $\CF((L_0, \del_0), (L_1,\del_1))$ by these bounding cochains to obtain $\CF((L_0,\del_0, b_0), (L_1, \del_1,b_1))$, a $\CF(L_0,\del_0, b_0)-\CF(L_1,\del_1, b_1)$ bimodule. Since $\CF(L_j, \del_0, b_i)$ have no curvature, the differential $m^1:\CF((L_0,\del_0, b_0), (L_1,\del_1, b_1))\to \CF((L_0,\del_0, b_0), (L_1,\del_1, b_1))$ squares to zero, giving us cohomology groups which we can study. 
\begin{definition}
    Let $(L,\nabla, b)\subset X_A$ be an admissible Lagrangian brane. The \emph{$A$-tropicalization} of $(L,\nabla, b)$ is the set 
    \[\tropa(L,\nabla, b):=\left\{q \st \text{$\exists (F_q,\nabla')$  such that , $ \HF((L,\nabla, b), (F_q, \nabla'))\neq 0.$}\right\}\]
\end{definition}

\begin{remark}
    Suppose that there is a bounding cochain $b'$ for $F_q$ so that 
    \[\HF((L,\nabla, b), (F_q, \nabla', b')\neq 0.\]
     As $F_q$ is tautologically unobstructed, a general principle of Lagrangian Floer cohomology (the divisor axiom) states that there exists a local system called $\nabla''$ so that 
    \begin{equation}
        \HF((L,\nabla, b), (F_q, \nabla''))= \HF((L, \nabla,b),(F_q, \nabla', b')) \label{eq:divisor}
    \end{equation}
        The local system $\nabla''$ is usually denoted as $\exp(b')$. 
    To our knowledge, the divisor axiom has not been proven for \cite{charest2019floer} model of Lagrangian intersection Floer cohomology. In \cite{auroux2008special} a proof of the divisor axiom was given for the de Rham version of open Gromov-Witten invariants. The central idea of the proof is that the coefficients in the exponential function make an appearance through the application of the ``forgetting boundary points'' relation between moduli spaces of holomorphic disks. The coefficients $\frac{1}{k!}$ in the expansion of the exponential function show up via the number of ways one can forget boundary marked points. Under the assumptions that Auroux uses, the forgetful axiom for pseudoholomorphic disks holds. In the Charest-Woodward model for $\CF(L_0, L_1)$ we do not expect that perturbations for Morse theory admit a ``forgetting marked point'' axiom. In our setting (where $\omega(\pi_2(X_A, F_q))=0$) the arguments used in \cref{thm:support} show that for all $(F_q, \nabla', b')$ there exists $(F_q, \nabla'')$ so that the identity on $(F_q, \nabla', b')$ factors through $(F_q, \nabla'')$ and vice-versa. Provided that a Charest-Woodward model of the Fukaya \emph{category} with homotopy unit exists, this would prove \cref{eq:divisor} (although not give the closed-form expression for $\nabla'$ as the exponential of the bounding cochain as in the de Rham version).
    From the divisor axiom, it follows that the $A$-tropicalization can be rewritten as:
    \begin{equation}
        \tropa(L,\nabla, b)=\{q \st \exists (F_q,\nabla',b') \text{ with }  \HF((L,\nabla, b), (F_q, \nabla',b'))\neq 0\}
        \label{eq:divisoraxiomassumption}
    \end{equation}
    There remain some subtle differences between bounding cochains and local systems in general. It is clear that we can only expect to replace bounding cochains with local systems in the setting where $L$ is tautologically unobstructed. Furthermore, we do not expect that when $L$ is tautologically unobstructed that we can replace $(L, \nabla)$ with $(L, b)$.  This is because $\val(b)>0$, so it can only be expected to represent local systems whose holonomy is of the form $\id+T^\lambda A$ where $A\in U_\Lambda, \lambda >0$. In the specialization to Lagrangian tori in a Lagrangian torus fibration, we believe that the requirement that $\val(b)>0$ may be loosened to $\val(b)\geq 0$ by application of the reverse isoperimetric inequality (in the same fashion way that the reverse isoperimetric inequality is used to prove that the family Floer sheaf has structure coefficients defined over an affinoid algebra). 
\end{remark}

\subsection{\texorpdfstring{$A$}{A}-tropicalization of tropical Lagrangian lifts}
\label{subsec:supportmatchestropicalization}
In general, it is difficult to compute $\tropa(L, b)$, as it requires having a very good understanding of the differential on $\CF((L,b), (F_q, b'))$.  In \cref{exam:pairofpantsupport} we performed this computation for the pair of pants $V\subset \RR^2$. Computation of the $A$-tropicalization is more tractable when the Lagrangian $L_V$ is a geometric lift of a tropical subvariety because we have good control of leading order contributions to the differential.

The main tool that we use to compute the $A$-tropicalization is the following lemma.
\begin{lemma}
    Let $L_{V}$ be a Lagrangian lift of a tropical curve, and let $U\subset Q$ be an open set such that  $L|_{\syza^{-1}(U)}=L_{\underline V}|_{\syza^{-1}(U)}$. For $q\in U$, let $R(q)$ be the distance from $q$ to $Q\setminus U$. 
    There exists a function $A_{L_V, U}: \RR_{\geq 0}\to \RR$ so that 
    every holomorphic strip $u$ with boundary on $L\cup F_q$ with $q\in U$ is either:
    \begin{itemize}
        \item ``small'' and has image contained within $\syza^{-1}(U)$, and therefore describes a holomorphic strip with boundary on $L_{\underline V}\cup F_q$ or
        \item ``large'' and has symplectic energy greater than $A_{L_V, U}(R(q))$.
    \end{itemize}
    Furthermore, there exists a constant $C_{L_V, U}$ so that 
    \[\lim_{R\to\infty}{\frac{A_{L_V, U}(R)}{R}}=2\cdot C_{L_V, U}.\]
            Additionally, we may replace $F_q$ with a small Hamiltonian push off of $F_q$ while preserving the bound.
    \label{lemma:reverseIso}
\end{lemma}
\begin{proof}
    The lemma is an application of the reverse isoperimetric inequality from \cite{groman2014reverse}. We use the proof for holomorphic strips which is employed by \cite{chasse2023reverse} following \cite{duval2016result,abouzaid2017homological}. Recall that the reverse isoperimetric inequality states that given a Lagrangian $L$ and choice of almost complex structure $J$ there exists a constant $A_{L,J}$ so that  we can lower-bound the energy of pseudoholomorphic disks $u$ with boundary on $L$ by: 
    \begin{equation}
        A_{L,J}\cdot \ell(\partial u)\leq  \int_u \omega, \label{eq:reverseIso}
    \end{equation}
    where $\ell$ is the length as computed by the metric determined by $J, \omega$. 
    The reverse isoperimetric inequality for pseudoholomorphic strips requires the intersections of our Lagrangians $L_V, F_q$ to be ``locally standard'' \cite[Definition II.1]{chasse2023reverse}. Any $F_q$ and $L_{\underline V}$ satisfy this criterion, so whenever $q\in U$ the Lagrangian submanifolds $L_V, F_q$ have locally standard intersection. 
    The reverse isoperimetric inequality from \cite{chasse2023reverse} can be stated as
    \begin{equation}
        s\cdot A_{L_V, F_q}\cdot \ell(\partial u\cap B(C)^c)\leq  \omega( u\cap \tilde U_s) \label{eq:reverseIsoCorner}
    \end{equation}
    where $C$ is chosen so that the radius $s$ normal neighborhoods  $N_s(L_{V}), N_s(F_q)$ are defined for all $s<C$, and 
    \begin{align*}
        \tilde U_s= N_s(L_V)\cup N_s(F_q) && B(C)= N_C(L_V)\cap N_C(F_q).
    \end{align*}
    We obtain a weaker, but more applicable bound by making the replacement  $B(C):=\syza^{-1}(B_C(q))$, for which $F_q$ is a subset. With this substitution, the left-hand side of \cref{eq:reverseIsoCorner} only depends on the length of the boundary of $u$ in $L_V$.
    As the excluded neighborhood $B(C)$ is monotonic in $C$, if we choose  $C(R)< \min(R/4, C_{L_V, U})$ where $C_{L_V, U}$ is the injectivity radius of $L_V$ we can impose the additional condition that $\syza(B(C(R)))\subset U$. 
    We now bound the constant $A_{L_V, F_q}$, which is called $K$ in  \cite[Corollary II.11]{chasse2023reverse}. It is the product of the constants 
    \begin{itemize}
        \item $C_1, C_2^{-1}$ from \cite[Proposition II.8]{chasse2023reverse}, which provide constants of domination between the pseudometric given by a particular plurisubharmonic function $h$ and the standard metric, and 
        \item $A$, which provides a bound for  $|\grad h|$ over $\tilde U_s$
    \end{itemize} 
    A special feature of tropical Lagrangian submanifolds is that over the region $\syza^{-1}(U)\setminus B(C)$, the function $h$ agrees with the distance to $L_{\underline V}$. As $L_V$ is totally geodesic over $\syza^{-1}(U)$, we obtain that $dd^c h(-, \sqrt{-1}-)$ agrees with the metric induced by the standard metric over this chart. Therefore, the constants $C_1, C_2^{-1}$ and $A$ are all $1$ over this region.
    By restricting the integral on the penultimate line of \cite[Equation 9]{chasse2023reverse} to the region $\syza^{-1}(U)$, we may replace $A_{L_V, F_q}$ with $1$ to obtain the bound $ C(R)\cdot\ell(\partial u\cap B(C(R))^c \cap \syza^{-1}(U))\leq  \omega(u)$.
    
    We now show that every strip is either ``small'' or ``large'':
    \begin{itemize}
        \item Suppose that $\partial u\subset \syza^{-1}(U)$; then $u$ describes a strip with boundary on $L_{\underline V}\cup F_0$; we know that all such strips are contained within $\syza^{-1}(U)$ and have an upper bound for their energy.
        \item Otherwise $\partial u\not\subset \syza^{-1}(U)$. Let $\ell_Q$ be distance as measured on $Q$. Observe that for any path $\gamma$ with one endpoint in $F_q$ and another endpoint in $X\setminus \syza^{-1}(U)$ we have the bound 
        \[R-C(R) \leq \ell_Q(\syza(\gamma\cap B(C)^c))\leq   \ell(\gamma\cap B(C)^c)\]
         Since the boundary of $u$ must have at least two such paths, 
        \[A_{L_V, U}(R):= 2C(R) \cdot( R-C(R))< \omega (u).\]
    \end{itemize}
    As $R\to\infty$, we have that $C(R)\to C_{L_V, U}$, from which we obtain the asymptotic behavior of $A_{L_V, U}$.
\end{proof}
    The constant $C_{L_V, U}$ giving the injectivity radius of  $L_{\underline V}$ can be computed from the tropical data of $\underline V$. 
        In the 2-dimensional setting, we obtain the following nice relation. At a top dimensional stratum (edge) $e$ with integral primitive direction $\vec v$, the constant $C_{L_V, U}$ is in \cref{lemma:reverseIso} is $\frac{1}{2|\vec v|}$. The bound for the holomorphic energy of the strips becomes 
    \[2 C\cdot (R-C_{L_V, U})=\frac{R-C_{L_V, U}}{|\vec v|}\]
    We observe that  $\frac{R}{|\vec v|}$ is the \emph{affine radius} of the neighborhood around the point $q$. This can be observed in \cref{exam:pairofpantsupport,example:obstructedline} where the affine lengths of edges in tropical curves govern the areas of holomorphic disks and strips which appear in those computations.
    
\begin{lemma}
    Let $U\subset Q$ be a neighborhood of $q$.
    Suppose that $(L,\nabla, b_0)\subset X_A$ is a Lagrangian brane, whose restriction to $\syza^{-1}(U)$ is 
    \[L|_{\syza^{-1}(U)}=L_{\underline V,m}|_{\syza^{-1}(U)}\]
    where $\underline V\subset U$ is a $k$-dimensional linear subspace, and  $m$ the multiplicity.
    Then there exists a choice of bounding cochain and local system on $F_q$ so that
    \[HF^0((L,\nabla_0, b_0), (F_q,\nabla, b))= \Lambda.\]
    \label{thm:support}
\end{lemma}
\begin{proof}
    To reduce notation in the proof, we will take the same simplifying assumptions as in \cref{lemma:reverseIso}. Additionally, we assume that the local system $\nabla_0$ and bounding cochain $b_0$ on $L$ are trivial.

    We see that $L|_{\syza^{-1}(U)}\cap F_q$ cleanly intersect along a $T^{n-k}\subset F_0$ ; morally we now apply spectral sequence of \cite{pozniak1994floer,schmaschke2016floer} to compute the Floer cohomology of $\CF(L,F_0)$ as a deformation of $C^\bullet(T^{n-k}).$
    Because $L|_{\syza^{-1}(U)}=L_{\underline V}|_{\syza^{-1}(U)}$, we can apply \cref{lemma:reverseIso}. 
    Following \cref{exam:boundingcochain}, apply a Hamiltonian isotopy to $L$ so that $L, F_q$ intersect transversely. Take the perturbation small enough so that the area of the holomorphic strips $\lambda_0$ is less than the bound $\lambda_1:=A_{L_V, U}(R)$ provided by \cref{lemma:reverseIso}. 
    By \cref{lemma:reverseIso}, the map $ m^{2}: \CF(L_V, F_q)\tensor \CF(F_q)\to \CF(L_V, F_q)$  agrees with \cref{exam:boundingcochain} at valuation less than $\lambda_1$.
    \begin{equation}    
        m^{2}(x^{01}_I\tensor x^{1}_J)\equiv \left\{\begin{array}{cc} T^{\lambda_0}x^{01}_{I\cup J} & \text{if $I\cap J=\emptyset$ and $I\cup J\subset\{0, \ldots, n-k\}$}\\ 0 & \text{otherwise} \end{array} \mod T^{\lambda_1}\right.\label{eq:loworderproduct}
    \end{equation}
    Let $\CF(F_q, \Lambda_{\geq 0})$ and $\CF(L_V, F_q, \Lambda_{\geq 0})$ be the filtered $A_\infty$ algebra and bimodule where we use $\Lambda_{\geq 0}$ rather than $\Lambda$-coefficients.     It follows that the map on chains
    \[m^{2}:(x^{01}_\emptyset)\tensor CF^1(F_q, \Lambda_{\geq  0}) \to CF^1(L_V, F_q, \Lambda_{\geq \lambda_0})/CF^1(L_V, F_q, \Lambda_{\geq \lambda_1})\]
    surjects.
    Therefore $\CF(L_V, F_q, \Lambda_{\geq 0})$ as a right  $\CF(F_q, \Lambda_{\geq 0})$ module satisfies the criterion of \cref{lem:submoduledeformation} and there exists $b\in \CF(F_q)$ so that $HF^0(L_V, (F_q,b))\neq 0$.

    To extend to the setting where $L$ has a local system $\del_0$, we simply require that $F_q$ be equipped with a local system $\del_1$ which agrees with $\del_0$ on torus spanned by the classes $\{c_1, \ldots c_{n-k}\}$.
\end{proof}
\begin{remark}
    The constant $\lambda_0$ can be taken to zero provided that one works with a model of $\CF(L_V, F_q)$ which allows for clean intersections between $L_V$ and $F_q$; the proof of \cref{lem:submoduledeformation} becomes slightly simpler in that setting. The pearly model developed by \cite{charest2019floer} allows for such configurations of Lagrangian submanifolds.
\end{remark}
\begin{corollary}
    Let $(L,\del_0 ,b_0)$ and $(F_q,\del, b)$ be as above. Then 
    \[\HF((L, \del_0, b_0), (F_q, \del, b))= \bigwedge_{i\in \{1, \ldots, n-k\}} \Lambda\langle x_i \rangle\]
\end{corollary}
\begin{proof}
    Again for expositional purposes, we assume that $\del_0$ and $\del$ are trivial local systems, assume that the multiplicity of the local model $\underline V$ is one, and suppress the bounding cochain on $L$.
    On chains, the action of $\CF(F_q, b)$ on $\CF(L, (F_q, b))$ is a deformation of the action of $\CF(F_q, b)$ on $\CF( L, (F_q, b))$. 
    By using an argument on filtration similar to the one above, the map
    \[m^2(x_\emptyset, -): \CF((F_q, b)) \to \CF(L, (F_q, b))\]
    is a surjection.
    As every class in $\CF((F_q, b))$ is closed and we've proven that $x_\emptyset$ is closed,  every element in $\CF(L, (F_q, b))$ is closed. 
    This proves that $m^1_{\CF(L,(F_q, b'))}=0$, and that $\HF(L, (F_q, b))= \CF(L, (F_q, b))=\bigwedge_{i\in \{1, \ldots, n-k\}} \Lambda\langle x_i \rangle$.
\end{proof}

\begin{corollary}
    Let $(L_V, b)$ be an unobstructed geometric Lagrangian lift of $V$. Then $V^{(0)}\setminus V^{(1)}\subset \tropa(L_V, b)$.
    \label{cor:atropicalization}
\end{corollary}
If we assume \cref{eq:divisoraxiomassumption}, this immediately follows.
\begin{proof} 
    The proof of \cref{thm:support} can be modified to replace the bounding systems everywhere with local systems. The needed observation is that the map from the space of local systems 
    \begin{align*}
        H^1(F_q, U_{\Lambda}) \to& CF^1(L_V, F_q, \Lambda_{\geq \lambda_0})/CF^1(L_V, F_q, \Lambda_{\geq \lambda_1})\\
        \nabla \mapsto& m^1_{(F_q, \nabla)}(x_{\emptyset})
    \end{align*}
    is surjective. The same argument as in \cref{lem:submoduledeformation} can be used to construct a local system term by term so that $ m^1_{(F_q, \nabla)}(x_{\emptyset})=0$. See also \cite[Proposition 5.13]{sheridan2020lagrangian} which proves a similar statement for tropical curves using the implicit function theorem \cite[Section 10.8]{abhyankar2001local}.
\end{proof}
 \section{\texorpdfstring{$B$}{B}-realizability and unobstructedness}
	\label{sec:realizableHMS}
	\subsection{HMS for $(\CC^*)^n$}
\subsubsection{Construction of the mirror space}
Given $\syza: X_A\to Q$ a Lagrangian torus fibration, there is a rigid analytic space $X_B$ with a tropicalization map $\tropb: X_B\to Q$. 
As a set, $X_B$ is the set of Lagrangian torus fibers equipped with $U_{\Lambda}$ local system,
\[X_B:= \{(F_q, \nabla)\}\]
which comes with a map $\syzb:X_B\to Q$ sending $(F_q, \nabla)\mapsto q$.
When $Q=\RR^n$, the points of $X_B$ are in bijection with $(\Lambda^*)^n$. We now describe, following \cite{abouzaid2014family,einsiedler2006non}, how this can be realized as the set of points of a rigid analytic space. We also recommend the discussion in \cite[Section 5.1]{sheridan2020lagrangian}.

The \emph{Tate algebra} in $n$-variables over $\Lambda$ is the set of formal power series 
\[T_n:=\left\{\sum_{A\in \ZZ^n}f_A z^A \st f_A\in \Lambda, \val(f_A)\to \infty \text{ as } |A|\to\infty \right\},\]
which is equipped with the \emph{sup-norm}
\[\left\|\sum_{A\in \ZZ^n} f_A z^A\right\|:=\max_{A}|f_A|\geq 0\]
We note that the maximal ideals of $T_n$ are $\{(f_1, \ldots, f_n)\st \val(f_i)\leq 1\}$.

To build our spaces we will glue together \emph{affinoid algebras}, which are quotients of the Tate algebra. The affinoid algebras we will look at are the polytope algebras. Given a bounded rational polytope $P\subset \RR^n$, define
\[\mathcal O_P:=\left\{\sum_{A\in \ZZ^n} f_A z^A \st \val(f_A)+A\cdot p\to \infty \text{ as } \|A\|\to \infty \text{ for all } p \in P \right\}.\]
This is the affinoid algebra.  The elements of this affinoid algebra have the property that they converge when evaluated on $z\in (\Lambda^*)^n$ with $\val(z)\in P$. Furthermore, the points of $\mathcal O_P$ are seen to be in bijection with the points of $\syzb^{-1}(P)$. When $Q$ is compact $X_B$ can be covered by finitely many sets $\syzb^{-1}(P)$ giving $X_B$ the structure of a rigid analytic space.

\subsubsection{From Lagrangians to Coherent Sheaves}
Due to the limitations on currently existing constructions for Fukaya categories, we do not have homological mirror symmetry for a category of non-exact Lagrangian submanifolds in $(\CC^*)^n$. However, different aspects of this homological mirror symmetry statement exist in the literature with strengthened hypotheses. 
\begin{itemize}
    \item The family Floer functor associates to a compact Lagrangian torus fibration $\syza: X_A\to Q$  a rigid analytic space $X_B\to Q$ whose points are in bijection with Lagrangian tori $F_q\subset X_A$ equipped with $U_\Lambda$ local system.  Furthermore, \cite[Theorem 2.10]{abouzaid2017family} constructs a faithful $A_\infty$-functor $\mathcal F: \Fuk^{taut}(X_A)\to \text{Perf}(X_B)$. Here $\Fuk^{taut}(X_A)$ is the Fukaya category of tautologically unobstructed Lagrangian submanifolds.
    \item In the exact setting, we have a complete proof of homological mirror symmetry for $(\CC^*)^n$. The proof comes from recasting a section $L(0)$ of the fibration $\syza: (\CC^*)^n\to Q$ as a cotangent fiber in $T^*T^n$, which is known to generate the exact Fukaya category. A computation shows that the $A_\infty$ algebra $\CF(L(0), L(0))$ is homotopy equivalent to $\hom(\mathcal O_{\CC^n}, \mathcal O_{\CC^n})$.
        In fact, we have a little bit more: it is known that the partially wrapped Fukaya category is mirror to the derived category of coherent sheaves on a toric variety (\cite{abouzaid2006homogeneous,kuwagaki2020coherent}).
\end{itemize}
For this paper, we will only compute $\CF((L_V,b), (F_q, \nabla))$, which means that we need substantially less than a HMS functor of \cite{abouzaid2017family}.
\begin{theorem}[\cite{abouzaid2014family}]
    Consider the Lagrangian torus fibration $\syza:X_A\to Q$, with $Q$ \emph{compact}.
        From this data we can construct a rigid analytic mirror space $X_B$ whose points $z$ are in bijection with pairs $(F_q, \nabla)$.
    For any \emph{tautologically unobstructed} Lagrangian brane $L\subset X_A$, there exists a coherent sheaf $\mathcal F(L)$ on $X_B$ so that 
    \[\hom(\mathcal F(L), \mathcal O_{z})= HF^0(L, (F_q, \nabla)).\]
    \label{thm:familyfloer}
\end{theorem}
\begin{assumption}
    \Cref{thm:familyfloer} still holds under the following weakened assumptions:
    \begin{itemize}
        \item[(*)] The base is allowed to be $Q=\RR^n$, and we additionally require that the Lagrangian $L$ be monomially admissible.
        \item[(**)] The Lagrangian $L$ is allowed to be unobstructed by bounding cochain, in which case there exists a coherent sheaf $\mathcal F(L,b)$ on $X_B$ so that 
        \[\hom(\mathcal F(L,b), \mathcal O_{z})= HF^0((L,b), (F_q, \nabla)).\]
    \end{itemize}
    \label{ass:hmscn}
\end{assumption}

We now discuss the difficulties, expectations, and progress of proving the assumption. The primary difficulties arise from non-compactness and unobstructedness.

Non-compactness presents three immediate issues.  The first is Gromov compactness. We expect that after one places appropriate taming conditions on our Lagrangian submanifolds (as in \cref{app:pearlymodel}) the moduli spaces needed to construct the family Floer functor can be given appropriate compactifications. 

The second more difficult issue regards the role that wrapping plays in computing the Floer cohomology between two non-compact Lagrangians. In the exact setting, the morphism space between two Lagrangians is computed as the limit of $\CF(\phi^i(L_0), L^i)$, where $\phi^i$ is a wrapping Hamiltonian, and the limit is taken over continuation maps. In the non-exact setting, these continuation maps have a non-zero valuation, and only have inverses defined over the Novikov field (with possibly negative valuation). To our knowledge, this version of the Fukaya category has not been constructed. However, since for our application, we only need to compute Floer cohomology against Lagrangian torus fibers (which are compact), we can ignore the issues of the wrapping Hamiltonian. 

Finally, there is the issue of coherence of $\mathcal F(L, b)$. Here, we use the monomial admissibility condition. We recall the proof of coherence when $Q$ is compact. The sheaf $\mathcal F(L, b)$ is constructed by defining it over affinoid domains on the mirror, which correspond to convex domains $U\subset Q$. The convex domain  $U$ is ``small enough''  if there exists a Hamiltonian isotopy of $L$ so that it intersects all Lagrangian torus fibers $F_q$ with $q\in U$ transversely. Over each small enough $U$, the sheaf is computed by $\CF((L, b), (F_q, \nabla))\tensor \mathcal O_U$, where $\mathcal O_U$ affinoid ring of the affinoid domain $X_{U, B}$ associated to the convex domain $U$. Since $\CF((L, b), (F_q, \nabla))$ is finitely generated, and (in the compact setting) we can cover $Q$ with finitely many such $U$, we obtain that the mirror sheaf is coherent. If we drop the condition of $Q$ being compact, and impose the condition that $L$ is monomially admissible, we can still cover $Q$ with a finite set of convex (possibly non-compact) small enough domains $U\subset Q$ by using invariance of the Lagrangian submanifold under symplectic flow in the direction of the monomial ray over each monomial region.

We now remark upon the difficulty of unobstructedness. \cite[Remark 1.1]{abouzaid2017family} states that the ``tautologically unobstructed'' hypothesis for construction of the family Floer functor is technical in nature, and it is expected that the family Floer functor should carry through using unobstructed Lagrangian submanifolds. As we do not require functoriality, such an adaptation of family Floer cohomology to the Charest-Woodward model would not require studying moduli spaces beyond those already studied in \cite{charest2019floer}. We believe the main items left to prove for this construction are the following:
\begin{itemize}
    \item Showing that ``Fukaya's trick'' for pulling back perturbation datum between Lagrangian fibers over sufficiently small convex domains can be worked out in the more technically challenging setting of domain-dependent perturbations. This does not appear to present a problem when working with the set up of \cite{charest2019floer}.
    \item Showing that ``homotopies of continuation maps'' exist in the version of Lagrangian intersection Floer cohomology one is working with. In \cite{charest2019floer}, continuation maps are constructed using holomorphic quilts. There is also an additional challenge of showing that one can construct homotopies of continuation maps corresponding to changes in the choice of stabilizing divisor.
\end{itemize}
Finally, we note that the work in progress of Abouzaid, Gromann, and Varolgunes generalizing \cite{groman2018wrapped,varolgunes2021mayer} to the Fukaya category will prove homological mirror symmetry for unobstructed Lagrangian submanifolds of $(\CC^*)^n$, giving us \cref{ass:hmscn}.
\begin{remark}
    A different approach that would bypass family Floer theory would be to expand homological mirror symmetry for toric varieties \cite{abouzaid2006homogeneous,kuwagaki2020coherent} in the non-exact setting. This would involve developing \cite{ganatra2018sectorial} to the non-exact setting. While there is no clear obstruction to expanding the Liouville sector framework to include obstructed Lagrangian submanifolds that are geometrically bounded, there are at least two technical and challenging issues that would need to be overcome. 
    Many of the arguments used in \cite{ganatra2018sectorial} would have to be carefully redone by replacing geometric bounds obtained by energy and exactness with other methods for bounding holomorphic disks. In fact, these techniques are already employed in a limited capacity in \cite{ganatra2018sectorial} for the proof of the K\"unneth formula (as products of cylindrical Lagrangians are usually not cylindrical). 
    The second issue is understanding how to incorporate curvature into the homological algebra constructions employed by \cite{ganatra2018sectorial}. One possible workaround would be to first construct the partially wrapped pre-category of Lagrangian branes that are equipped with bounding cochains (which is an uncurved filtered $A_\infty$ pre-category) and localize at continuation maps to construct the partially wrapped category. This already requires some care, as it is not immediately clear how the filtration would play a role in this localization (the continuation maps would have positive energy, so there may be convergence issues). 
    The second, more ambitious approach would be to attempt to construct a ``curved partially wrapped Fukaya category'', by starting with a partially wrapped Fukaya pre-category whose objects are (potentially obstructed) Lagrangian submanifolds. 
    This second approach would require one to understand what a filtered $A_\infty$ pre-category is and also to construct localizations of these categories.
\end{remark}
\subsection{Unobstructed Lagrangian lift implies $B$-realizability}
By employing \cite{abouzaid2017homological} (with the possible extensions stated in \cref{ass:hmscn}) we can associate to each Lagrangian brane $(L_V, b)$ a closed analytic subset of $X_B$:
\[
     \YB(L_V, b):= \Supp(H^0(\mathcal F(L_V, b))).
\]
\begin{corollary}
    Consider the Lagrangian torus fibration $\syza:(\CC^*)^n=X_A\to Q$ and a tropical subvariety $V\subset Q$. 
    Suppose that $(L_V, b)$ is a Lagrangian brane lift of $V$. Then:
    \begin{itemize}
        \item $(L_V,b)$ is an $A$-realization of $V$ in the sense that $\tropa(L_V,b)=V$
        \item $V$ is $B$-realizable.
    \end{itemize}
    \label{cor:realizability}
\end{corollary}
\begin{proof}
    By \cref{ass:hmscn}, $\tropa(L_V,b)=\tropb(\YB(L_V, b))$. 
    In \cref{cor:atropicalization}, we proved that $V^{(0)}\subset \tropa(L_V, b)\subset V$. 
    Since $\YB(L_V, b)$ is a closed analytic subset,  $\tropb(\YB(L_V, b))$ is the union of closed rational polyhedra in $N_\RR$ \cite[Proposition 5.2]{gubler2007tropical}. As a result, $\tropb$ is closed and contains $\overline{V^{(0)}}=V$.
    It follows that $\YB(L_V, b)$ is a closed analytic subset of $X_B$ which realizes $V$. 
\end{proof}

\begin{corollary}
    Assuming \cref{ass:hmscn} (*), (**), let $V$ be a smooth hypersurface or a smooth genus zero tropical curve in $\RR^n$. Then $V$ is $B$-realizable. 
    \label{cor:hypersurfacerealizable}
\end{corollary}
\begin{corollary}
    Assuming \cref{ass:hmscn} (**), let $V$ be a  smooth tropical hypersurface of a tropical abelian variety $Q=\RR^n/M_\ZZ$. Then $V$ is $B$-realizable.
    \label{cor:abelianvarietyrealizable}
\end{corollary}
\begin{corollary}
    \emph{Without} assuming any portion of \cref{ass:hmscn}, let $V$ be a 3-valent tropical curve in a tropical abelian surface $Q$. Then $V$ is $B$-realizable.
    \label{thm:abeliansurfaceunobstructedness}
\end{corollary}
\begin{proof}
    The condition of 3-valency comes from using 
    \begin{itemize}
        \item  \cite{holmes2022affine} to build an affine dimer model associated to each 3-valent vertex and
        \item  \cite{hicks2021tropical} to build tropical Lagrangian lifts from a dimer model.
    \end{itemize}
    We now address why \cref{ass:hmscn} may be dropped. Since $Q$ is a tropical abelian surface (and is therefore compact), the symplectic manifold $X_A$ is compact. Since the Lagrangian lift $L_V$ is graded of dimension 2, it is tautologically unobstructed for a generic choice of almost complex structure (as Maslov index 0 disks appear in expected dimension -1). 
\end{proof}

\subsection{Nonplanar tropical curves do not have tautologically unobstructed lifts}
\label{subsubsec:obstructedness}
Even in the setting where $V$ is a genus zero tropical curve, it is rare for the Lagrangian lift $L_V$ to be a tautologically unobstructed Lagrangian submanifold. 

Before constructing an example, we observe that the valuations of the ``big-strips'' in \cref{thm:support} are dictated by the radius of the neighborhood $U_q$ that we can construct around the point $q$ which is disjoint from $V^{(1)}$.
In particular, this can be applied to  \cite[Proposition 5.10]{sheridan2020lagrangian} to show that tautologically unobstructed Lagrangian lifts of tropical curves have supports that extend to an appropriate toric compactification of the mirror algebraic torus.
\begin{prop}
    Let $\Sigma$ be a fan. Suppose that $V$ is a tropical curve with semi-infinite edges in the directions of the rays of  $\Sigma$. Suppose the fan of $\Sigma$ has the additional property that $\langle \alpha, \beta\rangle \leq 0$ for all 1-dimensional cones $\alpha \neq  \beta$ and $\langle -, - \rangle$ is the standard inner product.
        Then $\YB((L_V,b),0)$ compactifies to a rigid analytic space inside $X_B(\Sigma)$, the rigid analytic toric variety with fan $\Sigma$.
    \label{prop:compactification}
\end{prop}
\begin{proof}
    We first describe the rigid analytic structure on $X_B(\Sigma)$ given by \cite{rabinoff2012tropical}.
    From \cite{payne2009fibers}, the space $X_B(\Sigma)$ comes with a fibration $\tropb: X_B(\Sigma)\to Q(\Sigma)$, which is a partial compactification of $Q$ (see \cite[Definition 3.6]{rabinoff2012tropical}). \Citeauthor{rabinoff2012tropical} then covers $X_B(\Sigma)$ with charts given by the max-spec of affinoid algebras.    

    Let $P_\sigma\subset Q$ denote a convex set which can be written as the form $P'+\sigma$ for $\sigma\in \Sigma$ and some convex compact polytope $P'\subset Q$. Associated to $P_\sigma$ is a subset $\overline P_\sigma\subset Q(\Sigma)$, and  an affinoid algebra
    \[\mathcal O_{P_\sigma}:=\left\{\sum_{A\in (\sigma^\vee\cap \ZZ^n)} f_A z^A \st\val(f_a)+a\cdot p\to \infty \text{ as } \|A\|\to \infty \text{ for all } p \in P_\sigma \right\}.\]
    We can cover $X_B(\Sigma)$ with charts given by the max-spec of $\mathcal O_{P_\sigma}$ (which covers $\tropb^{-1}(\overline P_\sigma))$.
    
    We now unpack what it means for a Lagrangian submanifold $(L,b)$ constructed via family Floer theory to give a coherent sheaf $\mathcal F((L, b))$ on the rigid analytic space $X_B(\Sigma)$. 
    In the family Floer construction, for a sufficiently small convex polytope $P$ in the base of $Q$, one takes a Hamiltonian perturbation $L_{ P}$ of $L$ so that $L_{P}|_{\syza^{-1}(P)}$ is a disjoint set of flat sections of $\syza^{-1}(P)\to P$, and that the bounding cochain is similarly parallel to the flat section. As a result, we may identify the chains $\CF((L_P,b), (F_q, \nabla))$ for all $q\in P$. Additionally, for $q\in P$, one can appropriately choose almost complex structures (using Fukaya's trick) so that the moduli spaces of strips contributing to the differential on $\CF((L,b), (F_q, \nabla))$ does not depend on $q$. Because the bounding cochain on $L$ is parallel to the flat section, the contribution of $b$ to the differential on $\CF((L,b), (F_q, \nabla))$ does not depend on $q\in P$. As a consequence, the dependence of the structure coefficients $\langle m^1_{(L, b), (F_q, \nabla)}(x), y\rangle$  on $(F_q, \nabla)$ factors through the flux homomorphism. Pick a base point $x_0$ on $F_q$ and for each $x\in L\cap F_q$ a path $\gamma_{x}$ from $x_0$ to $x$. By identifying $(F_q, \nabla)$ with a point $z\in \tropb^{-1}(P)$, we obtain that 
    \[\langle m^1_z(x), y\rangle = \sum_{a\in H_1(F_q)} c_az^a\]
    where $c_a$ is the area and local-system weighted count of pseudoholomorphic strips $u$ so that $[\gamma_x\cdot \partial_{F_q}u \cdot \gamma_y^{-1}]=a$. 
    The Lagrangian $L$ defines a complex of sheaves $\mathcal F(L)$ over $(X_B|_{P})$ if these structure coefficients belong to $\mathcal O_P$. The restriction maps compose up to homotopy of the chain complex. This is proven using the reverse isoperimetric inequality to bound the area of holomorphic strips $u$ (which govern convergence) below by the winding of the $F_q$ boundary component of $U$ (which governs the exponent appearing in $z^a$).  To obtain a coherent sheaf of complexes on $X_B$, one must be able to cover $Q$ with finitely many sufficiently small sets $P$. When $Q$ is compact, this is always possible. In the setting we study, we must take some of the sets $P$ to be of the form $P_\sigma$ in order to construct a finite cover.
    
    We now perform this construction for $L_V$ our tautologically unobstructed Lagrangian lift of a tropical curve $V$. Let $e$ be a semi-infinite edge of $V$ pointing in the $\alpha$ direction, where $\alpha\in \Sigma$ is a 1-dimensional cone. Then there exists a $P_\alpha$ so that $V|_{P_\alpha}$ is a 1-dimensional ray. Since $\langle \alpha, \beta \rangle <0$ for all 1-dimensional rays $\beta\neq \alpha$, the projection $\chi^\alpha: X_A\to \CC$ given by the $\alpha$-monomial has the property that the $\chi^\alpha|_{L_V}: L_V\to \CC$ fibers over a real ray outside of a compact set. This is the main input needed in \cite[Proposition 5.10]{sheridan2020lagrangian} to show that the differential on $\CF(L_V, (F_q, \nabla))$ is of the form $\sum_{a\in H_1(F_q), \langle a, \alpha\rangle\geq 0} c_az^a$, and that $\val(c_a)+a\cdot p\to \infty \text{ as } \|A\|\to \infty \text{ for all } p \in P_\alpha$. It follows that $\langle m^1_z(x), y\rangle\in \mathcal O_{P_\alpha}$. 

    We can choose a finite cover of $Q$ by sets of the form $P_\sigma$ so that $\syza(L_V)\subset P_\sigma$ if and only if $|\sigma|\leq 1$. It follows that $\mathcal F(L_V)$ define a sheaf on $X_B(\Sigma)$.
\end{proof}

\begin{figure}
    \centering
\begin{subfigure}{.45\linewidth}
    \centering
    \begin{tikzpicture}[yscale=-1]

\fill[fill=red!20] (-1.5,-2) -- (-1.5,-1) -- (1,-1) -- (1,-2);
\draw (-2.5,0) -- (-0.5,0) -- (-0.5,-2);
\draw (1,1.5) -- (-0.5,0);

\node at (0,1) {$C$};

\node at (0.5,-1.5) {$P_\sigma$};
\node[fill=red, scale=.5, circle] at (-0.5,-1.5) {};
\node[left] at (-0.5,-1.5) {$F_q$};
\end{tikzpicture}     \caption{Condition on Tropical Curve}
    \label{fig:valprojection}
\end{subfigure}
\begin{subfigure}{.45\linewidth}
    \centering
    \begin{tikzpicture}[scale=.4]

\usetikzlibrary{patterns}
\draw[fill=gray!20]  (5.5,-4.5) rectangle (-5.5,4.5);
\draw[pattern=north west lines, pattern color=blue!20, draw=red]  (0,0) ellipse (4 and 4);

\draw[fill=gray!50] (2.5,0) .. controls (2,0) and (1.5,0.5) .. (1,1) .. controls (0.5,1.5) and (-1.5,1.5) .. (-2,1) .. controls (-2.5,0.5) and (-2,-0.5) .. (-2,-1) .. controls (-2,-2) and (1,-1.5) .. (1.5,-1) .. controls (2,-0.5) and (2,0) .. (2.5,0);
\draw (2.5,0) -- (5.5,0);
\node at (-4.5,3.5) {$\mathbb C$};
\node at (-0.5,0) {$\chi^\alpha(L_V)$};
\node[right, red] at (2,3.5) {$F_{q}$};
\node[circle, fill=black, scale=.2] at (4,0) {};
\end{tikzpicture}     \caption{Projection on $\chi^{\alpha}: X_A\to \CC$}
    \label{fig:eprojection}
\end{subfigure}
\caption{Monomial admissibility forces strips with $\langle \partial u, \alpha\rangle>0$ to have large symplectic area.}
\end{figure}
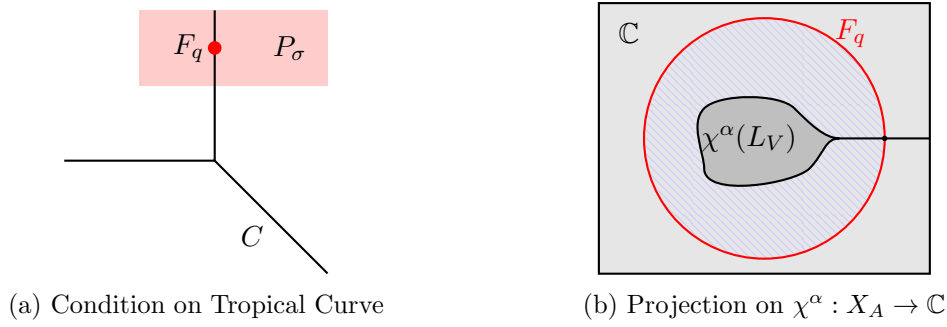
In the setting above (where $L_V$ is tautologically unobstructed and equipped with the trivial local system), the above computation not only shows that $\mathcal F(L_V)$ extends to $X_B(\Sigma)$ but also shows that we can compute the points in the compactifying locus. 
For a semi-infinite edge $e$, let $P_\alpha=P+\langle \alpha\rangle$ a convex polytope whose only intersection with $V$ is along the edge $e$. Without loss of generality, we will assume that the edge $e$ is of the form $(t,0, \ldots, 0)\subset Q=\RR^n$, with $t$ tending to $\infty$. We can write the max-spec of $P_\alpha$ as 
\[\{(z_1, \ldots, z_n)\in \Lambda\times (\Lambda^*)^{n-1} \st \val(z_1, \ldots, z_n)\in \overline{P_\alpha}\subset (\RR\cup \infty)\times \RR^{n-1}\}.\] We prove that the point $(0,1,\ldots 1)\in \Supp(\mathcal F(L_V))$. The $\langle a, \alpha \rangle=0$ terms of $\langle m^1_z(x_\emptyset), x_I\rangle$ agree with holomorphic strips for the differential on $\CF(L_{\underline e}, (F_q, \nabla))$, so we can write 
\[\langle m^1_z(x_\emptyset), x_I\rangle =(1-z^{\langle I, a\rangle})+ \sum_{a\in H_1(F_q), \langle a, \alpha\rangle >0} c_az^a\]
When we have a sequence of points $\{z^k\}_{k\in \NN}$ with the property that $m^1_{z^k}(x_\emptyset)=0$ (i.e. $z^k\in \Supp(\mathcal F(L_V))$) and $\lim_{k\to\infty} \val(z_1^k)= \infty$ (so that the limit belongs to the compactifying toric divisor), the above equation states that $\lim_{k\to\infty } \val(z^k_i)=1$ for all $i\neq 1$. 
We conclude that the closure of $\Supp(\mathcal F(L_V))$ inside of $X_B(\Sigma)$ contains the point $(0, 1, \ldots, 1)$. 

 We now construct an example of a Lagrangian brane lift of a tropical curve that is unobstructed, but not tautologically unobstructed. 
\begin{example}
    \label{example:obstructedline}
    Consider the tropical line $V_c\in \RR^3$ drawn in \cref{fig:lineinrr3}. The tropical line $V_c$ has two pants centered at the points $(0,0,0)$ and $(-c,-c,0)$, and whose legs at $(0,0,0)$ point in the directions
    \begin{align*}
        e_1= \langle 1, 0,0\rangle && e_2=\langle 0, 1, 0 \rangle, && e_c= \langle -1, -1, 0 \rangle 
    \end{align*}
    and whose legs at $(-c,-c,0)$ point in the directions
    \begin{align*}
        e_3= \langle 0, 0, 1 \rangle && e_4= \langle -1, -1, -1 \rangle && -e_c = \langle 1, 1, 0 \rangle.
    \end{align*}
    We prove that $L_{V_c}$ bounds a holomorphic disk for all but at most 1 value of $c$.

    Assume for contradiction that for all values of $c$ the Lagrangian submanifold $L_{V_c}$ is tautologically unobstructed, and requires no bounding cochain.
    Then the Lagrangians $L_{V_c}$ satisfy the conditions of \cref{prop:compactification}, so each $Y_{V_c}:=\Supp(\mathcal F(L_{V_c}))$ compactifies to give a curve inside of $\mathbb P^3$. Since this curve intersects each of the toric divisors at a single point, we conclude that every $Y_{V_c}$ is a line in $\mathbb P^3$.
    Furthermore, every one of these lines contains the points $(1:0:0:0)$ and $(0:1:0:0)$ in $\mathbb P^3$.
    Since a line in $\mathbb P^3$ is determined by two points, this implies that $ Y_{V_c}=  Y_{V_{c'}}.$ However, as $V_c\neq V_{c'}$, they cannot be realized by the same subvariety, a contradiction. 
\end{example}

\begin{figure}
    \centering
    \begin{tikzpicture}[xscale=-1]

\draw[fill=gray, opacity=.25] (-6.5,4) -- (-6.5,0) -- (-0.5,0) -- (-0.5,4) -- cycle;
\draw[fill=gray, opacity=.25] (-6.5,0) -- (-5.5,-1.5) -- (0.5,-1.5) -- (-0.5,0) -- cycle;
\begin{scope}[shift={(-1.5,0)}]

\draw[thick, ->] (-3,1.5) -- (-5.5,1.5);
\draw[thick, ->] (-3,1.5) -- (-3,3);
\draw[thick,->] (-2,0.5) -- (-1,-1);
\draw[thick,->] (-2,0.5) -- (0,1.5);
\draw[thick, red] (-2,0.5) -- (-3,1.5);

\node at (-5.5,1) {$e_1$};
\node at (-2.5,3) {$e_2$};
\node at (-0.5,-0.5) {$e_3$};
\node at (-0.5,1.5) {$e_4$};
\end{scope}
\begin{scope}[shift={(6,0)}]

\draw[fill=gray, opacity=0.25] (-6.5,4) -- (-6.5,0) -- (-5.5,-1.5) -- (-5.5,2.5) -- cycle;
\end{scope}
\node at (-4.5,0.5) {$c$};

\end{tikzpicture}      \caption{A tropical line $V_c$. The Lagrangian lift $L_{V_c}$ necessarily bounds a holomorphic disk; we conjecture that the projection to the base of the Lagrangian torus fibration of this holomorphic disk lives over the red edge and has area controlled by the affine length $c$.}
    \label{fig:lineinrr3}
\end{figure}
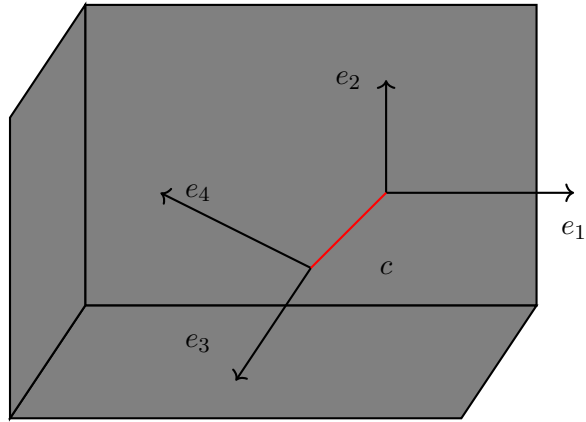
This doesn't contradict the realizability of $V_c$. Indeed, by \cref{thm:genuszerounobstructed}, the bounding cochain on $L_{V_c}$ need only be supported on three of the four legs of $L_{V_c}$. However, the above argument shows that one cannot construct a bounding cochain for $L_{V_c}$ which restricts to zero on the two semi-infinite edges which share a vertex (which implies that the bounding cochain cannot be zero).

Using mirror symmetry, we can ``back solve'' for the valuation of the holomorphic disk which necessitates the use of a bounding cochain on $L_{V_c}$. We may assume that the bounding cochain has trivial restriction to the $e_1$ edge. It follows that the tropical line $Y_{V_c}$ may intersect toric divisors at the points $(0,1,1), (1+\exp(b_1), 0,1+\exp(b_3)), (z^{-c}+\exp(c_1), z^{-c}+\exp(c_2), 0)$. Since these have to satisfy the equation of a line, there exists $t$ so that 
\[(1-t)(0,1,1)+t(\exp(b_1), 0, \exp(b_2))=(z^{-c}(\exp(c_1)),z^{-c}(\exp(c_2)), 0).\]
From examining the third term, $t=(1-\exp(b_2))^{-1}$, we already see that $b_2\neq 0$. From examining the third term
\[(1-\exp(b_2))^{-1}\exp(b_1)=z^c(\exp(c_1)).\]
from which we see that $\val(b_2)=c$. From this, we conclude that there exists a pseudoholomorphic disk of energy $c$ on $L_{V_c}$.  
 	\subsection{Speculation on Speyer's well-spacedness criterion}
	\label{sec:speyer}
\Cref{cor:realizability} proves the forward direction of \cref{conj:unobstructednessisrealizable}.
To investigate the reverse direction, we look at an example of a non-realizable tropical curve.
In \cite{Mikhalkin2004Amoebas} it was observed that every cubic curve in $\CP^3$ is planar (\cref{fig:realizable}). Consequently, the example drawn in \cref{fig:nonrealizable} --- a tropical cubic which is not contained within any tropical plane --- cannot arise as the tropicalization of any curve in $\CP^3$.
\begin{figure}
    \centering
    \begin{subfigure}[t]{.45\linewidth}
        \centering
    \begin{tikzpicture}[scale=.75]

\begin{scope}[]

\draw[fill=gray, opacity=0.25] (0.5,0) -- (-0.5,2) -- (-0.5,4.5) -- (0.5,3) -- cycle;
\draw[fill=gray, opacity=0.25] (-4.5,3) -- (-4.5,0) -- (0.5,0) -- (0.5,3) -- cycle;
\draw[fill=gray, opacity=0.25] (0.5,3) -- (-0.5,4.5) -- (-5.5,4.5) -- (-4.5,3) -- cycle;
\draw[fill=gray, opacity=0.25] (-0.5,4.5) -- (0.5,3) -- (2.5,4) -- (1.5,5.5)-- cycle;
\draw[fill=gray, opacity=0.25] (0.5,0) -- (2.5,1) -- (2.5,4) -- (0.5,3) -- cycle;
\draw[thick] (-3,2) -- (-2.5,2) -- (-2.5,1.5) -- (-3,1) -- (-3.5,1) -- (-3.5,1.5) -- (-3,2) 
(-4.5,0.5) -- (-4,0.5) -- (-4,0) (-3.5,1) -- (-4,0.5) 
(-3.5,1.5) -- (-4.5,1.5) (-3,1) -- (-3,0) 
(-3,2) -- (-3,2.5) -- (-2.5,3) (-3,2.5) -- (-4.5,2.5)
(-2.5,2) -- (-1.5,3) (-2.5,1.5) -- (-2,1.5) -- (-0.5,3)
(-2,1.5) -- (-2,0) (-3.5,4.5) -- (-2.5,3) (-1.5,3) -- (-2.5,4.5) (-0.5,3) -- (-1.5,4.5);
\draw[fill=gray!40, opacity=0.75](-4.5,3) -- (-2.5,4) -- (2.5,4) -- (0.5,3) -- cycle;
\draw[thick] (-2.5,3) -- (-0.5,4) (-1.5,3) -- (0.5,4) (-0.5,3) -- (1.5,4);

\draw[ultra thick, red] (-3,2) -- (-3,2.5) -- (-2.5,3);
\draw[ultra thick, green] (-2.5,2) -- (-1.5,3);
\end{scope}

\end{tikzpicture}      \caption{ A well-spaced tropical curve. The affine lengths of the red and green match. }
    \label{fig:realizable}
    \end{subfigure}
    \begin{subfigure}[t]{.45\linewidth}
        \centering
    
\begin{tikzpicture}[scale=.75]
\begin{scope}[shift={(10,0)}]

\draw[fill=gray, opacity=0.25] (0.5,0) -- (-0.5,2) -- (-0.5,4.5) -- (0.5,3) -- cycle;
\draw[fill=gray, opacity=0.25] (-4.5,3) -- (-4.5,0) -- (0.5,0) -- (0.5,3) -- cycle;
\draw[fill=gray, opacity=0.25] (0.5,3) -- (-0.5,4.5) -- (-5.5,4.5) -- (-4.5,3) -- cycle;
\draw[fill=gray, opacity=0.25] (-0.5,4.5) -- (0.5,3) -- (2.5,4) -- (1.5,5.5)-- cycle;
\draw[fill=gray, opacity=0.25] (0.5,0) -- (2.5,1) -- (2.5,4) -- (0.5,3) -- cycle;

\draw[fill=blue!20] (-3.5,1.5) -- (-3.5,1) -- (-3,1) -- (-2.5,1.5) -- (-2.5,2) -- (-3,2) -- cycle;
\draw[thick] (-3,2) -- (-2.5,2) -- (-2.5,1.5) -- (-3,1) -- (-3.5,1) -- (-3.5,1.5) -- (-3,2) 
(-4.5,0.5) -- (-4,0.5) -- (-4,0) (-3.5,1) -- (-4,0.5) 
(-3.5,1.5) -- (-4.5,1.5) (-3,1) -- (-3,0) 
(-3,2) -- (-3,2.5) -- (-2.25,3.25) (-3,2.5) -- (-4.5,2.5)
(-2.5,2) -- (-1.5,3) (-2.5,1.5) -- (-2,1.5) -- (-0.5,3)
(-2,1.5) -- (-2,0) (-3.25,4.75) -- (-2.25,3.25) (-1.5,3) -- (-2.5,4.5) (-0.5,3) -- (-1.5,4.5) (-2.25,3.25) -- (0.25,4.25) ;

\draw[red, ultra  thick] (-3,2) -- (-3,2.5) -- (-2.25,3.25);
\draw[green, ultra thick] (-2.5,2) -- (-1.5,3);
\draw[fill=gray!40, opacity=0.75](-4.5,3) -- (-2.5,4) -- (2.5,4) -- (0.5,3) -- cycle;
\draw[thick] (-1.5,3) -- (0.5,4) (-0.5,3) -- (1.5,4);

\end{scope}
\end{tikzpicture}     \caption{A non-realizable tropical curve. The affine length of the green segment is uniquely minimal among all paths from the cycle to the non-linearity locus of the curve.}
    \label{fig:nonrealizable}
    \end{subfigure}
    \caption{}
\end{figure}
\begin{corollary}
    Let $V$ be the tropical curve from \cite[Example 5.12]{Mikhalkin2004Amoebas}. Then the standard lift of $L_V$ is an obstructed Lagrangian.
\end{corollary}
A general criterion for understanding this phenomenon was stated in \cite{speyer2014parameterizing}. 
\begin{theorem*}[Speyer's Well-Spacedness]
    Let $V$ be a genus-one tropical curve whose cycle is contained within a linear subspace $H$. Let $d_1, \ldots d_k$ be the affine lengths of paths along the edges of $V$ to the boundary of $V\cap H$. If the minimal distance occurs at least twice, the curve $V$ is realizable.
\end{theorem*} 
We now speculate on how Speyer's well-spacedness criteria can be understood in terms of holomorphic disks with boundary on $L_V$. 
For $L_V$ to be unobstructed, it is necessary for the lowest energy terms in $m^0$ to be null-homologous. In particular: the set
\[\left\{u \st \omega(u)\leq \min_{\substack{0\neq [\partial u'] \in H^2(L, M)}} \omega(u')\right\}\]
of minimal area non-null-homologous disks must contain at least two elements. This matches the ``two minimal distance'' criterion of Speyer's Well-Spacedness theorem.

In \cite{hicks2021observations}, we saw that tropical cycles on $W\subset \RR^2$ are related to non-regular Maslov-index zero disks with boundaries on the Lagrangian lifts $L_W$; it was speculated that these Maslov-index zero disks could appear regularly if they were glued onto a regular holomorphic disk or strip.  In \cref{example:obstructedline} we saw that the Lagrangian brane lift of a small neighborhood of the green segment in \cref{fig:nonrealizable} must have a regular disk with energy given by the affine length of the edge.

In the example given by \cref{fig:nonrealizable}, we conjecture that there are  regular holomorphic disks with boundaries on $L_V$ whose projections under the moment map are:
\begin{itemize}
    \item The union of the blue hexagon (a non-regular disk) and green path (a regular disk); call this speculative disk $u_1$ and
    \item The union of the blue hexagon (a non-regular disk) and red path (a regular disk); call this speculative disk $u_2$.
\end{itemize}
Using that the area of homology classes of disks with boundary on $L_V$ correspond to affine length, the disks $u_1$ and $u_2$ have the matching symplectic area if the affine lengths of the green and blue path match.
In this case, the homology class of $[\partial u_1]-[\partial u_2]$ doesn't wrap around the portion of the homology of $L_V$ which arises from $V$, and by a similar argument used in \cref{thm:genuszerounobstructed} we see that $[\partial u_1]-[\partial u_2]\subset H_1(L_{V_\infty^{(0)}})$. We could then apply the methods used in the proof of \cref{thm:genuszerounobstructed} to conclude that $L_V$ is unobstructed.

In the event that $\omega(u_2)$ is uniquely minimal, the boundary of $\partial(u_2)$ is a non-trivial homology class in $H_1(L_V)$, suggesting that the contribution to $m^0\in \CF(L_V)$ is a non-removable obstruction. 
 	\subsection{Deformations, superabundance, and not-wide}
	\label{app:jacobian}
	\subsubsection{Geometric Deformations of \texorpdfstring{$L$ and $(V, \mathcal L)$}{L and (V, L)}}
Given $V\subset \RR^n$ a tropical subvariety, a Lagrangian $L_V$ should correspond to a lift of $V$  equipped with a line bundle. In this section, we examine how the deformations of $L_V$ up to Hamiltonian isotopy match deformations of a tropical curve equipped with a line bundle $(V, \mathcal L)$.

Given a fixed tropical line bundle $\mathcal L\to V$ we can identify deformations of $\mathcal L$ with $H^1(V, \RR)$: this is because deformations of invertible locally integral affine functions from $U$ to $\RR$ correspond to constant differences. 
Similarly, the deformations of $V\subset \RR^n$ as a smooth tropical subvariety can be computed sheaf-theoretically. 
We choose a cover conducive to this computation. 
To each $v\in V$, let $\str(v)$ be the union of the edges that contain $v$. We allow $v$ to be a leaf (at the end of a semi-infinite edge).  Then the $\str(v)$ form a cover of $V$, with $\str(v)\cap \str(w)=\underline V_{vw}$ whenever $vw$ is an edge. There are two types of vertices $v$ that we must consider:
\begin{itemize}
    \item If $v$ is an internal vertex, then the deformations of $\str(v)$ are identified with the integral affine space $Nv=T_v\RR^n=\RR^n$. 
    \item If $v_\infty$ is a boundary vertex incident to edge $e$, then the deformations of $\str(v_\infty)$ are identified with the integral affine space $\RR^{n-1}$ perpendicular to the semi-infinite edge attached to $v_\infty$.
\end{itemize}

Over each edge $e$, the deformations of the tropical curve are given by the normal bundle to $e$.
In summary, let $\Def_{V}$ be the sheaf of deformations of the tropical embedding of $V$, and let $\Def_{\mathcal L}$ be the deformations of a fixed line bundle $\mathcal L$ over $V$. We have:
\begin{align*}
    \Def_V(\str(v))= \RR^n&& \Def_V(\str(v_\infty))=e_{v_\infty}^\bot&& \Def_V(\str(e))=e^\bot
\end{align*}

For compact Lagrangian $L$, the infinitesimal deformations of $L$ up to Hamiltonian isotopy are described by classes in $H^1(L, \RR)$. Since $L_V$ is non-compact, we only consider the \emph{admissible} deformations of non-compact $L_V$ which preserve the condition in \cref{def:admissiblitycondition}.
Let $\Omega^1_{admis}(L_V, \RR)$ be the 1-forms on $L_V$ with the property that 
\begin{itemize}
    \item For each monomial region $U_\alpha$, the 1-form  $\eta|_{L_V\cap U_\alpha}$  is invariant under the flow in the $\alpha$-direction.
    \item Furthermore, $\eta(\alpha)=0$.
\end{itemize}
We let $\Omega^0_{admis}(L_V, \RR)$ be those functions which, outside of a compact set, are invariant under the flow in the $\alpha$ direction of the corresponding monomial region from \cref{def:admissiblitycondition}.

We can similarly decompose $L_V$ into sets $L_{\str(v)}$, which we will take to be:
\begin{itemize}
    \item The standard Lagrangian pair of pants when $v$ in an interior vertex so that $L_{\str(v)}\cap L_{\str(w)}=L_{\underline V_{vw}}$. In this case $\Omega^i_{admis}(L_V, \RR)=\Omega^i(L_V, \RR)$. 
    \item A non-compact cylinder extending to the boundary whenever $w$ is a vertex at a non-compact edge. 
\end{itemize}
We then compute $H^1(\Omega^\bullet_{admis}(L_V))$. The cohomology is the same as the first cohomology of the total complex; the first page in the spectral sequence is
\[\begin{tikzcd}
        \bigoplus_{v\in V} H^0(\Omega^\bullet(L_v))  \arrow{d} & \bigoplus_{v\in V} H^1(\Omega^\bullet(L_v))\arrow{d} &\cdots\\
        \bigoplus_{e\in V} H^0(\Omega^\bullet(L_e))  \arrow{d}& \bigoplus_{e\in V}H^1(\Omega^\bullet(L_e))\arrow{d} &\cdots\\
        0 & 0 
    \end{tikzcd}\]
We now start to identify these with deformations of tropical curves. 

\begin{align*}
    \Def_V(\str(v))= H^1(\Omega^\bullet(L_v))&& \Def_V(\str(e))=H^1(\Omega^\bullet(L_e))\\
    \RR=H^0(\Omega^\bullet(L_v)) &&  \RR=H^0(\Omega^\bullet(L_e))
\end{align*}
turning the first page of the spectral sequence into
\[\begin{tikzcd}
    \bigoplus_{v\in V} \RR \arrow{d}  & \bigoplus_{v\in V}\Def_V(e) \arrow{d}&\cdots\\
    \bigoplus_{e\in V}\RR \arrow{d}& \bigoplus_{e\in V}\Def_V \arrow{d}&\cdots\\
    0 & 0 
\end{tikzcd}
\]
The spectral sequence for $H^1(\Omega^\bullet(L))$ converges at the second page for this covering, so
\[H^1(\Omega^\bullet(L_{V}))= H^0(V, \Def_V))\oplus H^1(V,\RR)=H^0(V, \Def_V)\oplus H^0(V, \Def_{\mathcal L}).\]

In general, understanding the moduli space of Lagrangian submanifolds isotopic to $L_V$ modulo Hamiltonian isotopy is a difficult question. In the setting of Lagrangian torus fibrations, there is a smaller class of isotopies that we can hope to understand. We say that a Lagrangian isotopy $i_t: L_V\to X_A$ is a fiberwise isotopy if $\syza(i_t(q))$ is constant for all $q\in L_V$. 

\begin{conjecture}
    Let $L_V$ be a homologically minimal Lagrangian lift of a tropical curve $V$. Then the subspace of $H^1(L_V, \RR)$ arising from the flux classes of fiberwise Lagrangian isotopies is identified with $H^0(V,\Def_{\mathcal L})$. 
    Additionally, 
    \[\{\text{Fiberwise isotopies}\}/\{\text{Fiberwise Hamiltonian isotopies}\}\simeq H^1(V, \Aff^*_{V})_0,\]
    where $\Aff^*_V(U)$ is the sheaf of invertible locally integral affine functions from $U$ to $\RR$, and $H^1(V, \Aff^*_{V})_0$ is the connected component of the group which contains the identity.
\end{conjecture}

Previous work of Mikhalkin \cite{mikhalkin2009tropical} identifies the tropical Picard group of $V$ with $H^1(V, \Aff^*_V)$. 
Mirror symmetry, therefore, identifies fiberwise isotopies of a tropical Lagrangian $L_V$ with modifications of the line bundle on the mirror curve $Y_V$.
\subsubsection{Not-Wide and Superabundance}
A tropical curve is called \emph{superabundant} if the space of deformations $\Def_V$ has a higher dimension than the expected dimension of deformations of the $B$-realization. Superabundance is a computable criterion that indicates that a curve may not be realizable. For example, the tropical curve examined in \cref{sec:speyer} is a superabundant curve. It is known in certain cases \cite{cheung2016faithful} that non-superabundant implies realizable. 

In symplectic geometry, there are two ways to make sense of deformations of Lagrangian submanifolds. 
The first kind of deformation is the deformation of geometric Lagrangian submanifold up to Hamiltonian isotopy. The infinitesimal deformations of Lagrangian submanifolds modulo Hamiltonian isotopy are given by $H^1(L, \RR)$.  
The second deformation space which we can consider is the component of the moduli space of objects at $L$ \cite{toen2007moduli}. The tangent space to this moduli space is $HF^1(L)$. We note that as $\CF(L)$ is a deformation of $C^\bullet(L)$, we have that $\dim HF^1(L)\leq HF^1(L)$. If $\dim HF(L)=\dim H(L)$, then the Lagrangian $L$ is called \emph{wide}.

As the previous section identifies infinitesimal deformations of the pair $(V, \mathcal L)$ with $H^1(L_V)$, we are led to conjecture:
\begin{conjecture}
    Let $V$ be a smooth tropical curve, and let $L_V$ be its Lagrangian lift. Then $V$ is superabundant if and only if $L_V$ is not wide.
\end{conjecture}

 \appendix
\section{Pearly model in symplectic fibrations}
\label{app:pearlymodel}
Given a compact, spin, and graded Lagrangian $L$ inside of a rational compact symplectic manifold $X$, \cite{charest2019floer} constructs a filtered $A_\infty$ algebra $\CF(L, h, \mathcal P, D)$.
In \cite{charest2019floer} it is assumed that the space $X$ is compact. In this appendix, we outline how to extend \cite{charest2019floer} to the setting where $X$ is non-compact and is equipped with a potential function $W: X\to \CC$; and $L$ is a Lagrangian submanifold which is admissible with respect to $W$.
\begin{definition}
    Let $X$ be a symplectic manifold, and let $W: X\to \CC$ be a function. We say that $W$ is a potential if there exists a compact subset $U\subset \CC$ so that 
    \begin{itemize}
        \item $W^{-1}(U)$ is compact, and
        \item the restriction $W: X\setminus W^{-1}(U)\to \CC\setminus U$ is a symplectic fibration with compact fibers.
    \end{itemize}
    We say that a Lagrangian $L$ is \emph{$W$-admissible} if there exists $R\in \RR$ so that 
    $W(L)\cap \{z\st |z|>R\} \subset \RR_{>R}$. 

    Given a $W$-admissible $L$, we say that a Morse function $h:L\to \RR$ is admissible if there exists $R'>R$ so that 
    \[W(\Crit(h))\cap \{z\st |z|>R\}\subset \{R'\}.\]
    and $\grad h$ points outwards from $R'$ under the projection $W$.
    \label{def:bottleneck}
\end{definition}
Let $Y=W^{-1}(R')$. Given a $W$-admissible Lagrangian submanifold $L$, the restriction to the fiber $M:=L\cap Y$ is a Lagrangian submanifold of $Y$.
Because $h$ points outwards along the collar $M\times \RR_{>R'}\subset L$, the Morse complex $\CM(L, h)$ is well defined. 
The compatibility of the Morse function with the potential function means that $h^+:= h|_M$ is a Morse function for $M$ and that we have a map of $A_\infty$ algebras
\[
    \begin{tikzcd}
    \underline \pi:\CM(L, h)\arrow{r} &\CM(M, h^+ ).
    \end{tikzcd}
\]
This should be interpreted as the pullback map of the inclusion of the boundary.

We show that \cite{charest2019floer} extends to the setting of $W$-admissible Lagrangian submanifolds.
\begin{theorem}
    Let $W: X\to \CC$ be a potential function. Let $L$ be a $W$-admissible Lagrangian submanifold, whose restriction to a large fiber is $ M \subset Y= W^{-1}(t)$. Let $h: L\to \RR$ and $h^+:=h|_M: M\to \RR$ be admissible Morse functions.
    There exist:
    \begin{itemize}
        \item  stabilizing symplectic divisors $D_X\subset X, D_Y\subset Y$; and
        \item regular choices for perturbation systems $\mathcal P_L, \mathcal P_M$ for $L$ and $M$;
    \end{itemize}
     so that the construction of \cite{charest2019floer} can be applied to give a well defined $A_\infty$ algebra $\CF(L,h, \mathcal P_L, D_X)$. Furthermore, the choices of perturbations and divisors can be taken so that the projection on chains    
    \[\pi:\CF(L, h, \mathcal P_L, D_X)\to \CF(M, h^+, \mathcal P_M, D_Y)\]
    is a $\Lambda$-filtered $A_\infty$ algebra homomorphism.
    \label{thm:sympfibrationexactsequence}
\end{theorem}
    The theorem consists of two statements: the construction of a pearly model of stabilized treed disks in the setting of Lagrangians which are admissible for a potential function, and the compatibility between the pearly model of total space of the fibration and the pearly model of the fiber. These are analogous to     \cite[Corollary C.4.2 and Theorem C.5.1]{hicks2019wallcrossing} which handle the setting where $X=Y\times \CC$ and $W: X\to \CC$ is projection to the second factor.
    In this appendix, we prove that $\CF(L,h, \mathcal P_L, D_X)$ is well defined; the existence of the projection $\pi:\CF(L, h, \mathcal P_L, D_X)\to \CF(M, h^+, \mathcal P_M, D_Y)$ is the same as the proof of \cite[Theorem C.5.1]{hicks2019wallcrossing}.

    To construct $\CF(L,h, \mathcal P_L, D_X)$ one needs to 
    \begin{enumerate}
        \item Construct a stabilizing divisor for $X$ which is suitably compatible with the potential $W: X\to \CC$ \label{item:constructstabilizing};
        \item Show that we can pick perturbations for the almost complex structure so that the map $W: X\to \CC$ is holomorphic outside of a compact set; and \label{item:pickperturbations}
                \item Prove that for such choices of perturbations the moduli spaces have appropriate Gromov compactifications.  \label{item:compactifications}
    \end{enumerate}

    \subsection*{\Cref{item:constructstabilizing}: Constructing a stabilizing divisor} Pick $R$ sufficiently large so that outside of $U=B_{R}(0)\subset \CC$ the Lagrangian submanifold $L$ fibers over the positive real ray, and the map $X\setminus W^{-1}(U)\to \CC\setminus U$ is a symplectic fibration. 
        For $\theta\in [0, 2\pi],r\geq R$ we take a path 
    \[
        \gamma_{\theta, r}(t)=\left\{\begin{array}{cc} Re^{i\theta(2t)} & t\in [0, 1/2)\\
        (R+(2t-1)(r-R))e^{i\theta} & t\in [1/2, 1)
        \end{array}\right.
    \] which travels first in the angular, then radial direction from $R$ to $re^{i\theta}$.
     To every path $\gamma(t):I\to \CC_{|z|>R}$ we have a symplectic parallel transport map $P_{\gamma}: Y_{\gamma(0)}\to Y_{\gamma(1)}$. 
    Consider the monodromy $P_{\gamma_{2\pi,R}}: Y_{R}\to Y_{R}$ given by parallel transport around the loop $Re^{i\theta}$ in the positive direction. 
    Pick a path of $\omega_Y$-tamed almost complex structures $J_{Y_R,\theta}: [0, 2\pi]\to  J_\tau(Y_R, \omega_Y)$ such that $P_{\gamma_{2\pi,R}}^* J_{2\pi}=J_0$. This gives us endomorphism of the sub-bundle of the tangent spaces to the fibers
    \begin{align*}
        J_{re^{i\theta}}: TY_{re^{i\theta}}\to TY_{re^{i\theta}}\\
        J_{re^{i\theta}}=P_{\gamma_{\theta, r}}^*J_{Y_R,\theta}
    \end{align*}
    Since over every point with $|z|>R$ we have a splitting $T_{(y,z)}X=T_yY\oplus T_z\CC$, we can give $T(X\setminus W^{-1}(U))$ the tame almost complex structure locally defined by $J_{re^{i\theta}}\oplus \jmath_\CC$. 
    \begin{definition}
        We say that a $\omega$-tame almost complex structure on $X$ is $W$-admissible if, when restricted to $W^{-1}(U)$ it can be written as $J_{re^{i\theta}}\oplus \jmath_\CC$ for some path of almost complex structures  $J_{Y_R, \theta}\in \mathcal J_\tau(Y_R, \omega)$. We denote the space of such almost complex structures $\mathcal J_{\tau, W, R}(X, \omega_X)$. 
    \end{definition}
    The goal will to be to construct a stabilizing divisor $D_X\subset X$ in such a way that $D_X$ is transverse to all $Y_{re^{i\theta}}$ with $r\geq R$, and subsequently show that there exists an open dense set of almost complex structures belonging to $\mathcal J_{\tau, W, R}(X, \omega_X)$ which are $E$-stabilized by $D_X$  (\cite[Definition 4.24]{charest2019floer}). In the setting of Lagrangian cobordisms, the comparable statements are proven in \cite[Appendix C.3 and Lemma C.1.3]{hicks2019wallcrossing}.

    We first construct the divisor $D_X$. Take $E_X\to X$ a vector bundle whose first Chern class is $1/2\pi[\omega_X]$, so that the pullback $E_Y\to Y$ a vector bundle whose first Chern class is $1/2\pi[\omega_Y]$. Pick a family of Hermitian structures on $E_{Y,\theta}\to Y_{Re^{i\theta}}$ depending on $\theta$ so that the curvature is $-i\omega_Y$ and so that $P_{\gamma_{2\pi, R}}^* E_{Y, 2\pi}=E_{Y, 0}$ as Hermitian line bundles. Let $i_{\theta_0}:Y_{Re^{i\theta_0}} \to X$ be the inclusion of the fiber over $Re^{i\theta_0}$. Take a Hermitian structure on $E_X\to X$ with curvature $-i\omega_X$ and the property that $i^*{\theta_0}P_{\gamma_{\theta_0, r}|_{[1/2,1]}}^*E_X= E_{Y, \theta_0}$ as Hermitian line bundles.

    We will construct the stabilizing divisor $D_X$ as the zero locus of an asymptotically holomorphic section $s_{k, X}: X\to E^k_X$. First, using \cite{auroux2001symplectic} we can pick asymptotically holomorphic sections $s_{k,Y}: Y\to E^k_Y$ with the property that $s_{k, Y}^{-1}(0)$ is disjoint from $M$. We obtain a second asymptotically holomorphic section by pullback $P_{\gamma_{2\pi, R}}^* s_{k, Y}$. By \cite{auroux1997asymptotically} we can find a family $s_{k, Y, \theta}$ of such sections so that $s_{k, Y , 0}=s_{k, Y}$ and $s_{k, Y, 2\pi}=P_{\gamma_{2\pi, R}}^* s_{k, Y}$.

    Using this family of sections, we create an asymptotically holomorphic section $s_{k, X, out}: X\to E_X^k$ which is given by 
    \[s_{k, X, out}:= \rho_{k, R+1}(|z|)\cdot P_{\gamma_{\theta, r}}^* s_{k, Y, \theta}\]
    where $\rho_{k, R+1}(|z|): \CC\to \RR$ is a function which is concentrated (in the sense of \cite[Definition 2]{auroux2001symplectic}) at the circle of radius $R+1$. The zero set $s_{k, X, out}^{-1}(0)$ enjoys the properties that 
    \begin{itemize}
        \item $s_{k, X, out}^{-1}(0)$ is disjoint from  $L$;
        \item for $k$ sufficiently large, $s_{k, X, out}^{-1}(0)$ is a symplectic divisor in $W^{-1}(\{z\st |z|>R\})$; and 
        \item $s_{k, X, out}^{-1}(0)$ intersects $Y_{re^{i\theta}}$ transversely for all $r>R$.
    \end{itemize}
    This constructs the sections taking the place of $s_{k, X\times \CC, out}$ in \cite[Appendix C.3.2]{hicks2019wallcrossing}. The remainder of the construction of $D_X$ involves subsequently perturbing this section over the region $W^{-1}(U)$ which exactly follows \cite[Appendix C.3.2]{hicks2019wallcrossing}.

\subsection*{\Cref{item:pickperturbations}: Finding perturbations}The construction of an open dense set of $E$-stabilized almost complex structures proceeds in the same fashion as \cite[Section C.3.3]{hicks2019wallcrossing} (itself based on the argument of \cite[Section 4.5]{charest2019floer}). The main tool needed for the argument to run is to show that the space of almost complex structures regularizing holomorphic disks of energy up to $E$ is dense in $\mathcal J_{\tau, W, R}(X, \omega_X)$. By application of the open mapping principle to $W$, every pseudoholomorphic disk in consideration must either:
    \begin{itemize}
         \item pass through $W^{-1}(U)$, where they can be made regular through  perturbations confined to the region $W^{-1}(U)$ by application of \cite[Lemma 5.6]{cieliebak2007symplectic}; or
         \item be confined to a fiber $W^{-1}(t)$ with $t\in U$, in which case they can be made regular through perturbations constrained in the fiberwise direction. Since the fiber is compact, the set of such perturbations is open and dense.
    \end{itemize}

\subsection*{\Cref{item:compactifications}: Compactness of moduli spaces}The proof that the moduli spaces of pseudoholomorphic treed disks considered are compact uses that we may apply open-mapping principle type arguments for perturbations chosen from $\mathcal J_{\tau, W, R}(X, \omega_X)$, and that the Morse flow line components of treed disks point outwards at the boundary. (\cite[Proposition C.4.1]{hicks2019wallcrossing}) 
\begin{remark}
    In the examples we consider (potentials coming from tropicalized superpotentials associated to a monomial admissibility data) the fibers of the potential will in general not be compact.
    However, the monomially admissible condition ensures that the restriction of monomially admissible $L\subset X$ to $M\subset Y$ will be compact. As a result, all pseudoholomorphic disks contributing to treed disks will have boundary contained within a compact subset of $X$; we conclude that the moduli space of treed disks has compactification given by broken treed disks. 
\end{remark}  \section{Auxiliary results for filtered \texorpdfstring{$A_\infty$}{A-infinity} algebras and modules}
\label{app:ainfinity}
In this section, we give some background for filtered $A_\infty$ algebras and bimodules, as well as provide some methods for constructing bounding cochains using the filtration on the $A_\infty$ algebra.
\subsection{A short review of bounding cochains}
\label{subsub:boundingcochains}
The \emph{Novikov ring} with $\CC$-coefficients is the ring of formal power series
\[\Lambda_{\geq 0}:=  \left\{\sum_{i=0}^\infty  a_i T^{\lambda_i}\st a_i\in \CC, \lambda \in \RR_{\geq 0}, \lim_{i\to\infty} \lambda_i = \infty\right\}.
\]
The field of fractions is the Novikov field $\Lambda$.
A \emph{filtered $A_\infty$} algebra $(A, m^k_A)$ is a free graded $\Lambda_{\geq 0}$-module $A^\bullet$ equipped with $\Lambda_{\geq 0}$-linear products 
\[
    m^k:(A^\bullet)^{\tensor k}\to (A^{\bullet+2-k})
\]
for each $k\geq 0$. These are required to satisfy the axioms of \cite[Definition 3.2.20]{fukaya2010lagrangian}. Among these axioms are: 
\begin{itemize}
    \item The quadratic filtered $A_\infty$ relationship
\[
    0=\sum_{k_1+k'+k_2=k} (-1)^{\clubsuit(\underline x, k_1)} (m^{k_1+1+k_2})\circ (\id^{\tensor k_1}\tensor m^{k'}\tensor id^{\tensor k_2}) (x_1, \cdots ,x_k)
\]
The sign is determined by $\clubsuit(\underline x,k_1):= k_1+\sum_{j=1}^{k_1} \deg(x_j)$.
    \item Each $A^i$ have a filtration $F^{\lambda}A^i$ respecting the filtration on $\Lambda_{\geq 0}$, and a basis belonging to $F^0(A^i)\setminus \bigcup_{\lambda >0} F^{\lambda}A^i$.
\end{itemize}
Given a filtered-$A_\infty$ algebra, we can also consider the $\Lambda$-linear products on $A\tensor_{\Lambda_{\geq 0}} \Lambda$. We call $A\tensor_{\Lambda_{\geq 0}} \Lambda$ a $\Lambda$-filtered $A_\infty$ algebra.

Let $(A, m_A^k)$ and $(B, m_B^k)$ be $A_\infty$ algebras. A \emph{filtered $A_\infty$ homomorphism} from $A$ to $B$ is a sequence of filtered graded maps
\[
    f^k:A^{\tensor k}\to B
\]
satisfying the \emph{quadratic $A_\infty$ homomorphism relations:}
\begin{align*}
    \sum_{k_1+k'+k_2=k}(-1)^{\clubsuit(\underline x, k_1)}&f^{k_1+1+k_2}\circ
    (\id^{\tensor k_1}\tensor m^{k'}_A \tensor \id^{k_2})\\
    =&\sum_{i_1+\cdots i_j=k}m^j_{B} \circ (f^{i_1}\tensor \cdots \tensor f^{i_j}).
\end{align*}
There similarly exists a notion of a homotopy between filtered $A_\infty$ homomorphisms.

The main difficulty with filtered $A_\infty$ algebras is that they do not have cohomology groups, as 
\[(m^1_A)^2= m^2_A(m^0_A\tensor \id)\pm m^2_A(\id\tensor m^0_A).\]
When $m^0_A=0$, the right-hand side of the relation is zero and we say that $A$ is a \emph{tautologically unobstructed}  $A_\infty$ algebra. 

It is desirable to work with tautologically unobstructed $A_\infty$ algebras as they can be studied with the standard tools employed for cochain complexes. 
Therefore, one might restrict one's study to tautologically unobstructed filtered $A_\infty$-algebras.
Problematically, tautologically unobstructed filtered $A_\infty$ algebras are not closed under the relation of filtered $A_\infty$ homotopy equivalence.
This can be remedied by considering filtered $A_\infty$ algebras equipped with bounding cochains.

Let $A$ be a filtered $A_\infty$ algebra. A deforming cochain is an element $d\in A^1$ with $\val(d)>0$. The $d$-deformation of $A$ is the filtered $A_\infty$ algebra $(A, d)$ whose 
    \begin{itemize}
        \item underlying chains groups agree with $A$ and,
        \item composition maps are given by the $d$-deformed $A_\infty$ products,
\begin{equation}
    m^{k}_{(A, d)} = \sum_{l=0}^\infty \sum_{j_0+\cdots + j_k=l} m^{k+l}_A(d^{\tensor j_0}\tensor\id\tensor d^{\tensor j_1}\tensor \cdots \tensor \id \tensor d^{\tensor j_k}).
    \label{eq:deformingcochain}
\end{equation}
    \end{itemize}
\begin{definition}
    When $m^0_{(A, b)}=0$, we say that $b$ is a \emph{bounding cochain}, and we say that the algebra $A$ is \emph{unobstructed}.\label{def:boundingcochain}
\end{definition}
Given $f: A\to B$ a filtered $A_\infty$ homomorphism and $b\in A$ a bounding cochain, there is a pushforward bounding cochain $f_*(b)\in B$ so that $(B, f_*(b))$ is unobstructed. 
When $f^{k}=0$ for $k\neq 1$ then $f_*(b)=f(b)$. 
The existence of a pushforward bounding cochain shows that unobstructedness is a property of filtered $A_\infty$ algebras which is preserved under the equivalence relation of filtered $A_\infty$ homotopy equivalence. 

In applications, we use $\Lambda$-filtered $A_\infty$ algebras as opposed to filtered $A_\infty$ algebras\footnote{This is because the continuation maps in Lagrangian intersection Floer cohomology are usually only weakly-filtered.}. However, the homological algebra of filtered $A_\infty$ algebras is notationally easier to describe (as there exist elements living in a minimal filtration level).
A computation allows us to understand deformations and bounding cochains for the former (defined using the \cref{eq:deformingcochain}) in terms of the latter.
\begin{claim}
    Suppose that $A$ is a filtered $A_\infty$ algebra and $b$ a bounding cochain for $A$. Then $b\tensor 1\in A\tensor_{\Lambda_{\geq 0}}\Lambda$ is a bounding cochain for the $\Lambda$-filtered $A_\infty$ algebra $A\tensor_{\Lambda_{\geq 0}}\Lambda$.
\end{claim}

\subsection{Extending an unobstructed ideal}
Following ideas from \cite{fukaya2010lagrangian}, we will provide a method for constructing bounding cochains by inducting on the valuation. In order to do this, we need a slight refinement of a filtered $A_\infty$ algebra which states that the valuation of the structure coefficients is ordered by a monoid.
A gapped $A_\infty$ algebra is a filtered $A_\infty$ algebra for which there exists a finitely generated monoid $G$ and a monoid homomorphism $\omega: G\to \RR_{\geq 0}$ so that $\omega(\beta)=0$ implies that $\beta=0$, and so that we have the decomposition 
    \[m^k=\sum_{\beta\in G}T^{\omega(\beta)} m^{k, \beta}\]
where $m^{k, \beta}$ are graded with respect to the filtration. We say that it satisfies the gapped $A_\infty$ relations if for all $\beta\in G$
\[\sum_{\beta_1+\beta_2=\beta} \sum_{j_1+j+j_2=k} (-1)^\clubsuit m^{j_1+1+j_2, \beta_1}(\id^{\tensor j_1} \tensor m^{j, \beta_2} \tensor \id^{\tensor j_2})=0\]
Given $b=\sum_{\beta\in G\setminus \{0\}} b_\beta$, we can deform the product structure by 
    \[m^{k, \beta}_{(B, b)}=\sum_{\substack{ \beta_0+\cdots + \beta_k=\beta\\ \beta_i=\sum_{j=1}^{l_i} \beta_{i,j}}} m^{k+l}_B (\beta_{0,0}\tensor\cdots \tensor \beta_{0,l_0}\tensor \id \tensor \cdots \tensor \id \tensor \beta_{k,0}\tensor \cdots \tensor \beta_{k, l_k}).\]
    so that $m^k_{(B, b)}:= \sum_{\beta\in G}  T^{\omega(\beta)} m^{k, \beta}_B$ gives a $G$-gapped $A_\infty$ algebra satisfying the gapped $A_\infty$ relations.
There similarly exists $G$-gapped filtered $A_\infty$ homomorphisms, which also contain the data of a morphism of monoids $\phi: G_A\to G_B$.

We will also need some basic statements about ideals in filtered $A_\infty$ algebras.
\begin{definition}
    A subspace $A\subset B$ is a weak $A_\infty$ ideal if for all $k=k_1+1+k_2>0$, the map 
    \[m^k: B^{\tensor k_1}\tensor A \tensor B^{\tensor k_2}\to B\]
    has image contained in $A$.

    Notably, we \emph{do not} require that the curvature term $m^0_A$ be an element of $A$. As a result, it is not necessarily the case that $A$ is itself a filtered $A_\infty$ algebra. We say that $A$ is a \emph{strong $A_\infty$ ideal} if additionally $m^0_B\in A$.
\end{definition}
\begin{claim}
    Let $A\subset B$ be an $A_\infty$ ideal. The quotient $C=A/B$ inherits a filtered $A_\infty$ structure. $A$ is a strong $A_\infty$ ideal if and only if $C$ is tautologically unobstructed.
\end{claim}
\begin{proof}
    The filtered $A_\infty$ structure is the natural one,
    \[m^k_C([x_1]\tensor \cdots \tensor  [x_k]):=[m^k_B(x_1\tensor\cdots  \tensor x_k)]\]
    Because the $m^k_B$ are multilinear, we see that if $[x_i]=[x_i']$, that $m^k_C([x_1]\tensor \cdots\tensor  [x_i] \tensor\cdots \tensor  [x_k])=m^k_C([x_1]\tensor \cdots\tensor  [x_i'] \tensor\cdots \tensor [x_k])$.
    $A$ is a strong $A_\infty$ ideal if and only if $m^0_C=[m^0_B]=[0]$.
\end{proof}
\begin{example}
    Given a formal filtered $A_\infty$ morphism $f: B \to C$ (so that $f^k=0$ for all $k\neq 1$) the kernel of $ f$ is a weak $A_\infty$ ideal. 
\end{example}
\begin{example}
     Given a filtered $A_\infty$ algebra $A$, the set $A_{>0}$ of positively filtered elements is an example of a strong $A_\infty$ ideal.
    The quotient $\underline{A}:= A/A_{>0}$ is an example of a tautologically unobstructed $A_\infty$ algebra. 
    A relevant example comes from Lagrangian Floer cohomology, where $\underline{\CF(L)}=\CM(L)$.
\end{example}
\begin{claim}
    Suppose that $A\subset B$ is an $A_\infty$ ideal, and $d\in B$ is a deforming cochain. Then $A$ is an $A_\infty$ ideal of $(B, d)$. If $A\subset B$ is a strong $A_\infty$ ideal, and $m^0_d\in A$, then $A$ is a strong $A_\infty$ ideal of $B$. In particular, if $d\in A$ then $A$ is a strong $A_\infty$ ideal of $(B, d)$.
    \label{claim:idealafterdeformation}
\end{claim}
\begin{proof}
    Suppose that $a\in A$ is some element. Then 
    \begin{align*}
        m^k_{(B, d)}(x_1\tensor &\cdots \tensor a \tensor \cdots \tensor x_k)\\=&  \sum_{l=0}^\infty \sum_{j_0+\cdots + j_k=l} m^{k+l}_A(d^{\tensor j_1}\tensor\id\tensor d^{\tensor j_1}\tensor \cdots \tensor a \tensor \cdots \tensor \id \tensor d^{\tensor j_k})\in A.
    \end{align*}
    proving that $A$ is a $A_\infty$ ideal of $(B, d)$. 
\end{proof}
The vector space $H^1(A)$ is a lowest order approximation to the space of bounding cochains. When $\overline C$ is an anti-commutative differential graded algebra elements of $H^1(\overline C)$ are bounding cochains.
\begin{claim}
    Suppose that $C$ is tautologically unobstructed. Suppose that $f: C\to \overline  C$ is an $A_\infty$ map with gapped $A_\infty$ homotopy inverse $g: \overline C\to C$. Assume that $\overline C$ is an anti-commutative differential graded algebra. 
    Then for every class $[c]\in H^1(C)$ with $\val(c)>0$, there exists a bounding cochain $c'\in C$ and $\lambda>\val(c')$ with $[c']= [c]\in H^1(C/T^{\lambda}C)$.
    \label{claim:gradedanticommutativechoice}
\end{claim}
\begin{proof}
    Since $C, \overline C$ are gapped, we can select $\lambda>\val(c)$ so that $\omega(\beta)<\lambda$ implies $\omega(\beta)\leq \val(c)$.
    We observe that $f(c)\in  \overline C$ is closed, and therefore provides a bounding cochain for $\overline C$, as 
    \[m^0_{(\overline C, f(c))}= m^1_{\bar C}(f(c))+m^2_{\bar C}(f(c),f(c))=0.\] 
    We then take $c'$ to be the pushforward bounding cochain
    \begin{align*}
        g_*(f(c))=\sum_{k=1}^\infty g^k((f(c))^{\tensor k})
    \end{align*}
    Since $c'= (g\circ f )(c)\mod T^{\lambda}$, we obtain that $[c]=[c']\in H^1(C/T^{\lambda } C)$.
\end{proof}
\begin{claim}[Claim A.4.8 \cite{hicks2019wallcrossing}]
    Suppose that $A'=(A, a)$. Given a deforming cochain $a'\in A'$, the chain $a''=a+a'\in A$ is a deforming cochain so that  $(A', a')=(A, a'')$.
    \label{claim:deformationofdeformation}
\end{claim}

We now come to the main lemma of this appendix. Suppose that we have an exact sequence (on the chain level) $A\to B\to C$. If $A$ is a strong $A_\infty$ ideal containing the curvature of $B$, then we prove that there is no obstruction to finding a bounding cochain for $B$. The argument is in the style of \cite[Theorem 3.6.18]{fukaya2010lagrangian}.

\begin{lemma}
    Consider a $G$-gapped $A_\infty$ algebra $B$ satisfying the gapped $A_\infty$ relations. Suppose that:
    \begin{enumerate}[label=(\roman*)]
    \item $A$ is a strong $A_\infty$ ideal of $B$, and $C=B/A$ giving us an exact sequence $A\xrightarrow{i} B\xrightarrow{\pi}  C$  of gapped  $A_\infty$ algebras, \label{item:ses}
    \item There exists $\overline C$ which is $A_\infty$ homotopic to $C$ and is a anti-commutative DGA\label{item:anticommutes}
    \item Additionally, suppose that the connecting map $\delta: H^1(\underline C)\to H^2(\underline A)$ surjects. \label{item:surjects}
    \end{enumerate}
    Then for every $\lambda>0$ there exists a deforming cochain $b=\sum_{\beta\in G\setminus\{0\} }b_\beta$ for $B$ so that for all $\beta$ with $\omega(\beta)\leq \lambda$, $m^{0,\beta}_{(B, b)}=0$.
    \label{lemma:unobstructing}
\end{lemma}
\begin{proof}
    Because $A, B, C$ are gapped $A_\infty$ algebras, there exists $\{\lambda_i\}_{i=1}^n$ an ordering of the image $\omega(G)\in[0, \lambda]$.

    We prove the statement by induction on $\lambda_i$.
    Suppose that $A'=(A, a_{i-1}), B'=(B, b_{i-1}), C'=(C,c_{i-1})$ are $G$-gapped $A_\infty$ algebras satisfying \cref{item:ses,item:anticommutes,item:surjects} and additionally 
    \begin{enumerate}[label=(\roman*),start=4]
        \item  The curvature has large valuation, $\val(m^0_{B'})>\lambda_{i-1}$\label{item:deformedvaluation}
    \end{enumerate}
    The inductive step will construct deforming cochains $a', b', c'$ so that the algebras $(A', a'), (B', b'), (C', c')$ satisfy   \cref{item:ses,item:anticommutes,item:surjects,item:deformedvaluation} where $\lambda_{i-1}$ is replaced with $\lambda_{i}$.  By \cref{claim:deformationofdeformation}, we can then construct the $A_\infty$ algebras $(A, a_i), (B, b_i)$ and $(C, c_i)$.

    Write $m^0_{B'}=\sum_{j=i}^{\infty} \sum_{\omega(\beta)=\lambda_j}\underline{b}_{j, \beta} T^{\lambda_j}$, where the $\underline{b}_{j,\beta }$ are elements of $\underline B'=\underline{B}$ of degree 2.
    Because $A$ is a strong $A_\infty$ ideal, we can find  $\underline{a}_{i,\beta}\in \underline A$ with $i(\underline{a}_{i,\beta})=\underline{b}_{i,\beta}$. 

    We examine the lowest order terms of the $A_\infty$ relation $m^1_{A'}\circ m^0_{A'}=0$, and obtain
    \[
        \underline m^1_{A'}(\underline{a}_{i,\beta})=0
    \]
    Since $[\underline{a}_{i,\beta}]\in H^2(\underline A)$, by \cref{item:surjects} $[\underline{b}_{i,\beta}]=0$. Therefore there exists $\underline{\hat b}_{i, \beta}$ so that  $\underline m^1_{B'}(\underline {\hat b}_{i, \beta} )=\underline b_{i, \beta}$. 
    The class $\underline c_{i, \beta}:=\pi(\underline {\hat b}_{i, \beta})$ is closed.
    Using \cref{claim:gradedanticommutativechoice}, we can find $\underline{c}'_{j, \beta}$ with $j\geq i$ so that $c'=\sum_{j=i}^\infty \sum_{\beta \st \omega(\beta)=\lambda_i} \underline{c'}_{j, \beta} T^{\lambda_j}$ is a bounding cochain for $C'$ with the property that $[c'_{i, \beta}]= [c_{i, \beta}]\in H^1(C'/T^{\lambda_{i+1}}C').$
    
    Because $\pi : B\to C$ surjects, we can find for all $j\geq i$ cochains $\underline{b'}_{\beta,j}\in \underline B$ with $\pi(\underline{b'}_{j, \beta})=\underline{c''}_{j, \beta}$. 
    Let 
    \[b'=-\sum_{j=i}^\infty \sum_{\beta\st \omega(\beta)=\lambda_i} \underline{b}'_{i,\beta} T^{\lambda_i}.\]

    This constructed $b'$ satisfies the property
    \[m^1_{B'}(b')\equiv -m^0_{B'} \mod T^{\lambda_{i+1}}.\]
    Since $\pi$ is a filtered $A_\infty$ homomorphism without higher terms, the pushforward $\pi_*(b')=c'$ and 
    \[\pi\circ m^0_{(B,b')}=m^0_{(C,c')}=0\]
    Therefore $m^0_{(B', b')}$ is contained in $A'$, and we write $a'$ for the corresponding element in $A'$.
    \Cref{claim:idealafterdeformation} states that $(A', a')$ is a strong $A_\infty$ ideal of $(B', b')$, whose quotient is $(C', c')$.

    This gives us the $G$-gapped $A_\infty$ algebras $(A', b'), (B',b')$ and $(C', c')$, which we've shown satisfy \cref{item:ses}.    
    We now show these algebras satisfy \cref{item:anticommutes,item:surjects,item:deformedvaluation}.
    \Cref{item:anticommutes} follows from observing that deformations by Maurer-Cartan classes preserves having an anti-commutative model. Since the deformation occurs at valuation greater than 0, the map $H^1(\underline C)\to H^2(\underline A)$ continues to surject (\cref{item:surjects}).
     
    To check \cref{item:deformedvaluation},
    \begin{align*}
        \val(m^0_{(B', b')})=&\val\left( \sum_{k=0}^\infty m^k_{B'}((b')^{\tensor k}))\right)\\
        \geq&\min\left(\val(m^0_{B'}+m^1_{B'}(b')),  \sum_{k=2}^\infty m^k_{B'}((b')^{\tensor k})\right)
        \intertext{Given that $m^0_{B'}\equiv m^1_{B'}(b') \mod T^{\lambda_i}$}
        \geq&\lambda_{i+1}.
    \end{align*}
\end{proof}
\begin{corollary}
    Let $A, B, C$ be $A_\infty$ algebras as in \cref{lemma:unobstructing}. Then there exists a bounding cochain for $B$.
\end{corollary} 
\begin{proof}
    The deforming cochains constructed in the above proof satisfy the condition that 
    \[b_i\equiv b_{i+1} \mod T^{\lambda_i}.\]
    It follows that if we use the inductive procedure to build a sequence of deforming cochains $\{b_i\}_{i=0}^\infty$ so that $\val(m^0_{(B, b_i)})>\lambda_i$, the limit  $\lim_{i\to\infty} b_i$  is a bounding cochain.
\end{proof}

\subsection{\texorpdfstring{$A_\infty$}{A infinity}-bimodules and bounding cochains}
    Let $A, B$ be  $A_\infty$ algebras. An \emph{$(A, B)$- bimodule} is a filtered graded $\Lambda_{\geq 0}$-module $M$, along with a set of maps for all $k_1, k_2\geq 0$.
    \[
        m^{k_1|1|k_2}_{A|M|B}: A^{\tensor k_1}\tensor M \tensor B^{\tensor k_2}\to M
    \]
    satisfying filtered  quadratic $A_\infty$ module relations for each triple $(k_1|1|k_2)$
    \begin{align*}
        0=&\sum_{\substack{ j_1+j+j_2=k_1+1+k_2\\ j_1+j\leq k_1 }} m_{A|M|B}^{k_1-j+1|1|k_2}\circ ( \id_A^{\tensor j_1}\tensor m^{j}_A \tensor \id^{\tensor k_1-j_1-j}\tensor \id_M\tensor \id_B^{k_2})\\
        &+\sum_{\substack{ j_1+j+j_2=k_1+1+k_2\\j_1\leq k_1\leq j_1+j-1}} m_{A|M|B}^{j_1|1|j_2} \circ (\id_A^{\tensor j_1}\tensor m_{A|M|B}^{k_1-j_1|1|k_2-j_2}\tensor \id_B^{\tensor{j_2}})\\
        &+\sum_{\substack{ j_1+j+j_2=k_1+1+k_2\\k_1+1<j_1}} m_{A|M|B}^{k_1|1|k_2-j+1} \circ (\id_A^{\tensor k_1}\tensor \id_M\tensor \id_B^{k_2-j_2-j}\tensor m^j_B\tensor \id_B^{\tensor j_2}).
    \end{align*}
There is a $G$-gapped version of an $A_\infty$ bimodule, where we have the data of map of monoids $\omega: G_M\to \RR_{\geq 0}$ and our $A_\infty$ bimodule products can be decomposed as $m^{k_1|1|k_2, \beta}_{A|M|B}$; we also have morphisms $\phi_{A/B}: G_{A/B}\to G_M$ which intertwine with $\omega$.

If $M$ is a filtered $(A, B)$ bimodule, and $a\in A, b\in B$ are deforming cochains, then the filtered $A_\infty$ bimodule products on $M$ can be deformed to give it the structure of an $((A, a),(B, b))$-bimodule. 
As in the setting of $\Lambda$-filtered $A_\infty$ algebras, we can define $\Lambda$-filtered $A_\infty$ bimodules.
\begin{lemma}
    Let $M$ be a $G$-gapped $(A, B)$-bimodule. Suppose that $A, B$ are tautologically unobstructed and that $A$ has an anti-commutative DGA model $\overline A$ as in \cref{claim:gradedanticommutativechoice}.
    Suppose that there exists $\lambda_0<\lambda_1\in \RR$ with the following properties:
    \begin{enumerate}[label=(\roman*)]
        \item The maps $m^{k|1|0}_{A|M|B}:A^{\tensor k}\tensor M\to M$ all have image contained within $T^{\lambda_0}M$ and \label{item:initialvaluation}
        \item There exists $[\underline e]\in H^1(\underline M)$ an element so that the map \label{item:productsurjects}
    \begin{align*}
        H^1(A)\to& H^1(T^{\lambda_0}M/T^{\lambda_1}M)\\
        [a] \mapsto&  [m^{1|1|0}_{A|M|B}(a\tensor \underline e)]
    \end{align*}
    is surjective. 
\end{enumerate}
 Then there exists a choice of bounding cochain $a\in A^1$ and element $e\in M^0$ so that $m^{0|1|0}_{(A,a)|M|B}(e)=0$.
    \label{lem:submoduledeformation}
\end{lemma}
\begin{proof}
    We again use the gapped structure and induct on valuations. For simplicity of exposition, we will assume that the monoid $G$ is $\NN$, so that $\omega(G)=\{n \lambda\st n\in \NN \}\subset \RR$.
    We will construct a sequence of bounding cochains $a_i$ and elements $e_i\in M^0$ with the property that
    \begin{enumerate}[label=(\roman*)]
        \setcounter{enumi}{2}
        \item $a_i$ are bounding cochains; \label{item:closed}
        \item $m^{0|1|0}_{(A, a_i)|M|B}(e_i)\in T^{\lambda_0+i\lambda_1}M$; and\label{item:approachingzero}
        \item For $i>1$, $a_i-a_{i-1}\in T^{\lambda_0+ i\lambda_1}A$ and $e_i-e_{i-1}\in T^{\lambda_0+i\lambda_1}M$.\label{item:largevaluation}
    \end{enumerate}

    \emph{Base case:}
    Let $a_0=0$, and $e_0=e$. \Cref{item:initialvaluation,item:productsurjects} are given by the hypothesis. \Cref{item:closed} is trivial, and \cref{item:approachingzero} follows from the gapped structure. \Cref{item:largevaluation} has no content.
    
    \emph{Inductive Step:}
    Suppose we have constructed $a_{i}, e_i$ satisfying the induction hypothesis. 
    By \cref{item:approachingzero}, we can write  $m^{0|1|0}_{(A, a_i)|M|B}(e_i)\equiv  c_i  \mod T^{\lambda_0+(i+1)\lambda_1}$, where $c_i\in T^{\lambda_0+i\lambda_1}M$. 
    At  order $T^{\lambda_0+(i+1)\lambda_1}$,
    \[\underline m^{0|1|0}_{A|M|B}(c_i)\equiv  m^{0|1|0}_{(A, a_i)|M|B}(c_i)\equiv m^{0|1|0}_{(A, a_i)|M|B}\circ m^{0|1|0}_{(A, a_i)|M|B}(e_i)= 0 \mod T^{\lambda_0+(i+1)\lambda_1}\]
    We therefore obtain a class $[c_i]\in T^{\lambda_0+i\lambda_1}H^1(\underline M)$.
    Using \cref{item:productsurjects}, we have a homology class $\underline a\in T^{\lambda_0+i\lambda_1}A$ with 
    \[[m^{1|1|0}_{A|M|B}(\underline a\tensor \underline e)]\equiv [c_i] \mod T^{\lambda_0+(i+1)\lambda_1}.\]
    By \cref{claim:gradedanticommutativechoice}, there exists a bounding cochain  $a'\in T^{i\lambda_1}A^1$ for the product structures $m^k_{(A, a_i)}$ satisfying
    \begin{align*}
        a'\equiv& \underline a \mod T^{\lambda_0+(i+1)\lambda_1}\\
        [m^{1|1|0}_{A|M|B}(a'\tensor \underline e)]=& [c_i] \text{ in $H^1(T^{\lambda_0+i\lambda_1}M/T^{\lambda_0+(i+1)\lambda_1}M)$}.
    \end{align*}
    Write $m^{1|1|0}_{A|M|B}(a'\tensor \underline e)=c_i+ m^{0|1|0}_{A|M|B} e'$, where $e'\in T^{\lambda_0+i\lambda_1}$.
    Then let $e_{i+1}= e_i+e'$ and let $a_{i+1}=a_i-a'$. By construction, we satisfy \cref{item:largevaluation}.  By \cref{claim:deformationofdeformation}, $a_{i+1}$ is a bounding cochain for $A$, and we therefore obtain \cref{item:closed}. 
    Conditions \cref{item:initialvaluation,item:productsurjects}  are unchanged by deformations. It remains to prove \cref{item:approachingzero}:
    \begin{align*}
        m^{1|1|0}_{(A, a_{i+1})|M|B}(e_{i+1})\equiv &m^{0|1|0}_{(A, a_i)|M|B}(e_i)\\
        &-m^{1|1|0}_{A|M|B}(a',\underline e)+ m^{0|1|0}_{A|M|B}(e') - m^{1|1|0}_{A|M|B}(a',e') \mod T^{\lambda_0+(i+1)\lambda_1}\\
        \equiv & 0 \mod T^{\lambda_0+(i+1)\lambda_1}.
    \end{align*}
    To complete the proof of the lemma, we can take the bounding cochain $a$ and element $e$ to be 
    \begin{align*} a=\lim_{i\to\infty}a_i && e=\lim_{i\to\infty} e_i. \end{align*}
\end{proof}

\printbibliography
\end{document}